\documentclass[11pt]{article}
\usepackage{palatino}
\usepackage{aliascnt}
\usepackage{amsmath}
\usepackage{amsrefs}
\usepackage{amssymb}
\usepackage{amsthm}
\usepackage{algorithm}
\usepackage{algpseudocode}
\usepackage{tabularx}

\usepackage{bm}
\usepackage{bbm}

\usepackage{calrsfs}
\usepackage{color}

\usepackage[sans]{dsfont}

\usepackage{enumerate}
\usepackage{epsfig}
\usepackage{extarrows}
\usepackage[mathscr]{euscript}
\usepackage[symbol]{footmisc}
\usepackage{framed}
\usepackage{fullpage}
\usepackage[margin=1in]{geometry}
\usepackage{graphicx}
\usepackage{import}
\usepackage{multirow}

\usepackage{paralist}
\usepackage{psfrag}
\usepackage{amsmath}
\usepackage{subfig}
\usepackage[]{times}
\usepackage{transparent}
\usepackage[normalem]{ulem}
\usepackage{url}
\usepackage{verbatim}
\usepackage{wrapfig}
\usepackage{xcolor}
\usepackage{xifthen}
\definecolor{Green}{rgb}{0.0,0.45,0.0}
\usepackage[colorlinks,linkcolor=blue,citecolor
=red,bookmarksopen=true]{hyperref}

\def\NewTheorem#1#2{%
	\newaliascnt{#1}{thmm}
	\newtheorem{#1}[#1]{#2}
	\aliascntresetthe{#1}
	\expandafter\def\csname #1autorefname\endcsname{#2}
}

\numberwithin{equation}{section}

\NewTheorem{lem}{Lemma}
\NewTheorem{thm}{Theorem}
\NewTheorem{cor}{Corollary}
\NewTheorem{prop}{Proposition}
\theoremstyle{definition}
\NewTheorem{defin}{Definition}
\theoremstyle{remark}
\NewTheorem{remark}{Remark}
\theoremstyle{definition}
\NewTheorem{eg}{Example}
\NewTheorem{convention}{Convention}

\newcommand{\widebar}[1]{%
  \mkern1.5mu\overline{\mkern-1.5mu#1\mkern-1mu}\mkern1mu}

\renewcommand{\P}{ \mathbb P }
\newcommand{\EXP}{\mathbb{E}}
\newcommand{\IND}[1] {{ \mathbbm{1}_{ #1 }} }
\newcommand{\1}{\mathbbm{1}}

\newcommand{\ZZ}{U}
\newcommand{\R}{\mathbb{R}}

\newcommand{\F}{\mathcal{F}}
\renewcommand{\H}{\mathcal{H}}

\newcommand{\T}{\mathbb{T}}

\newcommand{\TT}{\mathscr{T}}

\newcommand{\Vo}{\mathring{V}}

\newcommand{\lwrX}{\underline{X}}

\newcommand{\uprX}{\widebar{X}}
\newcommand{\uprXu}{\widebar{X}_{u_0}}

\newcommand{\sX}{\mathscr{X}}
\newcommand{\tsX}{\tilde\sX}

\newcommand{\bY}{\mathbf{Y}}

\newcommand{\pT}{\partial\T}
\newcommand{\io}{\textup{ i.o. }}
\newcommand{\pV}{\partial V}
\renewcommand{\overleftarrow}[1]{%
  \overset{%
    \raisebox{-0.30ex}[0pt][0pt]{%
      $\scriptscriptstyle\leftarrow$%
    }%
  }{#1}%
}
\newcommand{\p}[1]{\overleftarrow{#1}}

\def\tS{\tilde{S}}
\def\oo{\overline{\omega}}
\def\ooone{\overline{\omega}^{(1)}}
\def\ootwo{\overline{\omega}^{(2)}}
\def\cE{E}
\def\tcE{\tilde{E}}
\def\ww{{\mathbf{w}}}
\def\uu{{\mathbf{u}}}
\def\w{{{v}}}
\renewcommand{\overleftarrow}[1]{%
  \overset{%
    \raisebox{-0.30ex}[0pt][0pt]{%
      $\scriptscriptstyle\leftarrow$%
    }%
  }{#1}%
}

\begin{document}
\date{\today}
\title{\Large Transformation of Stochastic Recursions and Critical Phenomena in the Analysis 
of the Athreya-Aldous-Shields Cascade and Related Mean Flow Equations}
\author{Radu Dascaliuc\thanks{Department of Mathematics,  Oregon State University, 
Corvallis, OR, 97331. {dascalir@math.oregonstate.edu}}
\and Tuan N.\ Pham\thanks{Faculty of Math and Computing,  Brigham Young University-Hawaii,  
Laie, HI 96762. 
{tpham@byuh.edu}}
\and Enrique Thomann\thanks{Department of Mathematics,  Oregon State University, 
Corvallis, OR, 97331.
{thomann@math.oregonstate.edu}}
\and
Edward C.\ Waymire\thanks{Department of Mathematics,  Oregon State University, Corvallis, OR, 97331. 
{waymire@math.oregonstate.edu}}
}
\maketitle

\begin{abstract} 
We study the discounted branching  random walk process introduced by Athreya (1985) and Aldous and Shields (1988), referred to as the Athreya-Aldous-Shields (AAS) cascade. This process can be viewed as an idealized version of the stochastic model introduced by Le Jan and Sznitman in 1997 for the Navier-Stokes equations. The AAS cascade exhibits phase transitions at the critical branching rates $\alpha=1$ and $\alpha=2$. These transitions are reflected in the behavior of the associated mean flow equation, a nonlocal ODE referred to as the $\alpha$-Riccati equation.  It is known that if $\alpha>1$, the AAS cascade is explosive in the sense of infinitely many splittings in finite time, and the $\alpha$-Riccati equation has at least two solutions. In this paper, we establish finer structural properties of an explosive AAS cascade and obtain a stronger non-uniqueness result. Namely, for each initial data in a certain range, the $\alpha$-Riccati equation has an infinite family of distinct solutions parametrized by their convergence rates at infinity. Our approach exploits a simple transformation between the AAS cascade and its "linearized" counterpart, whose mean flow equation is precisely the well-studied pantograph equation. This transformation demonstrates how a relationship at the level of stochastic processes can reveal an otherwise hidden relationship at the level of mean flow equations.
\end{abstract}

\tableofcontents

\section{Introduction}
\label{introsec}
For a parameter $\alpha >0$, the two innocent looking differential equations 
\begin{equation}
\label{alphap0}
w'(t)=-w(t)+2w(\alpha t),\ \ w(0)=w_0, 
\end{equation}
and
\begin{equation}
\label{diffalpha-ricc0}
u'(t)=-u(t)+u^2(\alpha t),\ \ u(0)=u_0
\end{equation}
are the source of a remarkable amount of
research in theoretical and applied mathematics. Originating in the engineering and design of electrically powered
cable cars, Equation \eqref{alphap0} and its parametric generalization
\begin{equation}\label{genpantode}
w^\prime(t) = -w(t) + aw(\alpha t), \quad w(0) = w_0
\end{equation}  
with $\alpha > 0$ and $a\in\mathbb{R}\backslash\{0\}$, known as a \emph{pantograph equation},
enjoy diverse applications in statistical physics, mathematical analysis, number theory, graph theory, and combinatorics (see \cite{shapira2021quasirandom} 
and the references therein). The nonlinear equation \eqref{diffalpha-ricc0} is referred to as the \emph{$\alpha$-Riccati equation}
 in deference to the classic Riccati equation $\alpha=1$ (see \cite{alphariccati}). Equivalently, one may write \eqref{genpantode} and \eqref{diffalpha-ricc0} in integral forms as follows.
\begin{align}\label{mildpant}
 w(t) = e^{-t}w_0 + a \int_0^te^{-s}w(\alpha(t-s))ds, \quad  t\ge 0,\\
\label{mildalpha}
u(t) = e^{-t}u_0+ \int_0^t e^{-s}u^2(\alpha(t-s))ds\,,
\quad  t\ge 0.
\end{align} 
The well-posedness problems of \eqref{mildpant} and \eqref{mildalpha} are nontrivial due to their nonlocal nature, especially in the case $\alpha>1$ where the integrals involve the values of the solution from past and future times. 
Kato and McLeod \cite{kato1971functional} provides an extensive analytical treatment of \eqref{genpantode}\footnote{The coefficient $-1$ of $w(t)$ in \eqref{genpantode} can be replaced by $b<0$ by scaling time.}, 
including a proof of non-uniqueness of solutions for parameter values $\alpha>1$.  
An analysis of the well-posedness of \eqref{mildalpha} was
provided by Athreya \cite{athreya} using a {\em discounted branching random walk}, including a proof
 of non-uniqueness for $\alpha >1$. In \cite{Aldous1998Diffusion}, Aldous and Shields studied randomly-growing binary trees for which \eqref{mildalpha}
 may be viewed as a {\em mean flow equation} governing a
 particular stochastic 
 flow of initial data, with $\alpha>0$ representing the branching rate, and derived scaling laws for the case $\alpha<1$.  We will refer to this stochastic branching model as the {\it Athreya-Aldous-Shields (AAS) cascade}, whose construction can be summarized as follows. 
 
Consider a branching process in which a particle indexed by $\theta$ survives for an exponentially distributed
time $Y_\theta$ with mean $\alpha^0=1$ before splitting into
a pair of offspring indexed by $1$ and $2$. Each offspring then evolves independently in the same manner: a particle indexed by $j\in\{1,2\}$ survives for an exponentially distributed time $Y_{j}$ with mean $\alpha^{-1}$ before splitting into a pair of offspring indexed by $(j1)$ and $(j2)$, and the process continues indefinitely.  Thus, the holding time of the particle indexed by $v\in\cup_{n=0}^\infty\{1,2\}^n$ after $k$
 generations is exponentially distributed with mean $\alpha^{-k}$.  If one denotes the \emph{explosion time} of this branching process (AAS cascade) by  
 \begin{equation}\label{shortdef}
 S=\inf_{s\in\{1,2\}^\infty}\sum_{j=0}^\infty Y_{s|j},\end{equation}
where $s|j$ is the truncation of $s$ at length $j$ with the convention that $s|0=\theta$, then one obtains
 \begin{equation}\label{alphRrecur}
 S \xlongequal{\text{dist}} Y_\theta + \alpha^{-1}\min\{S^{(1)},S^{(2)}\},
 \end{equation}
where $S^{(1)}$ and $S^{(2)}$ are independent copies $S$. In writing $X(t) =  \IND{[S>t]}$, one has the stochastic recursion
 \begin{equation}\label{INDrecur}
 X(t) = \begin{cases} u_0=1 \ & \ \text{if} \ Y_\theta\ge t\\
 X^{(1)}(\alpha(t-Y_\theta))X^{(2)}(\alpha(t-Y_\theta)) \ &\  \text{if}\ Y_\theta<t,
\end{cases}
 \end{equation}
 where $X^{(1)}$ and $X^{(2)}$ are two copies of $X$
 applied to the independent left and right subtrees re-rooted
 at the respective offspring.
Moreover, $X(t)$  is integrable and
$u(t) = \mathbb{E}X(t)$ solves
\eqref{mildalpha} with $u_0=1$. The stochastic process $X(t) = \IND{[S>t]}$ illustrates
the notion of an stochastic flow of initial data, or
a
  {\it (stochastic) solution process}, for the mean flow equation \eqref{mildalpha}, in fact an {\it indicator} hence {\it integrable}, solution process
  in this case. Thus, two 
flows of general initial data $u_0$ are considered in this paper: (i) the mean flow defined as a solution to \eqref{mildalpha}, and (ii) the stochastic flow defined by \eqref{INDrecur}.

The present authors arrived at the $\alpha$-Riccati equation and its associated AAS cascade as the result of a self-similarity heuristic (see \cite{alphariccati}*{Sec.\ 1}) which incorporates the stochastic framework introduced by 
Le Jan and Sznitman \cite{lejan} for the 
deterministic Navier-Stokes equations (NSE) with the natural scaling and rotational symmetry of the equations, 
but without their critical thinning device. In this way, \eqref{mildalpha} can be regarded as a toy model obtained through a mean field approximation to NSE; see also \cite{dascaliuc2025navier}. It serves as a macroscopic, averaged description of a hypothetical physical system, whereas the stochastic equation \eqref{alphRrecur}
describes the same system at a microscopic, stochastic level.  Examples of this way of thinking include the relationship between the heat equation and the Brownian motion, and that between the Fisher-KPP equation and the branching Brownian process
\cites{mckean1975application, jorge}. The connection between shock formation in Burgers equation and critical behavior of spontaneous magnetization is another example \cites{newmanburgers, choquard2004}. Similar in spirit is the analysis of critical phenomena associated with the Ising magnetization via a diagrammatic statistical representation \cite[Remarks 1.2]{aizenman1986}.

The richness of the AAS cascade is recognizable in contemporary probability theory and its applications. For example,
a percolation model related to the AAS cascade is  noted in \cite{Aldous1998Diffusion}.  Connections with the 
 {\em smoothing transformations} of \cites{holleyliggett, durrettliggett} are heavily
 exploited in \cite{yule}.
 Applications in biology and physics are given in \cites{BestPfaffelhuber,dean2005phase}. Recent works have emerged in the context of a general class of discounted branching random walks with an emphasis on  the so-called boundedness problem \cites{zhangbrw, zhangpadic, aidekonbrw,robertstree,chauvintree}, where boundedness coincides with our notion of stochastic explosion, i.e.\ $\P(S=\infty)<1$.

In the present paper, we construct non-indicator stochastic solution processes associated with \eqref{mildalpha} as limits of stochastic Picard iterations initialized by suitably chosen \emph{ground states}. The non-indicator property is essential since indicator random variables are otherwise left invariant under key transformations. This technique has been employed in our previous works \cite{dascaliuc2019jan, alphariccati}. Non-indicator solution processes associated with \eqref{mildpant} are constructed in the same manner. There is a rather obvious family of transformations mapping a solution process associated with \eqref{mildpant} to a solution process  associated with \eqref{mildalpha} and defined on the same probability space. Our main goal is to show that these transformations lead to a \emph{one to many principle}. Namely, a single nontrivial solution process associated with the pantograph equation with zero initial data can generate a one-parameter family of distinct solutions to the $\alpha$-Riccati equation parametrized by their convergence rates at infinity (\autoref{main}). Such transformations are far from evident at the level of the mean flow equations, demonstrating how a relationship at the level of stochastic processes can reveal an otherwise hidden relationship at the level of mean flow equations. Our construction of the nontrivial solution process mentioned earlier varies depending on the relation between
\begin{equation}
\label{gma}
\gamma=\log_\alpha 2=\frac{\ln 2}{\ln\alpha}
\end{equation}
and 1, thus confirming the criticality of $\alpha=2$. Let 
\begin{equation}\label{Phifunc}
\Phi(x)=\frac{4x^3-14x^2+20x-5}{4(x-1)(2x-1)}.
\end{equation}
\begin{thm}\label{main}
Let $\alpha>1$ and define an interval $D_\alpha\subset \mathbb{R}$ by

\begin{equation}\label{R_alpha}
D_\alpha= \begin{cases} [0,1] \ & \ \text{if} \ \alpha\in(1,\alpha_0]\\
 \left[0,  \Phi(\alpha)\right) \ &\  \text{if}\ \alpha>\alpha_0
\end{cases}
\end{equation}
where $\alpha_0\approx 3.2540$ is the largest real root to the equation $\Phi(x)=1$.
Then for each $u_0\in D_\alpha$, there is a family
of integrable stochastic solution processes $X_\lambda(t)$,
$\lambda>0$ for $u_0$ 
such that the
 family of mean flow solutions 
 $u_\lambda(t)=\mathbb{E}X_\lambda(t), \lambda >0,$ 
 of \eqref{diffalpha-ricc0}
 have the following properties.
\begin{enumerate}[(i)]
\item $0<u_\lambda(t)<u_{\lambda'}(t)$ for all $t>0$ and $0<\lambda'<\lambda<\infty$.
\item $u_\lambda$ has an algebraic convergence rate at infinity:
\begin{equation}\label{u_lambda_ass}
\lim_{t\to\infty}\frac{1-u_\lambda(t)}{t^{-\gamma}}={\lambda}.
\end{equation}
\end{enumerate}
\end{thm}
Consequently, there are infinitely many solutions converging to 1 with an algebraic rate 
$t^{-\gamma}$ as $t\to\infty$. This non-uniqueness result is stronger than \cite[Thm 4.2]{alphariccati},  where it was proven that for each initial data $u_0\in D_\alpha$, the problem \eqref{diffalpha-ricc0} has two distinct integrable
solution processes 
\begin{equation}
\label{minmaxsp}
\underline{X}(t)=u_0^{N(t)}\IND{[S>t]},~~\textrm{and}~~~
\overline{X}(t)=u_0^{N(t)},
\end{equation}
where  $N(t)\in\{0\}\cup\mathbb{N}\cup\{\infty\}$ 
denotes the number of surviving particles at time $t$. We denote the mean values of these processes by
\begin{equation}
\label{minmaxsolutions}
\underline{u}(t) = \mathbb{E}(\underline{X}(t)),~~~\textrm{and}~~~\overline{u}(t) = \mathbb{E}(\overline{X}(t)).
\end{equation}
The proof in \cite{alphariccati} relies on the method of stochastic 
Picard iterations from a unit ground state; also see Proposition \ref{alphgrndstateprop} 1(ii),2(ii) of the present paper.
The parameter $\lambda$ is a 
free parameter in the stochastic
transformations of solution processes defined by stochastic
Picard iterations from other ground states designed for
other mean flows for given $u_0\in D_\alpha$.
Since the one-parameter family of solutions $\{u_\lambda\}_{\lambda>0}$ in \autoref{main} is decreasing with respect to $\lambda$, it is natural to ask about  the behaviors of the limiting profiles 
\begin{equation}\label{ulims}
u_\infty(t)=\lim_{\lambda\to \infty}u_\lambda(t),\qquad u_*(t)=\lim_{\lambda\to 0}u_\lambda(t).
\end{equation} 
A detailed analysis of the event $[S\le t]$ of explosion by time $t$ is needed to characterize  $u_\infty$ and $u_*$. 
 To this end, detailed structural characterizations of the AAS cascade on this event (\autoref{second}) are
 developed. Namely,
\begin{thm}\label{second}
Let $\alpha>1$ and $t>0$. Then
\begin{enumerate}[(i)]
\item Almost surely, on the event $[S\le t]$, there is a particle whose descendants all die out before time $t$.
\item Almost surely, on the event $[S\le t]$, there are finitely many surviving particles at time $t$ if and only if there are finitely many genealogically oldest\footnote{i.e.\ maximal in the partial order where $v\prec w$ if and only if particle $v$ is a direct descendant of particle $w$} particles whose descendants all die out before time $t$.
\item Almost surely, on the event $[S\le t]$, there are infinitely many surviving particles at time $t$ if and only if there is a lineage $s\in\{1,2\}^{\infty}$ whose lifetime $\sum_{j=0}^\infty Y_{s|j}$ is exactly $t$.
\end{enumerate}
\end{thm}
Our third goal is to obtain the following relationship between $(u_\infty,u_*)$ and $(\underline{u},\widebar{u})$, the proof of which uses the characterizations of the explosion event $[S\le t]$ mentioned in \autoref{second}, as
well as the phenomena $[N(t)=\infty]$, referred to as
$t$-leaf percolation. In \cite{dascaliuc2023errata} it was shown that this
event can occur with positive probability if and only 
if $1<\alpha<2$, and in such cases, 
 $\underline{\underline X}(t) =
{ \mathbbm{1}_{[N(t)<\infty]}}$ provides another
(distinct) integrable solution process
for $u_0=1$. 
\begin{thm}\label{X>0}  
Let $1<\alpha<2$. There is a family of non-indicator integrable
stochastic solution processes $X_{M,\delta}(t)$ for $u_0=1$, parameterized
by sufficiently large $M>0$ and small $\delta>0$,
such that 
\[[N(t) =\infty] = [X_{M,\delta}(t)=0].\]
\end{thm}
The family of solution processes $X_{M,\delta}(t)$ will be 
constructed from  a stochastic Picard iteration 
of the ground state $\rho_{M,\delta}(t) =1-\delta t^{-\gamma}
\IND{[M,\infty)}(t)$.
One may notice that the event that $t$ belongs to the
zero set of $X_{M,\delta}$ does not depend on the
underlying parameters $M,\delta$ from this class.

The main results stated above will be applied to prove
the following relationships between
the functions defined by \eqref{minmaxsolutions} and \eqref{ulims}.
\begin{thm}\label{main3}
Let $\alpha>1$, $u_0\in D_\alpha$, and $\overline{X}_{u_0}(t)$ be given by \eqref{minmaxsp}. Then 
\begin{enumerate}[(i)] 
\item There exists an integrable solution process $\ZZ(t)$ to \eqref{INDrecur} such that the functions
$$u_*(t)=\mathbb{E}[\overline{X}_{u_0}(t) \IND{[\ZZ(t)>0]}]\quad\text{and}\quad u_\infty(t)=\mathbb{E}[\overline{X}_{u_0}(t) \IND{[\ZZ(t)=1]}]$$ 
are solutions to \eqref{diffalpha-ricc0}.
\item $u_\infty=\underline{u}$\,.
\item $u_*=\widebar{u}$ except when $(u_0,\alpha)\in\{1\}\times(1,2)$, 
in which case 
$$u_*(t)=\P(N(t)<\infty)<\widebar{u}(t)=1\qquad\forall t>0.$$
\item The limits $\lim_{t\to\infty} e^{t}\,u_\infty(t)$ and $\lim_{t\to\infty}e^t(1-u_*(t))$ exist and are finite. The former is zero if and only if $(u_0,\alpha)\in\{0\}\times (1,\infty)$, in which case $u_\infty\equiv 0$. The latter is zero if and only if $(u_0,\alpha)\in\{1\}\times [2,\infty)$, in which case $u_*\equiv 1$.
\item Let us make explicit the dependence of solutions on the initial data by writing $u_\lambda(u_0;t)$. Then the solution map $h:D_\alpha\times[0,\infty]\times[0,\infty)\to \mathbb{R}$ defined below
is continuous\footnote{The topology on $[0,\infty]$ is understood as the one-point compactification topology of $[0,\infty)$.}
\begin{equation}\label{homotopy}
h(b,\lambda,t)= \begin{cases} u_\lambda(b;t) \ & \ \text{if} \ \lambda\in(0,\infty],\ t\in[0,\infty),\\
 u_*(b;t) \ &\  \text{if}\ \lambda=0,\ t\in[0,\infty).
\end{cases}
\end{equation}
\end{enumerate}
\end{thm}
Therefore, except for the case $u_0=1$ and $\alpha\in(1,2)$, the family of solutions $\{u_*\}\cup\{u_\lambda\}_{\lambda>0}\cup\{u_\infty\}$ is a homotopy between the two solutions $\underline{u}$ and $\widebar{u}$ found in \cites{alphariccati}. While the endpoint maps have an exponential rate of convergence as $t\to\infty$, the intermediate maps only have an algebraic rate of convergence.

The paper is organized as follows. \autoref{ssprobintrosec} introduces two stochastic models. The first is based on a unary tree 
 for the pantograph equation. The second is based on a binary branching tree for the $\alpha$-Riccati equation. We demonstrate how \eqref{genpantode} and \eqref{diffalpha-ricc0} can be viewed as self-similar versions of partial differential equations. Under restrictive conditions, it can be shown that  \eqref{genpantode} can be solved using Feller's semigroup theory for jump Markov processes. This approach and its limitation are presented in \autoref{fellerapp}.
The notion of \emph{hyperexplosive time} $L$ is introduced, and the asymptotic properties of the distribution of  $L$ and $S$ are established in \autoref{expl_asymptotics}.  A technical aspect in the proof of this theorem is postponed to \autoref{appendix1} for simplicity of presentation.

In \autoref{Sec3}, we describe a family of transformations that maps a single solution process for the pantograph equation to a one-parameter family of solution processes for the $\alpha$-Riccati equation. We construct a  nontrivial solution to \eqref{genpantode} with $w_0=0$ (\autoref{prop11112}). Such a solution has a slow algebraic decay as $t\to\infty$, consistent with  the
Kato-McLeod asymptotic characterization for solutions to \eqref{genpantode} in the case $\alpha>1$ and $a>1$ \cite[Thm 9]{kato1971functional}. 

\autoref{PicardIteration} gives details to the stochastic Picard iteration method. Similar to its role in the theory of stochastic differential
equations, this method is used to construct stochastic solution processes initialized by various choices of ground states.

\autoref{Sec6} is dedicated to the 
proof of \autoref{main}. We use two different nontrivial solution processes in two different cases $1<\alpha\le 2$ and $\alpha>2$. The critical role of $\alpha=2$ is exhibited in the different structure of the hyperexplosive event. Further characterizations of the limiting profile $u_*$ are given in \autoref{Sec5}. The proof of the latter is deferred to \autoref{Sec5}, as it requires further analysis of an explosive tree. Since the present
focus is on mean flow equations possessing 
certain types of quadratic nonlinearities
the stochastic analysis is restricted to binary trees. However
one may reasonably expect extensions to $b$-ary and 
Galton-Watson trees to be possible that may be of 
general interest to the theories related to the boundedness problem
for discounted branching random walks cited earlier.

\autoref{Sec5} is dedicated to the proofs of \autoref{second}, \autoref{X>0} and \autoref{main3}. The phenomenon of $t$-leaf percolation is introduced, together with its implication for the non-uniqueness of solutions to \eqref{mildalpha} {and Theorem \ref{X>0}.}

\autoref{Sec7} investigates non-uniqueness results obtained from ground states different from those used in \autoref{Sec6}. In \autoref{lastcase}, we use the Fr\'{e}chet extreme value distribution as a ground state to establish \autoref{main} for the case $u_0=0$. This function was used in \cite{athreya}, where three distinct solutions were constructed via a different method. Thus, our method extends Athreya's non-uniqueness result. In \autoref{groundstatesec}, we give a classification of non-uniqueness results via stochastic Picard iterations with constant ground states.

 \autoref{simulation} describes Monte Carlo algorithms for the pantograph equation and $\alpha$-Riccati equation to numerically illustrate the analytical results in \autoref{Sec6}. 

\section{Self-similarity, solution processes, and probability on trees}
\label{ssprobintrosec}
We start with a more general 
setting that places the pantograph equation and the $\alpha$-Riccati equation within the framework of the 
{\em self-similar} Doubly Stochastic Yule Cascades (DSY) introduced in \cite{part2_2021}. Consider space-time counterparts of \eqref{genpantode} 
 and \eqref{diffalpha-ricc0} as follows.
\begin{equation}\label{spacetimepant}
 \frac{\partial \ww}{\partial t}(\ell,t) = -\ell \ww(\ell,t) + a \ell \ww(\alpha \ell,t), \quad \ell>0,  
 t\ge 0, \quad \ww(\ell,0)=\ww_0(\ell)
 \end{equation} 
 and 
 \begin{equation}\label{spacetimealpha}
 \frac{\partial \uu}{\partial t}(\ell,t) = -\ell \uu(\ell,t) + \ell \uu^2(\alpha \ell,t), \quad \ell>0,  
 t\ge 0, \quad \uu(\ell,0)=\uu_0(\ell).
 \end{equation} 
It is useful to think of \eqref{spacetimepant} and \eqref{spacetimealpha} 
 as nonlocal differential equations in Fourier space with $\ell=|\xi|$. Equivalently, one can write \eqref{spacetimepant} and 
\eqref{spacetimealpha} in an integral form as \begin{equation}\label{mildspacetimepant}
 \ww(\ell,t) = e^{-\ell t}\ww_0(\ell) + a \int_0^t\ell e^{-\ell s}\ww(\alpha\ell,t-s)ds, \quad \ell>0,  t\ge 0
 \end{equation} 
 and
 \begin{equation}\label{mildspacetimericcati}
\uu(\ell,t) = e^{-\ell t}\uu_0(\ell)+ \int_0^t\ell e^{-\ell s}\uu^2(\alpha\ell,t-s)ds,
\quad \ell>0,  t\ge 0.
\end{equation} 
The problems \eqref{spacetimepant} and \eqref{spacetimealpha} have a natural scaling symmetry: If $\ww(\ell,t)$ and $\uu(\ell,t)$ are solutions, then  $\ww_\lambda(\ell,t)=\ww(\ell/\lambda,\lambda t)$ and $\uu_\lambda(\ell,t)=\uu(\ell/\lambda,\lambda t)$ are also solutions to corresponding 
 equations with initial data $\ww_\lambda(\ell,0)=\ww_0(\ell/\lambda)$ and $\uu_\lambda(\ell,0)=\uu_0(\ell/\lambda)$, respectively. This allows one to define {\em self-similar solutions} as solutions of the form $\ww(\ell,t) = w(\ell t)$ and $\uu(\ell,t) = u(\ell t)$. The \emph{self-similarity parameter} is $s=t$ by taking
 $\ell=1$.

One can easily check that self-similar solutions to \eqref{mildspacetimepant} and \eqref{mildspacetimericcati} 
satisfy the integral form \eqref{mildpant} and \eqref{mildalpha}, 
respectively. 
The integral form of these equations readily suggest them
to be the average behavior of the stochastic models as constructed
in the following subsections.

  \subsection{A probabilistic framework for pantograph
 equations}
\label{probpanto}

Let $\{T_j:j\ge 0\}$ be a sequence of  
exponentially distributed random variables having
unit intensity and defined on a probability space
$(\Omega,\mathcal{F},\P)$, and let $a\in\R,\alpha >0$
denote fixed parameters. For given initial data $w_0$,
 define a {\it unary stochastic solution process} for (\ref{mildpant}) 
as a progressively measurable stochastic process $\tsX$ satisfying the a.s. recursion
 \begin{equation}\label{pantrecursion}
\tsX(t) = \begin{cases} w_0 \ & \ \text{if} \ T_0> t\\
                             a\tsX^{(1)}(\alpha(t-T_0)) \ &\  \text{if}\ T_0\le t
                             \end{cases}
 \end{equation}
where 
$\tsX^{(1)}$ has the same distribution as $\tsX$ and is independent of $T_0$.  
 Iterating the recursion \eqref{pantrecursion},  one obtains
\begin{align}
\nonumber\tsX(t)&=a\tsX^{(1)}(\alpha(t-t_0))=a^2\tsX^{(2)}\left(\alpha^2\left(t-T_0-\frac{T_1}{\alpha}\right)\right)\\
&=...=a^k\tsX^{(k)}\left(\alpha^k\left(t-T_0-\frac{T_1}{\alpha}-...-\frac{T_{k-1}}{\alpha^{k-1}}\right)\right)=a^{\tilde{N}(t)}w_0
\end{align}
where $\tilde{N}(t):=\inf\left\{n\ge 0:\sum_{j=0}^n \tfrac{T_j}{\alpha^j}>t\right\}$ 
is the number of clock rings by time $t$, provided that it is finite.
Let 
 \begin{equation}\label{tilde_S}
\tilde{S}=\sum\limits_{j=0}^{\infty}\frac{T_j}{\alpha^j}
 \end{equation}
be the explosion time. By conditioning on $T_0$, one can check that $w(t) = \EXP[\tsX(t)\IND{[\tilde{S}>t]}]$ 
 satisfies \eqref{mildpant} and thus is a self-similar solution to (\ref{spacetimepant}), 
 provided that the expectation exists.  Therefore, a unary stochastic solution process is naturally associated with a {\it unary tree} rooted at vertex $\theta$ consisting of the waiting times $\left\{\frac{T_j}{\alpha^j}:\ j= 0,1,2,...\right\}$. See \autoref{unary-tree} for a realization of the unary tree.
\begin{figure}[h!]
\centering
\includegraphics[scale=1]{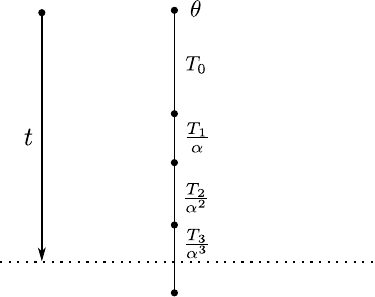}
\caption{Pantograph unary tree}
\label{unary-tree}
\end{figure}
By the convolution formula for non-identical 
 exponential distributions \cite[p.\ 40]{Feller2},
\begin{eqnarray}\label{distinctconv}
 \P(\tilde{N}(t) = n) &=&
 \P\left(\sum_{j=0}^{n}
 \frac{1}{\alpha^j}T_j>t\right)-
 \P\left(\sum_{j=0}^{n-1}
 \frac{1}{\alpha^j}T_j>t\right)\nonumber\\
&=& \sum_{j=0}^n
\frac{e^{-\alpha^j t}}{\prod_{k\in I_{n,j}}
(1-\alpha^{j-k})} -  \sum_{j=0}^{n-1}
\frac{e^{-\alpha^j t}}{
\prod_{k\in I_{n-1,j}}(1-\alpha^{j-k})}\nonumber\\
&=& \sum_{j=0}^n
\frac{e^{-\alpha^j t}\alpha^{j-n}}{\prod_{k\in I_{n,j}}(1-\alpha^{j-k})}
 \end{eqnarray}
where $I_{n,j}=\{0,1,...,n\}\backslash\{j\}$. Therefore,
\begin{equation}\label{treepant}
 w(t) = {\EXP}\left[w_0\, a^{\tilde{N}(t)}\IND{[\tilde{S}>t]} \right] =w_0\sum_{n=0}^\infty a^n\P(\tilde{N}(t)=n)= w_0
 \sum_{n=0}^\infty a^n\sum_{j=0}^n
 \frac{e^{-\alpha^jt}\alpha^{j-n}}{\prod_{k\in I_{n,j}}(1-\alpha^{j-k})}.
 \end{equation}
If 
\begin{equation}
\alpha>\max\{|a|,1\},
\end{equation}
one can exchange the order of summation to rearrange the series into
 \begin{equation}
 \label{KatoMcLeod}
 w(t) = w_0\, c_{a,\alpha} \sum_{n=0}^\infty \frac{a^n e^{-\alpha^n t}}{\prod_{j=1}^n (1-\alpha^j)}  \quad\text{where}\quad c_{a,\alpha} = \sum_{n=0}^\infty \frac{a^n}{\alpha^n\prod_{j=1}^n (1-\alpha^{-j})}
 \end{equation}
with the convention that $\prod_{j=1}^0 = 1$. The formula \eqref{KatoMcLeod} is also obtained in 
 \cite[Eq.\ (5.2)]{kato1971functional} by an analytic method. Under the additional restriction $0<a\le 1$, \eqref{KatoMcLeod} can be derived from Feller's semigroup theory (\autoref{fellerapp}).
 
For $w_0=a=1$, the solution $w$ reduces to the complementary distribution function $G(t) = 
 \P(\tilde{S}>t)$. Hence,
 \begin{equation}
 \label{complementary}
 G(t) = C_\alpha \sum_{n=0}^\infty \frac{e^{-\alpha^n t}}{ \prod_{j=1}^n (1-\alpha^j)}\quad\text{where}\quad
 C_\alpha=\sum_{n=0}^\infty \frac{1}{\alpha^n \prod_{j=1}^n (1-\alpha^{-j})}.
 \end{equation}

\begin{remark}\label{acasermk2}
 For $0<a<1$, the stochastic recursion \eqref{pantrecursion} can be
 ``thinned'' as follows to circumvent the explosion issue when $\alpha>1$:
 \begin{equation}\label{hatX}
\hat{\sX}(t) = \begin{cases} w_0 \ & \ \text{if} \ T_0> t\\
0 \ &\ \text{if}\ K=0, T_0\le t\\
\hat{\sX}^{(1)}(\alpha(t-T_0)) \ &\  \text{if}\ K =1, T_0\le t,
\end{cases}
 \end{equation}
 where $K\in\{0,1\}$ is a Bernoulli random
 variable independent of the holding time $T_0$ with $\mathbb{P}(K=1) =a$. The solution $w(t)=\EXP\hat{\sX}(t)$ has the same series representation as \eqref{treepant}. The term $a^n$ now represents the probability of 
 $n$ clock rings prior to the absorption at $0$.  
 \end{remark}
   
 \subsection{A probabilistic framework for \texorpdfstring{$\alpha$}{}-Riccati 
 equations}\label{binary_setup}
In view of the quadratic non-linearity,  the stochastic model associated with the $\alpha$-Riccati equation will require
 some basic notations on rooted binary tree graphs. 
 A tree rooted at  $\theta$
is a graph without loops and with a vertex $\theta$ designated as root. We identify
such a tree with its set of vertices. Let $\T=\{\theta\}\bigcup\left(\cup_{n=0}^\infty\{1,2\}^n\right)$ be a binary tree rooted at $\theta$. 
For a vertex $v=(v_1,\dots,v_n)\in\T$, denote $|v| = n, v|j = (v_1,\dots,v_j),
1\le j\le n$, and $v|0=\theta$. Also, denote by $\overleftarrow{v}=v|(|v|-1)$ the parent of a vertex 
$v\in \T\setminus\{\theta\}$. We denote $\partial \T=\{1,2\}^{\mathbb{N}}$ and call 
each $s\in\partial\T$ {a ray} of the tree $\T$. A vertex $v\in\T$ is said to belong to a ray 
$s\in\partial\T$ (written as $v\in s$) if $v=s\big||v|$, i.e.\ if $s=(v_1,\dots,v_n,\dots)$. Consequently, the root $\theta\in s$ 
for any ray $s\in\partial\T$.

Let $\{T_v:v\in\T\}$ be a tree-indexed
random field of  
exponentially distributed random variables having
unit intensity and defined on a probability space
$(\Omega,\mathcal{F},\P)$, and let $\alpha >0$
denote a fixed parameter. For given initial data $u_0$,
 a {\it stochastic solution process} for the mild form \eqref{mildalpha} of the $\alpha$-Riccati equation is defined by the stochastic recursion 
 \begin{equation}\label{alphrecursion}
X(t) = \begin{cases} u_0 \ & \ \text{if} \ T_\theta\ge t\\
                           X^{(1)}(\alpha(t-T_\theta))\,X^{(2)}(\alpha (t-T_\theta)) \ &\  \text{if}\ T_\theta< t,
                             \end{cases}
 \end{equation}
 where 
 $X^{(1)},X^{(2)}$ are copies of $X$ respectively defined on the independent left and right subfamilies 
 $\{T_{(j)}\}\cup\{T_v:v|1=j,v\in\mathbb{T}\}$ rooted at $(j),j=1,2$. This 
 recopying of $X$ is repeated throughout the recursion, with $1$ signifying
 left, and $2$ signifying right.
 We say that a solution process $X$ is integrable if $X(t)\in L^1(\Omega,\mathcal{F},\mathbb{P})$ 
 for each $t>0$.  
 A more precise and explicit formulation of {\it stochastic solution
 process} involves the canonical product space model utilized in 
Theorem \ref{X_n-conv} and for which we make the following convention throughout.

\begin{convention}\label{convention}
The canonical 
 product probability space is defined by
$(\Omega=[0,\infty)^{\T}, \mathcal{F}=\mathcal{B}_{\Omega}, \P)$  for countably many 
mean-one exponentially distributed random variables
indexed by $\T$. That is, 
$d\P=\prod_{v\in \T} d\P_v$, with $d\P_v(\omega_v)=e^{-\omega_v}d\omega_v$, with
 $T_v(\oo)=\omega_v$, $\oo =(\omega_v)_{v\in \T}\in \Omega$,
 and $\mathcal{B}_\Omega=\otimes\mathcal{B}_{[0,\infty)}$ is the 
 product $\sigma$-field for the Borel $\sigma$-field 
 $\mathcal{B}_{[0,\infty)}$.  Writing $\ooone=(\omega_v)_{v\in\T_1}\in\Omega_1=[0,\infty)^{\T_1}$ 
and $\ootwo=(\omega_v)_{v\in\T_2}\in \Omega_2=[0,\infty)^{\T_2}$,
$(\Omega,\P)$ can be represented as  a product of root, left, and right spaces defined
by:
\[\oo=(\omega_0,\ooone,\ootwo)\in\Omega=[0,\infty)\times\Omega_1\times\Omega_2,\] 
\[d\P(\oo)=d\P_\theta(\omega_\theta)\otimes d\P^{(1)}(\ooone)\otimes d\P^{(2)}(\ootwo)\] with 
$d\P_\theta(\omega_\theta)=e^{-\omega_\theta}d\omega_\theta$, and
$d\P^{(j)}(\oo^{(j)})=\prod_{v\in\T_j}e^{-\omega_v}d\omega_v$. Note that 
$(\Omega, \P)$, $(\Omega_1, \P^{(1)})$, $(\Omega_2, \P^{(2)})$ are isomorphic probability spaces up to relabeling.
\end{convention}
The recursion \eqref{alphrecursion} is satisfied almost surely for any $t>0$ and the 
solution processes are progressively measurable. 
 The rescaled exponential times in this
recursion naturally 
 form a tree-indexed random field on $(\Omega,\mathcal{F},\P)$ denoted by
$\bY = \{Y_v=\alpha^{-|v|}T_v:v\in\T\}$. 
$\bY$ is called an \emph{AAS cascade} associated with \eqref{mildalpha}. See \autoref{tree-edge2} for a realization of an AAS cascade. In this realization,
\begin{eqnarray}\nonumber X(t)&=&X^{(1)}(\alpha(t-T_\theta))X^{(2)}(\alpha (t-T_\theta))\\
\nonumber &=&X^{(11)}(\alpha^2(t-T_\theta-\alpha^{-1}T_1))X^{(12)}(\alpha^2(t-T_\theta-\alpha^{-1}T_1))\times\\
\nonumber &&X^{(21)}(\alpha^2(t-T_\theta-\alpha^{-1}T_2))X^{(22)}(\alpha^2(t-T_\theta-\alpha^{-1}T_2))\\
\label{recurex} &=&\ldots\end{eqnarray}
\begin{figure}[htp]
\centering
\includegraphics[scale=1]{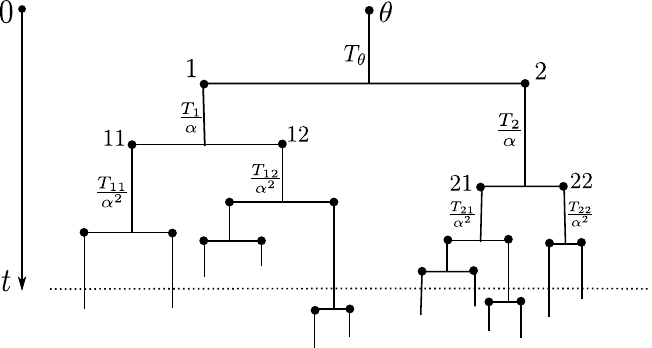}
\caption{Athreya-Aldous-Shields (AAS) cascade}
\label{tree-edge2}
\end{figure}

A {\em binary stochastic solution process} for
 the pantograph equation \eqref{mildpant} can also be defined recursively a.s. by
 \begin{equation}\label{pantbinary}
\sX(t) = \begin{cases} w_0 \ & \ \text{if} \ T_\theta> t\\
                             \frac{a}{2}\sX^{(1)}(\alpha(t-T_\theta))+
                             \frac{a}{2}\sX^{(2)}(\alpha(t-T_\theta)) \ &\  \text{if}\ T_\theta\le t,
                             \end{cases}
 \end{equation}
where $\sX^{(1)},\sX^{(2)}$  are copies of $\sX$ that are independent of $T_\theta$ and of each other.\footnote{Strictly speaking, the independence of $\sX^{(1)}$ and $\sX^{(2)}$ is not necessary. However, it is required for the 
transformations of stochastic processes in \autoref{Sec3}.}

If the indicated expectations exist, then $u(t)=\EXP X(t)$ and $w(t)=\EXP \sX(t)$ are each a solution to \eqref{mildalpha} and \eqref{mildpant}, respectively. 
For this reason, the equations \eqref{mildalpha} and \eqref{mildpant}
are referred to as the {\it mean flow equations} for solution processes  
$X(t)$ and $\sX(t)$, respectively.

Regarding the {\em time evolution} of the AAS cascade $\bY$, the recursion \eqref{alphrecursion} or \eqref{pantbinary} terminates at a vertex $v\in\T$ if and only if
 \[
 \sum_{j=0}^{|v|-1}Y_{v|j}\le t< \sum_{j=0}^{|v|}Y_{v|j}\,.
 \]
Such a vertex will be referred to as a \emph{$t$-leaf}. The continuous-parameter Markov process of sets of $t$-leaves is denoted by
\begin{equation}
\label{tleaves}
\partial{V}(t) =\left\{v\in\T: \sum_{j=0}^{|v|-1}Y_{v|j}\le t<\sum_{j=0}^{|v|}Y_{v|j}\right\}.
 \end{equation}
We denote the set of $t$-\emph{internal vertices} by
 \begin{equation}
\label{tancestors}
\Vo(t) = \left\{u\in\T:\ \sum_{j=0}^{|u|}Y_{u|j}\le t\right\},
\end{equation}
and the $t$-\emph{subtree} associated with $\bY$ by
\begin{equation}\label{ttree}
V(t) = \Vo(t) \cup \partial{V(t)}.
\end{equation}

 Analogous to the case of the unary tree, 
 $|\Vo(t)|$ 
 is the number of clock rings in the binary tree by time $t$.  Whenever $V(t)$ is a finite set, i.e.\ in the event of no explosion by time $t$, one has 
\begin{equation}\label{leafancestorcounts}
|\partial V(t)| = |\Vo(t)|+1,
\quad |V(t)| = 2 |\Vo(t)|+1, 
~~
  t\ge 0,
\end{equation}
Note that the root $\theta$ is a $t$-leaf (without 
 ancestors) if $T_\theta >t$.  

Up to the explosion time $S$,
i.e.\  on the event $[S>t]$, the set $V(t)$ takes values in a denumerable and partially ordered 
{\it evolutionary space} $\mathcal{E}$ consisting of all nonempty 
finite sets of vertices of subtrees rooted at $\theta$ of the binary tree $\T$.
Alternatively, the space $\mathcal{E}$ can be defined inductively as consisting of finite trees $A\subset\T$ 
such that either $A=\{\theta\}$ or $A=B\cup\{v1,v2\}$
for some $B\in\mathcal{E}$ and $v\in B$ whose offspring $v1,v2\notin B$. If $S>t$ then $N(t)=|\pV(t)|<\infty$ and 
the recursions  \eqref{alphrecursion} and \eqref{pantbinary} terminate in finitely many steps. This results in the 
{\em minimal solution processes}, in the sense of stopping the recursion at explosion, respectively defined by
\begin{equation}
\label{minimalalpha}
\lwrX(t)=u_0^{|\partial V(t)|}\IND{[S>t]}=u_0^{N(t)}\IND{[S>t]}
\end{equation}
for the $\alpha$-Riccati equation, and 
\[\underline{\sX}(t)=w_0 N(t)\IND{[S>t]}\;.
\]
for the pantograph equation when $a=2$.

Prior to  the explosion time $S$, $\partial{V}(t)$ evolves infinitesimally in time $t$ to $t+dt$ by removal of a $t$-leaf $v\in\partial{V}(t)$ 
  and replacement with its offspring $v1,v2$ at the 
  \emph{replacement} or \emph{branching time}
  \begin{equation}\label{agedef}
 \Theta_v=\sum_{j=0}^{|v|}Y_{v|j}\quad   (v\in\T),\quad\textup{with the convention }\  \Theta_{\overleftarrow{\theta}}=0\,.
 \end{equation}
In terms of its replacement time, a vertex $v\in\,\Vo(t)$ iff $\Theta_v\le t$, i.e.\ $v$ is replaced prior to time $t$. Also, $v\in \partial V(t)$ iff $\Theta_v>t$, 
but
its parent $\overleftarrow{v}$ is replaced prior to $t$.
One may say that \lq\lq $v\in \partial V(t)$ 
crosses $t$, while its parent $\overleftarrow{v}$ does not cross $t$\rq\rq.
We note that in the statement of \autoref{main3}, $N(t)=|\partial V(t)|$, i.e., \textit{surviving particles} correspond to $t-$leaves.

After the explosion time, i.e.\ $t>S$, the set of
 $t$-leaves $\partial V(t)$ remains well-defined but may become infinite. It may also evolve to an empty set at the (possibly 
  finite) hyperexplosion time defined by
  \begin{equation}\label{longdef}
  L = \sup_{s\in\partial{\T}}\sum_{j=0}^\infty Y_{s|j}.
  \end{equation}
  The event $[L<\infty]$ is referred to as {\it hyperexplosion}.
The notations $S$ and $L$ for explosion and hyperexplosion times, respectively,  are used to 
convey `shortest' and `longest' path lengths as measured in terms of the 
{\em{total time accumulated}} on a ray $s\in\partial\T$,
\begin{equation}\label{totaltime}
\Theta_s=\sum_{j=0}^\infty Y_{s|j}\,.
\end{equation}   
With this notation, one has
\begin{equation}\label{rayexpleq}
S =\inf_{s\in\partial\T}\Theta_s = \liminf_{ n\rightarrow\infty, |v|=n} \Theta_v,\quad L=\sup_{s\in\partial\T}\Theta_s =
\limsup_{ n\rightarrow\infty, |v|=n} \Theta_v.
\end{equation}

\color{black}

\begin{prop}\label{contdist} 
Let $\alpha >1$.  Then:
\begin{itemize}
\item[]
(i) The function $\Theta:\partial\T\rightarrow [0,\infty]$ is a.s.  a finite continuous function for the product topology on $\partial\T.$
\item[]
(ii) $S$ and $L$ have infinitely differentiable distribution functions $F_S$ and $F_L$, respectively, each with a strictly positive density function.
\end{itemize}
\color{black}
\end{prop}
\begin{proof}
For (i), recall that $\pT=\{1,2\}^\infty$ is a compact topological product space, and note that for each $n,$ $s\in\pT \mapsto \Theta_{s|n}\in[0,\infty)$ is continuous in the product topology.  Furthermore, a.s. $\Theta_{s|n}\rightarrow \Theta_s$ uniformly implying that $\Theta$ is a.s. continuous.  To see this, let $M_k = \max\{T_{s|k}, s\in\partial\T\}$ and note that this is the maximum of $2^k$ i.i.d.\ mean one exponential  random variables.  Also one has  $\EXP M_k = \sum_{j=1}^{2^k} \tfrac{1}{j} < k \ln(2) +1$,  as well as the following bound
independently of the choice of  $s\in\partial\T,$
\[\Theta_s - \Theta_{s|n} = \sum_{k=n+1}^\infty \frac{T_{s|k}}{\alpha^k} \le   \sum_{k=n+1}^\infty \frac{M_k}{\alpha^k}.\]
Observing that since
$\EXP[\sum_{k=0}^\infty\alpha^{-k}M_k]<\infty$,
it follows that $\sum_{k=0}^\infty\alpha^{-k}M_k<\infty$ a.s., the
above uniform
bound
on $\sup_{s\in\pT}|\Theta_s-\Theta_{s|n}|$ is a.s.
$o(1)$ as $n\to\infty$.

\color{black}
We now prove part (ii). 
According to \eqref{alphRrecur}, the complementary distribution $\bar{F}_S(t)=\P(S>t)$ satisfies \eqref{mildalpha} with $u_0=1$.  Likewise, the relation
\begin{equation}\label{alphRrecur2}
 L \xlongequal{\text{dist}} Y_\theta + \alpha^{-1}\max\{L^{(1)},L^{(2)}\}
 \end{equation}
implies that $F_L(t)=\P(L\le t)$ satisfies \eqref{mildalpha} with $u_0=0$. Note, also, that a bounded function that satisfies \eqref{mildalpha}, is necessarily continuous. The differentiability properties follow inductively from the implied smoothness of functions satisfying the integral equation \eqref{mildalpha}.

We now show that $\bar{F}'_S(t) < 0$ for all $t>0$. The proof for $F_L'(t)>0$ for all $t>0$ can be done similarly. Let $f=\bar{F}'_S$. Recall that $\bar{F}_S$ is a differentiable complementary distribution function, so $f=\bar{F}_S'\le 0$. Differentiating both sides of the differential equation $\bar{F}'_S(t)+\bar{F}_S(t)=\bar{F}^2_S(\alpha t)$, we get
\begin{equation}\label{fF}
f'(t)+f(t)=2\alpha f(\alpha t)\bar{F}_S(\alpha t)\quad\forall t>0.
\end{equation}
Multiplying both sides by $e^t$, we get that $g(t)=e^tf(t)$ is a non-increasing function on $[0,\infty)$. Suppose by contradiction that $f(t_0)=0$ for some $t_0>0$. Then $g(t_0)=0$, which implies $g=0$ on $[0,t_0]$ due to the monotonicity of $g$. Thus, $f=0$ on $[0,t_0]$. Substituting $t\in[0,t_0]$ into \eqref{fF} and noting that $\bar{F}_S(\alpha t)=\P(S>\alpha t)\ge \P(Y_\theta>\alpha t)=e^{-\alpha t}>0$, one gets $f(\alpha t)=0$ for all $t\in [0,t_0]$. That is, $f=0$ on $[0,\alpha t_0]$. Using \eqref{fF} repeatedly, we get $f=0$ on $[0,\alpha^2 t_0]$, $[0,\alpha^3 t_0]$,\ldots Hence, $f$ is identically zero. This contradicts the fact that $-f$ is a probability density function.
\end{proof}

\color{black}
On the non-explosion event $[S>t]$, the recursion \eqref{alphrecursion} must yield the minimal 
solution process $\lwrX$ given by \eqref{minimalalpha}. In the explosive case, however, non-minimal solution processes can 
be constructed. For $u_0\ge 0$, one such construction is the {\em maximal solution process} defined by (see \cite[Thm 4.1]{alphariccati})
\begin{equation}\label{max_sol_proc}
\uprX(t)=u_0^{|\partial V(t)|}\quad \forall t>0
\end{equation}
with the convention that $0^0=1$. 

The following theorem states that the distributions of $S$ and $L$ have an exponential rate of convergence as $t\rightarrow\infty$. The notation $f(t)\sim e^{-t}$ 
means that for some $c,t_0>0$, $f(t)/e^{-t}\in[1/c,c]$ for 
all $t>t_0$.
\begin{thm}\label{expl_asymptotics}
Let $\alpha>1$.  Then (i) $\P(L\ge t)\sim e^{-t}$, (ii) $\P(S\ge t)\sim e^{-t}$, (iii) ${\EXP}S\le \EXP L<\infty$.
\end{thm}

\begin{proof} Since $L\ge S\ge Y_\theta$,  
\[\P(L\ge t)\ge \P(S\ge t)\ge \P(Y_\theta\ge t)=e^{-t}\quad\forall t\ge 0.\]
It remains to show that $\P(L\ge t)=O(e^{-t})$ as $t\to\infty$. We will first show
 that 
\begin{equation}\label{first_expl_asymptotics}
\P(L>t)=O(e^{-\frac{\alpha-1}{\alpha}t})\qquad\mbox{as $t\to\infty$.} 
\end{equation}
Following the idea of \cite{athreya} ,
let $M_j = \alpha^{-j}\max_{|w|=j}T_{w}$ and
\[L_n=\max_{|v|=n}\sum_{j=0}^n\alpha^{-j}T_{v|j}
\le\sum_{j=0}^n M_j.\]
Then
\[L=\lim L_n\le
\sum_{j=0}^\infty M_j.\]
Fix $t\ge 0$ and define
\[\delta_j=\frac{t+j\ln2+2\ln(j+1)+\ln c_0}{\alpha^j},\quad j=0,1,\dots\]
where $c_0=\pi^2/6>1$. Then
$\sum_{j=0}^\infty\delta_j =  c_1t+c_2$ where
$$c_1=\sum_{j=0}^\infty\alpha^{-j}=\frac{\alpha}{\alpha-1},\quad c_2= \sum_{j=0}^\infty\frac{j\ln2+2\ln(j+1)+\ln c_0}{\alpha^j}.$$
On the event $[L\ge c_1t+c_2]$, there exists $j_0\ge 0$ such that $M_{j_0}\ge\delta_{j_0}$. Thus,
\begin{align}
\P(L\ge c_1t+c_2)&\le \sum_{j=0}^\infty \P(M_j\ge\delta_j)=\sum_{j=0}^\infty \P\left(\max_{|w|=j}T_w\ge \alpha^j\delta_j\right)\\
&\le \sum_{j=0}^\infty 2^j\P(T_{w_j}\ge \alpha^j\delta_j)=\sum_{j=0}^\infty 2^j e^{-\alpha^j\delta_j}=e^{-t}.
\end{align}
Here, $w_j=(1,1,...,1)\in\T$ is a vertex of height $j$.

Setting $t^\prime=c_1t+c_2$, i.e.\ $t=(t^\prime-c_2)/c_1$, one has
\[
\P(L\ge t)\le e^{-\frac{t-c_2}{c_1}} \qquad \forall t\ge c_2,
\] 
and \eqref{first_expl_asymptotics} follows. 

We have shown that the function $u(t)=\P(L\le t)$ satisfies \eqref{diffalpha-ricc0} such that $0\le 1-u(t)\le C e^{-\rho t}$, with $\rho=(\alpha-1)/\alpha>0$. By \autoref{alphasharprmk}, $1-u(t)\lesssim e^{-t}$ as $t\to\infty$. This completes the proof.
\end{proof}

 \section{Pantograph Non-uniqueness and 
 Transformations of Solution Processes}\label{Sec3}

Let us begin with a general orientation
to their construction  and 
the strategy for their use in forthcoming proofs.
First, to obtain the nonuniqueness results of \autoref{main}, we connect the a.s. additive solution processes $\sX(t)$ satisfying \eqref{pantbinary} with $a=2$ and $w_0=0$ 
to the multiplicative solution processes $X=X_\lambda$ satisfying  \eqref{alphrecursion} 
with $u_0=1$ via
\begin{equation}\label{trans1}
X_\lambda(t)=e^{-\lambda\sX(t)},\quad \lambda\in\mathbb{R}.
\end{equation}
This transformation establishes a relationship between solutions of (\ref{diffalpha-ricc0}) with $u_0=1$, and its linearization about the 
constant steady state $u\equiv 1$:
\begin{equation}\label{markovsspantographeq}
w^\prime(t) = -w(t) + 2w(\alpha t), \quad w(0)=0.
\end{equation}
While the transformation of solution processes is natural, it leads to a notable relation between these two mean flow equation that is not available via standard analysis methods.  A {\it one-to-many principle} is established by the proof of \autoref{main}. Namely, a
 single nontrivial solution process $\sX$ for \eqref{markovsspantographeq}, generates a one-parameter family of distinct multiplicative solution processes $X_{\lambda}$, all corresponding to $u_0=1$. 
These solutions are distinguished by their distinct asymptotic behavior as $t\rightarrow\infty$.
In the case $\sX\ge 0$ and $\lambda<0$, the $X_\lambda$'s are automatically integrable and 
therefore result in a-priori distinct solutions $u_\lambda=\EXP(X_\lambda)$ to the $\alpha$-Riccati equation with 
$u_0=1$.
Note that knowledge of $w = \EXP(\sX)$ is not 
sufficient to define $u_\lambda$. Detailed 
information obtained from $\sX$ is necessary to obtain $u_\lambda$. We use the transformation \eqref{trans1} to prove \autoref{main} in the case $\alpha>2$ (\autoref{alpha_bgr_2_case}).

A second way in which to generate solution processes satisfying  \eqref{alphrecursion}  with $u_0=1$ 
is by noting that if $X$ is such a solution process then 
\begin{equation}\label{trans2}
X_\lambda(t)=X(t)^{\lambda},\qquad \lambda\in\mathbb{R}
\end{equation}
is also an a.s. defined solution process corresponding to $u_0=1$.
In order to ensure that $X_\lambda$ are distinct processes, one needs to construct a 
{\em non-indicator} solution process $X$. We say that $X(t)$ is an {\em indicator} process if $$
\forall t>0, ~~ X(t)\in\{0,1\}$$ almost surely. If $X$ is non-indicator and $X(t)\in[0,1]$ a.s for any 
$t>0$ then $X_\lambda$ are integrable for all $\lambda>0$ and the family $\{u_\lambda=\EXP(X_\lambda)\}_{\lambda>0}$ are a-priori distinct solutions to \eqref{mildalpha} with $u_0=1$. The transformation \eqref{trans2} will be used to prove \autoref{main} in the case $1<\alpha\le 2$ (\autoref{firstcase}).

To generate solution processes satisfying  \eqref{alphrecursion} with $u_0\ne 1$,  the following observation is used. Given a solution process $\tilde{X}_{u_0}$ satisfying \eqref{alphrecursion} with initial data 
$u_0$ and solution processes $X_\lambda$ corresponding to $u_0=1$, it is simple to check that
\begin{equation}\label{trans3}
X_{u_0,\lambda}(t)=\tilde{X}_{u_0}(t) X_{\lambda}(t), \qquad \lambda\in\mathbb{R}
\end{equation}
are a.s. defined solution processes satisfying \eqref{alphrecursion} with the initial data $u_0$.
In order to obtain distinct processes $X_{u_0,\lambda}$ for
distinct $\lambda\in\mathbb{R}$, the original process $\tilde{X}$ must be {\em non-minimal}. 
Indeed, as noted in \autoref{binary_setup}, on the non-explosion event $[S>t]$, all the solution 
processes constructed over the AAS cascade $\bY$ by the recursion 
 \eqref{alphrecursion} must coincide with the minimal process, yielding
$\lwrX X_{\lambda}(t)=\lwrX X_{\lambda}(t)\IND{[S>t]}=\lwrX 1\,\IND{[S>t]}=\lwrX$.
The transformation \eqref{trans3} is used to prove \autoref{main} for initial data $u_0$ in a suitable range (\autoref{generalinitialdata}). As it turns out, a natural choice for $\tilde{X}_{u_0}$ is the maximal 
solution process given by \eqref{max_sol_proc} for $u_0>0$ and by $\IND{[L\le t]}$ for $u_0=0.$

To generate a one-parameter family of distinct solutions $u_\lambda$ to the $\alpha$-Riccati equation with $u_0=1$ using the transformation \eqref{trans1},
one needs the existence of a nontrivial solution process $\sX$ associated with the pantograph equation with $w_0=0$. 
For this purpose, we  construct a solution to the pantograph equation, $\eta(t),$ with algebraic decay using the unary tree approach described in 
\autoref{probpanto}. 
The algebraic decay of $\eta$ 
plays an essential role in the construction of the process $\sX$ via a stochastic Picard 
iterations method developed in (\autoref{alpha_bgr_2_case}). For a fixed $\alpha>0$, define
\begin{equation}\label{1190-1}
\gamma_a=\log_{\alpha}{a}=\frac{\ln\, a }{\ln\, \alpha}.
\end{equation}
\begin{thm}\label{prop11112}
Let $a>0, \alpha>\max\{a, 1\}$ so that $\gamma_a\in(-\infty,1)$.  Define 
 \begin{equation}\label{X_*-def}
 \tsX_*(t)=\tsX_*(t;a,\alpha)=(t-\tS)^{-\gamma_a} \IND{[\tS<t]},
 \end{equation}
  where $
\tS$ is the unary explosion time given by \eqref{tilde_S}.
Then
\begin{enumerate}[(i)]
\item $\tsX_*$ a.s. satisfies \eqref{pantrecursion} with $w_0=0$.
\item $\eta(t)=\EXP\tsX_*(t)$ satisfies \eqref{mildpant} with $w_0=0$.
\item $\lim\limits_{t\to\infty}t^{\gamma_a}\eta(t)=1$.
\end{enumerate}
\end{thm}

\begin{proof}
(i) Let $\tS^{(1)}=\sum_{j=1}^\infty\frac{T_j}{\alpha^{j-1}}$ and $\tsX_*^{(1)}(\tau)
=(\tau-\tS^{(1)})^{-\gamma}\IND{[\tS^{(1)}<\tau]}$.
Note that $\tS$ and $\tS^{(1)}$  are identically distributed. So are $\tsX_*$ and $\tsX_*^{(1)}$. For the rest of the proof, we will write $\gamma$ instead of $\gamma_a$ for simplicity.
\begin{align*}
\tsX_*(t)&=(t-\tS)^{-\gamma}\IND{[\tS<t]}=\left(t-T_0-\alpha^{-1}\tS^{(1)}\right)^{-\gamma}
\IND{[T_0+\alpha^{-1}\tS^{(1)}<t]}\\
&=
\alpha^\gamma\left(\alpha(t-T_0)-\tS^{(1)}\right)^{-\gamma
}\IND{[\tS^{(1)}<\alpha(t-T_0)]}=
a\tsX^{(1)}_*(\alpha(t-T_0))\,.
\end{align*}
If $T_0\ge t$, the calculations above yield $\tsX_*(t)=0$. Thus, in distribution, $\tsX_*(t)$ 
satisfies \eqref{pantrecursion} with $w_0=0$.

(ii) By \eqref{complementary}, one obtains the $pdf$ of the unary explosion time $\tS$
\[
g(t) = - G'(t)= C_\alpha \sum_{n=0}^\infty \frac{\alpha^n e^{-\alpha^n t}}{\prod_{j=1}^n (1-\alpha^{j})}
\]
since term-by-term differentiation can easily be justified. Consequently, $g(t) \le C e^{-t}$ for all $t\ge 0.$ 
Thus $\tsX_*(t)$ is integrable since the distribution of $\tilde{S}$ is exponentially decaying and $\gamma<1$. The fact that $\eta(t)$ satisfies \eqref{mildpant} with $w_0=0$ can be justified from \eqref{pantrecursion} as follows.
\begin{eqnarray*}
\eta (t)=\EXP \tsX_*(t)&=&\EXP\left[ \left( t-T_0-\alpha^{-1}\tS^{(1)} \right)^{-\gamma} 
\IND{[\tS^{(1)}<\alpha (t-T_0)]} \right] \\
&=& \alpha^{\gamma}
\int_0^t e^{-s} \EXP \left[\left(\alpha(t-s) - \tS^{(1)}\right)^{-\gamma} \IND{[\tS^{(1)} < \alpha(t-s)]}  \right] 
ds \\
&=& a \int_0^t e^{-s} \eta(\alpha(t-s)) ds.
\end{eqnarray*}
(iii) Let $\phi(t)=t^{-\gamma}\IND{t>0}$ and extend $g$ by 0 o
n the interval 
$(-\infty,0)$. Then
\[\eta (t)=\int_{0}^{t}{{{(t-s)}^{-\gamma}}g(s)ds}=\int_{-\infty }^{\infty }{\phi (t-s)g(s)ds}\]
and
\[\frac{\eta (t)}{{{t}^{-\gamma }}}=\int_{-\infty }^{\infty }{\frac{\phi (t-s)}{\phi (t)}g(s)ds}=\EXP W(t)\]
where $W(t)=\frac{\phi(t-\tS)}{\phi(t)}\IND{[\tS<t]}$. Due to explosion, $\tS<\infty$ a.s. Thus,
$\lim_{t\to\infty}W(t)=1$ a.s. We show that 
\begin{equation}\label{lim1}\lim_{t\to\infty}\EXP W(t)= \EXP[1]=1. 
\end{equation}
If $\gamma\le 0$ then $0\le W(t)\le 1$ for any $t>0$. Then \eqref{lim1} follows by Lebesgue's Dominated Convergence Theorem. For $\gamma\in(0,1)$, we will show \eqref{lim1} by showing that there exist $p\in(1,\infty)$ and $\tilde{C}>0$ such that 
\[\EXP[W(t)^p]\le \tilde{C}\ \ \ \ \forall\,t\ge 1.\]
Fix $p\in(1,\frac{1}{\gamma})$. It suffices to show that $\limsup_{t\to\infty}\EXP[W(t)^p]<\infty$. We have
\begin{eqnarray*}
\mathbb{E}[W{{(t)}^{p}}]&=&\EXP\left[ \frac{{{\phi }^{p}}(t-\tilde{S} )}{{{\phi }^{p}}(t)}{{1}_{\tilde{S} <t}} \right]
=\int_{0}^{t}{\frac{{{(t-s)}^{-\gamma p}}}{{{t}^{-\gamma p}}}g(s)ds}\\
&=&\int_{0}^{t}{\frac{{{s}^{-\gamma p}}}{{{t}^{-\gamma p}}}g(t-s)ds}\le C\int_{0}^{t}
{\frac{{{s}^{-\gamma p}}}{{{t}^{-\gamma p}}}{{e}^{-(t-s)}}ds}\\
&=&C\frac{\int_{0}^{t}{{{s}^{-\gamma p}}{{e}^{s}}ds}}{{{t}^{-\gamma p}}{{e}^{t}}}.
\end{eqnarray*}
By L'Hospital's Rule, 
\begin{eqnarray*}
\underset{t\to \infty }{\mathop{\limsup }}\,\mathbb{E}[W{{(t)}^{p}}]
&\le& C\underset{t\to \infty }{\mathop{\lim }}\,
\frac{\int_{0}^{t}{{{s}^{-\gamma p}}{{e}^{s}}ds}}{{{t}^{-\gamma p}}{{e}^{t}}}
=C\underset{t\to \infty }{\mathop{\lim }}\,\frac{{{t}^{-\gamma p}}{{e}^{t}}}{{{t}^{-\gamma p}}{{e}^{t}}
-\gamma p{{t}^{-\gamma p-1}}{{e}^{t}}}\\
&=&C\underset{t\to \infty }{\mathop{\lim }}\,\frac{1}{1-\frac{\gamma p}{t}}=C
\end{eqnarray*}
which completes the proof.
\end{proof}
 
At the level of the mean flow equations 
it is certainly unclear how  a solution to \eqref{genpantode} 
(with $w_0=0$) with parameters $(a,\alpha)$ relates to that with parameters $(a',\alpha)$.  However, the following corollary shows 
how a simple transformation at the scale of 
stochastic solution processes can forge such a connection. 

\begin{cor}{\label{exampleone}}
Under the same assumptions of \autoref{prop11112} and with $\gamma_a$ given by \eqref{1190-1}, for any $\delta>0$ such that $\gamma_a\delta<1$, the function
\begin{equation}\label{etadelta}\eta_\delta(t)=\EXP[\tsX_*(t;a,\alpha)^\delta]\end{equation}
satisfies
\begin{equation}\label{etadelta1}\eta'_{\delta}(t)=-\eta_{\delta}(t)+a^\delta\eta_{\delta}(\alpha t),\quad \eta_\delta(0)=0,\end{equation} 
and
\begin{equation}\label{11132}\lim_{t\to\infty}\frac{\eta_\delta(t)}{t^{-\gamma_a\delta}}=1.\end{equation}
\end{cor}

\begin{proof}

Let $a'=a^\delta<\alpha$. By their definitions,
$\gamma_{a'}=\delta\gamma_a$ and
$(\tsX_*(t;a,\alpha))^\delta=\tsX_*(t;a',\alpha)$.  \autoref{prop11112} implies that the function $\eta_\delta$ given by \eqref{etadelta}
satisfies \eqref{etadelta1} and \eqref{11132}.
\end{proof}
This analysis of the pantograph equation \eqref{mildpant} via the unary process \eqref{pantrecursion} complements the well-developed theory in \cite{kato1971functional} 
by providing a simple construction of a solution $\eta$ decaying at a critical algebraic rate. With the existence of such a solution, we next revisit  the problem of constructing solutions to \eqref{mildalpha} using the transformations \eqref{trans1} and \eqref{trans2}. The method of building solution processes is by stochastic Picard iterations 
described in the next section.

\section{Solution Processes via Stochastic Picard Iterations}\label{PicardIteration} 

The method of 
{\em stochastic Picard iterations} is designed
to iteratively construct non-minimal solution processes for given initial data. The initialization of
the iteration is
referred to as the ground state. When taken 
in conjunction with transformations of stochastic processes described in \autoref{Sec3}, 
this will allow one to prove nonuniqueness results for \eqref{mildalpha} for a range of initial 
data (\autoref{main}). The essential property that allows 
one to construct multiple solution processes corresponding 
to the same initial data $u_0$ is the {\em hyperexplosion} of the AAS cascade for $\alpha>1$.

The method of {\em stochastic Picard ground state iterations}, 
or simply {\em stochastic Picard iterations}, was essentially introduced in \cite{lejan} to prove a uniqueness result for the Navier-Stokes equations. It was later elaborated in \cite{alphariccati} as an approach to show the nonuniqueness of solutions to the mean flow equations from stochastic explosion.  This method is suitable for both linear and nonlinear equations in contrast to Feller's probabilistic approach to linear parabolic equations as Kolmogorov backward equations; 
see
\autoref{fellerapp}.

The stochastic Picard iteration corresponding to \eqref{alphrecursion} takes an especially simple form:
\begin{equation}\label{X_n-eq}
X_n(t) \stackrel{\rm a.s.}{=} 
\left\{ \begin{array}{*{35}{l}}
   u_0 & \text{if}\ T_\theta>t  \\
   {X^{(1)}_{n-1}}(\alpha(t-T_\theta))\,{X^{(2)}_{n-1}}(\alpha(t-T_\theta)) & \text{otherwise} ,\\
\end{array} \right.\end{equation}
where $T_\theta$ is a mean-one exponentially distributed random variable; $X_{n-1}^{(1)}$ and $X_{n-1}^{(2)}$ are independent copies of 
$X_{n-1}$ that are also independent of $T_\theta$. The initial stochastic process $X_0(t)$, i.e., the \emph{ground state}, is to be chosen to guarantee the convergence of the sequence $\{X_n\}$.
Since $X_0(t)$ is a stochastic process, one obtains by induction on $n\in\mathbb{N}$ that $X_n(t)$ is a 
well-defined progressively measured stochastic process. Moreover, if $ u^{(0)} =\EXP (X_0)$ is 
well-defined, then the sequence $u^{(n)}=\EXP(X_n)$ is well-defined and
formally satisfies the Picard iterations of \eqref{mildalpha}:
\begin{eqnarray}\label{Pic-iteration}
u^{(n)}(t) &=& u_0e^{-t} + \int_0^t e^{-s}[u^{(n-1)}(\alpha (t-s))]^2 ds.
\end{eqnarray}
Using the Convention \ref{convention} on the product probability space 
$(\Omega,\mathcal{F},\P)$, one has the following 
convergence of the stochastic Picard iterations to a solution process.
\begin{thm}\label{X_n-conv}
Let $X_n(t)$ be a sequence satisfying the stochastic Picard iterations \eqref{X_n-eq}. 
Suppose that for all $t>0$, 
$X_n(t)$ converges a.s.\ as $n\to\infty$. 
Then there exists a stochastic process $X(t)$ such that $X_n(t,\omega)\to X(t,\omega)$ a.e. on $[0,\infty)\times\Omega$ 
with respect to the product measure $\mu_{[0,\infty)}\otimes\P$ as $n\to\infty$, where 
$\mu_{[0,\infty)}$ is a 
Borel measure on $[0,\infty)$. Moreover, for all $t>0$, $X(t)$ is a solution process for 
the $\alpha$-Riccati equation, i.e.\
satisfying \eqref{alphrecursion}. 
\end{thm}

\begin{proof}
Denote by $\T_j=\{v\in\T\backslash\{\theta\} :\ v|1=j\}$, for $j=1,2$, the left and right subtrees of $\T$, respectively. 
Write $\ooone=(\omega_v)_{v\in\T_1}\in\Omega_1=[0,\infty)^{\T_1}$ 
and $\ootwo=(\omega_v)_{v\in\T_2}\in \Omega_2=[0,\infty)^{\T_2}$. 
Thus, we can view $(\Omega,\P)$ as a product space: 
\[\oo=(\omega_0,\ooone,\ootwo)\in\Omega=[0,\infty)\times\Omega_1\times\Omega_2,\] 
\[d\P(\oo)=d\P_\theta(\omega_\theta)\otimes d\P^{(1)}(\ooone)\otimes d\P^{(2)}(\ootwo)\] with 
$d\P_\theta(\omega_\theta)=e^{-\omega_\theta}d\omega_\theta$, and
$d\P^{(j)}(\oo^{(j)})=\prod_{v\in\T_j}e^{-\omega_v}d\omega_v$. Note that 
$(\Omega, \P)$, $(\Omega_1, \P^{(1)})$, $(\Omega_2, \P^{(2)})$ are isomorphic probability spaces up to relabeling.

Let ${\cE}=\{(t,\overline{\omega})\in[0,\infty)\times\Omega:\ X_n(t,\overline{\omega})\ \mbox{converges 
in $\mathbb{R}$}\}$. 
Since $X_n$ is measurable in $t$ and $\overline{\omega}$, ${\cE}$ is measurable with respect 
to the $\sigma$-algebra 
$\mathcal{B}_{[0,\infty)}\otimes\mathcal{B}_{\Omega}$. Since  
$X_n(t)$ converges a.s.\ for all $t>0$, one has 
$\P(\cE_t)=1$, where $\cE_t=\{\overline{\omega}\in\Omega:(t,\overline{\omega})\in\cE\}$. By the Fubini's Theorem,
\[
[\mu_{[0,\infty)}\otimes\P](\cE^c)=\int_0^{\infty}\P(\cE^c_t)\,dt=0.
\]
Define 
\[
X(t,\overline{\omega})=\begin{cases}
\lim\limits_{n\to\infty}X_n(t,\overline{\omega}), & (t,\overline{\omega})\in E\\
0, & \mbox{otherwise}.
\end{cases}
\]
Clearly, $X(t)$ is a well-defined progressively measured stochastic process and $X_n\to X$ a.e.

To show $X(t)$ is a solution process, fix a $t>0$. Suppose $T_\theta=\omega_\theta<t$. 
In this case, in \eqref{X_n-eq} and \eqref{alphrecursion}, $X_{n-1}^{(j)}(\alpha(t-T_\theta))
=X_{n-1}(\alpha(t-\omega_\theta),\oo^{(j)})$, 
and $X^{(j)}(\alpha(t-T_\theta))=X(\alpha(t-\omega_\theta),\oo^{(j)})$, $j=1,2$. 
Note that $X_n^{(j)}\IND{[\omega_{\theta}<t]}$ is measurable in $(t,\oo)$ for all $n\in\mathbb{N}$. Let 
\[
\begin{aligned}
\cE^{(j)}_t(\omega_\theta)=&\left\{\oo^{(j)}\in\Omega_j\,:\ \omega_{\theta}<t,\ \mbox{and}\ \right.\\
&\left. X_{n-1}(\alpha(t-\omega_\theta),\oo^{(j)})\to X(\alpha(t-\omega_\theta),\oo^{(j)})\ 
\mbox{as}\ n\to\infty\right\}.
\end{aligned}
\] 
Since $X_n\to X$ a.s., for $\omega_\theta<t$ one has
\[\P^{(j)}(\cE^{(j)}_t(\omega_\theta))=\int_{\Omega_j}\IND{\cE^{(j)}(\omega_\theta)}(\oo^{(j)})\, 
d\P^{(j)}(\oo^{(j)})=1,\quad j=1,2.\] 
Let 
\[\tcE_t=\left\{\oo=(\omega_\theta,\ooone,\ootwo)\, :\ \omega_\theta\ge t\ \mbox{or}\ 
\left(\omega_\theta<t\ \mbox{and}\ \oo^{(j)}\in\cE^{(j)}_t(\omega_\theta),\ j=1,2\right)\right\}\,.\] 
We have
\[
\begin{aligned}
\tcE_t=&\{\oo\, :\ T_\theta(\oo)\ge t\}\cup\{\oo\, :\ T_\theta(\oo)<t\ \mbox{and}\ \\
& X_n(\alpha(t-T_\theta(\oo)),\oo^{(j)})\to X(\alpha(t-T_\theta(\oo)),\oo^{(j)}),\ j=1,2\}.
\end{aligned}
\]
Since $X_n(\cdot,\cdot)$, $T_\theta(\cdot)$ and the mappings $(t,\oo)\to(\alpha(t-\omega_\theta),\oo^{(i)})$, $j=1,2$,  
are measurable, we conclude that $\tcE_t$ is measurable. By Fubini's Theorem,
\[
\begin{aligned}
\P(\tcE_t) &=
e^{-t}+\\
&+\int_0^t e^{-\omega_\theta}\, \int_{\Omega_1}\IND{\cE^{(1)}_t(\omega_\theta)}(\oo^{(1)})\, 
d\P^{(1)}(\oo^{(1)})\,
\int_{\Omega_2}\IND{\cE^{(2)}_t(\omega_\theta)}(\oo^{(2)})\, d\P^{(2)}(\oo^{(2)})\, d\omega_\theta\\
&=
e^{-t}+\int_0^t e^{-\omega_\theta}\, 1\cdot 1\, d\omega_\theta=1
\,.
\end{aligned}
\]
For any $\oo\in\tcE_t$,  \eqref{alphrecursion} follows from 
\eqref{X_n-eq} by taking $n\to\infty$.
\end{proof}

By a straightforward adaptation of the proof above, one obtains a similar convergence result 
for the stochastic Picard iterations corresponding to the pantograph equation  \eqref{genpantode} 
in both binary tree representation \eqref{pantbinary} and the unary tree representation \eqref{pantrecursion}.

\begin{cor}\label{pant_iter_conv}
Let $\sX_n(t)$ be the sequence of binary-tree stochastic Picard iterations for \eqref{genpantode}
\begin{equation}\label{pant_bin_iter}
{{\sX}_{n}}(t)
\stackrel{\rm a.s.}{=} \left\{ \begin{array}{*{35}{l}}
   w_0 & \text{if} & T_\theta\ge t,  \\
   \frac{a}{2}\mathscr{\sX}_{n-1}^{(1)}(\alpha (t-T_\theta))
   +\frac{a}{2}\sX_{n-1}^{(2)}(\alpha (t-T_\theta)) & \text{if} & T_\theta<t.  \\
\end{array} \right.
\end{equation}
and $\tsX_n(t)$ is the sequence of unary-tree stochastic Picard iterations for \eqref{genpantode}
\begin{equation}\label{pant_unary_iter}
{{\tsX}_{n}}(t)
\stackrel{\rm a.s.}{=} \left\{ \begin{array}{*{35}{l}}
   w_0 & \text{if} & T_0\ge t,  \\
   a\mathscr{\sX}_{n-1}^{(1)}(\alpha (t-T_0)) & \text{if} & T_0<t.  \\
\end{array} \right.
\end{equation}
Denote $\mathcal{Y}_n=\sX_n$ or $\tsX$ and $\mathcal{Y}=\sX$ or $\tsX$, respectively. Suppose that for all $t>0$, $\mathcal{Y}_n(t)$ converges a.s.\ as $n\to\infty$. Then there exists 
a stochastic process 
$\mathcal{Y}(t)$ such that $\mathcal{Y}_n(t,\omega)\to \mathcal{Y}(t,\omega)$ as $n\to\infty$ a.e.\ with respect to 
the product measure $\mu_{[0,\infty)}\otimes\P$. Moreover, for all $t>0$, $\mathcal{Y}(t)$ is a solution process for the pantograph equation, i.e.\ satisfying \eqref{pantbinary} or \eqref{pantrecursion}. 
\end{cor}

Some remarks about the impact of different ground states in the stochastic Picard iteration are appropriate. First, note that on the event $[S>t]$, the Picard iteration at level $n$, $X_n(t)$, defined by \eqref{X_n-eq} is an eventual constant sequence, say $X(t)$, clearly satisfying \eqref{alphrecursion}. Consequently, for $\alpha\le 1$ corresponding to a non-explosive AAS cascade, the choice of different ground states does not change the solution process obtained as limits of stochastic Picard iterations.

In contrast, 
in the explosive case, i.e. when $\alpha>1$, $\P(S<\infty)=1$, different choices of the ground state 
$X_0$, could lead to super-martingales yielding 
multiple solutions in the
limit for the same initial states $u_0$, \cite{alphariccati}. 
Notably, for the initial condition $u_0=0$, the choice of a random initial iteration, 
\begin{equation}
\label{X_0-G}
X_0(t)
\stackrel{\rm a.s.}{=} \left\{\begin{array}{ll}
0, & T_{\theta}\ge t,\\
G(t-T_{\theta}),\quad&  T_{\theta}< t,
\end{array}\right.
\end{equation}
where $G$ is a continuous function, leads to a {uniformly integrable} super-martingale $\{X_n\}$, provided 
$u^{(0)}=\EXP(X_0)$ satisfies $[u^{(0)}(\alpha t)]^2\le G(t)$. If one chooses $G(t)\in[0,1]$ (e.g. 
$G(t)\equiv0$ or $G\equiv 1$), one obtains a uniformly integrable super-martingale, yielding solutions in the limit of 
expected values for the $\alpha$-Riccati equation \eqref{diffalpha-ricc0}. One remarkable choice of ground state 
used in \cite{alphariccati} is a perturbation of the Fr\'{e}chet extreme value distribution given by
\begin{equation}
    \label{Frechet}
    F_\gamma(t) = e^{-t^{-\gamma}} \IND{(0,\infty)}(t).
\end{equation}
where $\gamma$ is defined by \eqref{gma}. More specifically, 
\[G_A(t)=e^{-t^{-\gamma}}(1+\gamma t^{-(\gamma+1)}) \IND{(0,\infty)}(t), \]
This choice yields a solution obtained earlier in \cite{athreya} using an extreme value method that prompted the subscript `A' to denote this ground
state. In \autoref{lastcase} we will use \eqref{Frechet} as a ground state to obtain a one parameter family of solutions of \eqref{mildalpha} corresponding to $u_0=0$.

With the tools developed in \autoref{Sec3} and \autoref{PicardIteration} we can now proceed to the proof of one of the main results of the paper.

\section{Proof of \autoref{main}}
\label{Sec6}

The purpose of this section is to combine the stochastic Picard iteration method introduced in \autoref{PicardIteration} 
with the stochastic transformation approach described in \autoref{Sec3}
to construct a family 
of distinct global solutions to \eqref{diffalpha-ricc0} for any $\alpha>1$ and a range of initial data $u_0$. Note that solutions to \eqref{diffalpha-ricc0} blow up 
in finite time for sufficiently large initial data \cite{alphariccati}. 

The details of the proof of \autoref{main} will be given in \autoref{firstcase}-\ref{generalinitialdata}.  We first consider 
the case of $u_0=1$ in two separate regimes: {$\alpha\in(1,2]$ and $\alpha>2$}. We then extend the results to other initial data $u_0\in D_\alpha$.

\subsection{The case \texorpdfstring{$u_0=1$}{} and \texorpdfstring{$1<\alpha\le 2$}{}}\label{firstcase}
 
Recall $\gamma$ is defined in \eqref{gma}, so this range of values of $\alpha$ corresponds to $\gamma \ge 1$.  In this range, different ground states can be used directly to obtain different fixed points of $F[u],$ where 
\[F[u](t)=e^{-t}+\int_0^t e^{-s}u^2(\alpha(t-s))ds.\]
Thus, each fixed point satisfies 
\eqref{mildalpha} with $u_0=1$. The choice of ground state gives the claimed asymptotic behavior of these solutions.
\begin{prop}\label{72191}
Let $1<\alpha\le 2$. Then there exists $M_\alpha>0$ such that for every $M>M_\alpha$, there exists $\delta_M>0$ such that for any $\delta\in(0,\delta_M)$, the function
\begin{equation}\label{eq:72191}\rho_{M,\delta}(t)=\left\{ \begin{array}{*{35}{l}}
   1 & \text{if} & t\le M ,  \\
   1-\delta{{t}^{-\gamma}} & \text{if} & t>M   \\
\end{array} \right.\end{equation}
satisfies $F[\rho_{M,\delta}]\le \rho_{M,\delta}$.
\end{prop}
\begin{proof}
Note that $\rho_{M,\delta}(t)\in(0,1]$ for all $t\ge 0$ if $\delta<M^{\gamma}$. To simplify the notation in this proof, we will drop the subscripts of $\rho_{M,\delta}$. Let 
\[\tilde\rho(t)=1-\rho(t)=\left\{ \begin{array}{*{35}{l}}
   0 & \text{if} & t\le M ,  \\
   \delta{{t}^{-\gamma}} & \text{if} & t>M   \\
\end{array} \right.\] 
and \[G[\tilde\rho](t)=1-F[1-\tilde{\rho}](t)=\int_{0}^{t}{{{e}^{-(t-s)}}(2\tilde\rho(\alpha s)-\tilde\rho{{(\alpha s)}^{2}})ds}.\]
It suffices to show that $G[\tilde{\rho}] \ge \tilde{\rho}$. For $t\ge M$,
\begin{eqnarray*}G[\tilde\rho](t)=\int_{0}^{t}{{{e}^{-t+s}}(2\tilde\rho(\alpha s)-\tilde\rho{{(\alpha s)}^{2}})ds}
&=&\int_{M /\alpha }^{t}{{{e}^{-t+s}}(2\delta{{\alpha }^{-\gamma}}{{s}^{-\gamma}}
-{{\delta}^{2}}{{\alpha }^{-2\gamma}}{{s}^{-2\gamma}})ds}\\
&=&2\delta {{\alpha }^{-\gamma}}{{e}^{-t}}\int_{M /\alpha }^{t}{{{e}^{s}}{{s}^{-\gamma}}
\left(1-\frac{\delta{{\alpha }^{-\gamma}}}{2}{{s}^{-\gamma}}\right)ds}\\
&=&\delta {{e}^{-t}}\int_{M /\alpha }^{t}{{{e}^{s}}{{s}^{-\gamma}}\left(1-\frac{\delta}{4} {{s}^{-\gamma}}\right)ds}.
\end{eqnarray*}
The inequality $G[\tilde\rho]\ge \tilde\rho$ will hold for all $t\ge 0$ provided that 
\[\int_{M /\alpha }^{t}{{{e}^{s}}{{s}^{-\gamma}}\left(1-\frac{\delta}{4} {{s}^{-\gamma}}\right)ds}\ge 
{{e}^{t}}{{t}^{-\gamma}}\qquad \forall\,t\ge M.\]
This inequality is equivalent to
\begin{equation}\label{eq:72192} \frac{\delta}{4} \le f(t):=
\frac{\int_{M/\alpha }^{t}{{{e}^{s}}{{s}^{-\gamma }}ds}-{{e}^{t}}{{t}^{-\gamma }}}{\int_{M/\alpha }^{t}{{{e}^{s}}{{s}^{-2\gamma }}ds}}
\qquad \forall\, t\ge M.\end{equation}
We will show that $f(t)>0$ for all $t\ge M$ for some sufficiently large $M>0$. Note that $f(t)$ has the same sign as
\[f_1(t)=\int_{M/\alpha}^te^ss^{-\gamma}ds-e^tt^{-\gamma}.\]
Since $f_1'(t)=\gamma e^tt^{-\gamma-1}>0$, one sees that $f_1(t)\ge f_1(M)$ for all $t\ge M$. 
On the other hand, $f_1(M)=f_2(M)$ where
\[f_2(x)=\int_{x/\alpha}^xe^ss^{-\gamma}ds-e^xx^{-\gamma}.\]
One has $f_2'(x)=\gamma e^xx^{-\gamma-1}-\frac{1}{\alpha}e^{x/\alpha}(x/\alpha)^{-\gamma}$. Note that 
$f'_2(x)>1$ for sufficiently large $x$ and hence, $\lim_{x\to\infty}f_2(x)=\infty$. Hence, there exists $M_\alpha>0$ 
such that $f_2(M)>0$ for all $M>M_\alpha$. Therefore, $f(t)>0$ for all $t\ge M>M_\alpha$. 

For a fixed $M>M_\alpha$, we will show that $\inf_{t\ge M}f(t)>0$. By L'Hospital Rule,
\[\underset{t\to \infty }{\mathop{\lim }}\,f(t)=\underset{t\to \infty }{\mathop{\lim }}\,
\frac{\gamma {{e}^{t}}{{t}^{-\gamma -1}}}{{{e}^{t}}{{t}^{-2\gamma }}}=\left\{ \begin{matrix}
   1 & \text{if} & \gamma =1  \\
   \infty  & \text{if} & \gamma >1.  \\
\end{matrix} \right. \]
Together with the continuity of $f$ on $(0,\infty)$ and the fact that $f(t)>0$ for all  $t\in [M,\infty)$, we conclude that $\inf_{t\ge M}f(t)>0$. The proof is completed by noticing that \eqref{eq:72192} 
is satisfied by choosing $\delta< \delta_M=\min\{M^\gamma,\,4\inf_{t\ge M}f(t)\}$.
\end{proof}

\begin{prop}\label{martingale-a}
Let $M,\delta>0$ and $\rho=\rho_{M,\delta}$ be defined as in \autoref{72191}. Consider the stochastic Picard 
iterations $X_n(t)=X_{M,\delta,n}(t)$ with the ground state $X_0(t)=X_{M,\delta,0}(t)=\rho_{M,\delta}(t)$ and 
the initial state $u_0=1$. Then :
\begin{enumerate}[(i)]
\item For all $n\in\mathbb{N},$ $t\ge0$ and $\Theta_v$ defined by \eqref{agedef},
\begin{equation}\label{eq6211} X_{n}(t)
\xlongequal{\rm a.s.} \prod_{v\in \T,\ |v|=n-1}\rho^2_{M,\delta}(\alpha^n(t-\Theta_v)),
\qquad\mbox{a.s.}\end{equation}
\item For each $t>0$, The sequence $\{ X_n(t)\}= \{X_{M,\delta,n}(t)\}$ is a non-negative supermartingale 
with respect to the filtration $\mathcal{F}_n=\sigma(T_v:\,|v|\le n-1)$.
\end{enumerate}
\end{prop}

\begin{proof}
By definition, the Picard iteration is
\begin{equation}\label{rho-iterations} 
X_0(t)=
\rho(t),\quad {{ X}_{n}}(t)
\stackrel{\rm a.s.}{=} \left\{ \begin{array}{*{35}{l}}
   1 & \text{if} & T_\theta\ge t,  \\
    X_{n-1}^{(1)}(\alpha (t-T_\theta))\, X_{n-1}^{(2)}(\alpha (t-T_\theta)) & \text{if} & T_\theta<t.  
\end{array} \right., n\ge1,
\end{equation}
where $ X_{n-1}^{(1)}$ and $ X_{n-1}^{(2)}$ are conditionally on $T_\theta$ i.i.d.\ copies of $ X_{n-1}$.

The formula \eqref{eq6211} from part {\em (i)} follows by induction from the stochastic iterations. Indeed, for 
$n=1$, since $\rho(t)=1$ for $t\le0$  we have
\[
X_1(t)\stackrel{\eqref{rho-iterations}}{=}\IND{[T_\theta\ge t]}+\rho^2(\alpha(t-T_\theta))
\IND{[T_\theta<t]}=\rho^2(\alpha(t-T_\theta)),
\]
So \eqref{eq6211} holds. The inductive step follows similarly, once we observe that the product in the 
right-hand side of \eqref{eq6211} is 1 if $t\le 0$.
To prove {\em (ii)}, we will show by induction on $n$ that $\EXP[ X_{n+1}(t)|\mathcal{F}_n]\le X_n(t)$. 
For $n=1$, $\mathcal{F}_1=\sigma(T_\theta)$ and, as noted above, $X_1(t)= \rho^2(\alpha(t-T_\theta))$. 
Note that $\mathcal{F}_0$ is the trivial $\sigma$-field. Thus, since $\rho(t)=1$ for $t\le 0$, using \autoref{72191}, we obtain
\begin{eqnarray}
\EXP(X_1(t)|\mathcal{F}_0)=\EXP(X_1(t))&=&\int\limits_0^\infty e^{-s} \rho^2(\alpha(t-s))\, ds \nonumber\\
& = & e^{-t}+\int\limits_0^t e^{-s} \rho^2(\alpha(t-s))\, ds \nonumber\\
&\le & \rho(t)=X_0(t) \quad \mbox{for all $t\ge 0$.} \label{n=1case}
\end{eqnarray}

For $n=2$, using \eqref{eq6211}, we write
\begin{eqnarray*}
 X_2(t)=\prod_{|v|=1}\rho^2(\alpha^2(t-\Theta_v))=\prod_{|v|=1}\rho^2(\alpha(\tau-T_v))
 = X_1^{(1)}(\tau) X_1^{(2)}(\tau)
\end{eqnarray*}
where $\tau=\alpha(t-T_\theta)$ and
\begin{eqnarray*}
 X_1^{(1)}(\tau)=\1_{[T_1\ge \tau]}+\rho^2(\alpha(\tau-T_1))\1_{[T_1< \tau]} = \rho^2(\alpha(\tau-T_1)), \\
 X_1^{(2)}(\tau)=\1_{[T_2\ge \tau]}+\rho^2(\alpha(\tau-T_2))\1_{[T_2< \tau]} = \rho^2(\alpha(\tau-T_2)). 
\end{eqnarray*}
Note that, conditionally on $T_\theta$, $X_1^{(1)}(\tau)$ and  $X_1^{(2)}(\tau)$ are i.i.d. and have the same 
distribution as $ X_1(\tau)$. Therefore, using the substitution property for conditional expectations 
together with \eqref{n=1case}, we have:
\begin{eqnarray}
\EXP( X_2(t)|T_\theta)&=&
\EXP( X_1^{(1)}(\tau)X_1^{(2)}(\tau)|\tau)=\EXP( X_1^{(1)}(\tau)|\tau)\cdot\EXP(X_1^{(2)}(\tau)|\tau)\nonumber\\
&=&\EXP( X_1(\tau)|\tau)^2\le \rho^2(\tau)=\rho^2(\alpha(t-T_\theta))=X_1(t).
\label{subpropcalc1}
\end{eqnarray}

Now suppose $\EXP( X_n(t)|\F_{n-1})\le X_{n-1}(t)$ for some $n\ge 2$. Using \eqref{rho-iterations} together with 
the fact that $X_n(t)=1$ for $t\le0$ (as follows from \eqref{eq6211}), we have 
\[
X_{n+1}(t)=X_{n}^{(1)}(\tau)\, X_{n}^{(2)}(\tau), 
\]
where, as before, $\tau=\alpha(t-T_\theta)$. Recall that, conditionally on $T_\theta$, $X_{n}^{(1)}(\tau)$ and  
$X_{n}^{(2)}(\tau)$ are independent and distributed as $X_n(\tau)$. For $k\in\{1,2\}$, denote 
$\F^{(k)}_{n-1}=\sigma(T_{kv}:|v|=n-2)$. Because $\F_n=\sigma(T_\theta)\vee \F^{(1)}_{n-1} \vee 
\F^{(2)}_{n-1}$ and that $\sigma(T_\theta)$, $\F^{(1)}_{n-1}$ and $\F^{(2)}_{n-1}$ are independent, 
we get by organizing random variables according to 
$$(T_v:|v|\le n-1) = \left(\,T_\theta, \left(\, (T_{1v}:|v|\le n-2),(T_{2v}:|v|\le n-2)\,\right)\,\right),$$
and applying the substitution property for conditional expectations in two directions, as in (\ref{subpropcalc1}), followed by the 
induction hypothesis,
\begin{eqnarray}\label{subpropcalc2}
&&\EXP({{ X}_{n+1}}(t)|{{\F}_{n}})\nonumber\\
&=&\IND{[T_\theta\ge t]}+\EXP(X_{n}^{(1)}(\tau) X_{n}^{(2)}(\tau){\IND{[{{T}_{\theta}}<t]}}|{{\F}_{n}})\nonumber\\
&=&\1_{[T_\theta\ge t]}+\EXP({\IND{[{{T}_{\theta}}<t]}} X_{n}^{(1)}(\tau)|\sigma({T}_{\theta})\vee\F_{n-1}^{(1)})\cdot
\EXP({\IND{[{{T}_{\theta}}<t]}} X_{n}^{(2)}(\tau)|\sigma({T}_{\theta})\vee\F_{n-1}^{(2)})\nonumber\\
&\le&\IND{[T_\theta\ge t]}+ X_{n-1}^{(1)}(\tau) X_{n-1}^{(2)}(\tau){\IND{[{{T}_{\theta}}<t]}} = {{ X}_{n}}(t).
\end{eqnarray}
\end{proof}

\begin{prop}\label{non-indicator} 
Assume $\alpha\in(1,2]$. Let $M,\delta>0$ and $\rho_{M,\delta}$ be defined as in \autoref{72191} and 
$\{ X_{M,\delta,n}(t)\}_{n\ge 1}$ be the stochastic process defined by \eqref{eq6211}. 
\begin{enumerate}[(i)]
\item As $n\to\infty$, $ X_{M,\delta,n}(t)$ converges a.s.\ and the limit process $ X_{M,\delta}(t)
=\lim_{n\to\infty} X_{M,\delta,n}(t)$ satisfies \eqref{alphrecursion} with $u_0=1$. Moreover, $0< X_{M,\delta}(t)<1$ 
on the event $[L<t]$, where $L$ denotes the longest path, see \eqref{longdef}.
\item The function $u_{M,\delta}(t)=\EXP X_{M,\delta}(t)$ solves \eqref{diffalpha-ricc0} with $u_0=1$. Moreover, 
\begin{eqnarray*}
\lim_{t\to\infty}\frac{1-u_{M,\delta}(t)}{t^{-\gamma}}&=&{\delta}\,.
\end{eqnarray*}
\end{enumerate}
\end{prop}
\begin{proof}
To simplify the notations in the proof, we will drop the subscripts $M$ and $\delta$ and will only keep the subscript $n$.

By \autoref{martingale-a} (ii) and Doob's Martingale Convergence Theorem for positive supermartingales, 
the sequence $\{X_n(t)\}= \{X_{M,\delta,n}(t)\}_{n\ge 1}$ converges a.s.. Denote the limit by $ X(t)= X_{M,\delta}(t)$. 
Since for all $n\in\mathbb{N}$ and $t>0$, $X_n(t)\in[0,1]$ a.s., we have $\EXP X(t)\in[0,1]$. By \autoref{X_n-conv}, 
$X(t)$ satisfies \eqref{alphrecursion} with $u_0=1$. Thus, $u(t)=\EXP X(t)$ satisfies \eqref{diffalpha-ricc0}.

On the event $[L<t]$, one has
\[ X_n(t)=\prod_{|v|=n-1}\rho^2(\alpha^n(t-\Theta_v)),\ \ \ \ \forall\,n\ge 1.\]
Since $t-\Theta_v\ge t-L>0$, we have $\alpha^n(t-\Theta_v)>M$ for sufficiently large $n$ independent on $v\in\T$. 
Here, $M$ is the number in \autoref{72191}. Thus, for sufficiently large $n$,
\[ X_n(t)\IND{[L<t]}=\prod_{|v|=n-1}(1-\delta (\alpha^n(t-\Theta_v))^{-\gamma})^2\IND{[L<t]}
=\prod_{|v|=n-1}\left(1-\delta \frac{(t-\Theta_v)^{-\gamma}}{2^n}\right)^2\IND{[L<t]}\]
Denote $S_n=\min_{|v|=n}\Theta_v$ and $L_n=\max_{|v|=n}\Theta_v$. Note that 
$S_{n-1}\le\Theta_v\le L_{n-1}$ for $|v|=n-1$. Thus,
\[\prod\limits_{|v|=n-1}{\left( 1-\delta \frac{{{(t-{{L}_{n}})}^{-\gamma }}}{{{2}^{n}}} \right)^2}\IND{[L<t]}
\le  X_n(t)\IND{[L<t]}\le \prod\limits_{|v|=n-1}{\left( 1-\delta \frac{{{(t-{{S}_{n}})}^{-\gamma }}}{{{2}^{n}}} \right)^2}\IND{[L<t]}\]
In other words,
\begin{equation}\label{Xnestimate}\left( 1-\delta \frac{{{(t-{{L}_{n}})}^{-\gamma }}}{{{2}^{n}}} \right)^{2\cdot2^{n-1}}\IND{[L<t]}
\le  X_n(t)\IND{[L<t]}\le \left( 1-\delta \frac{{{(t-{{S}_{n}})}^{-\gamma }}}{{{2}^{n}}} \right)^{2\cdot2^{n-1}}\IND{[L<t]}.\end{equation}
Letting $n\to\infty$, one gets
\begin{equation}\label{sandwich}
e^{-\delta(t-L)^{-\gamma}}\IND{[L<t]}\le X(t)\IND{[L<t]}\le e^{-\delta(t-S)^{-\gamma}}\IND{[L<t]}\,.
\end{equation}
Therefore, $0< X(t)<1$ on the event $[L<t]$.\\

To establish the limit in part {\em (ii)}, we estimate
\[\EXP[X(t)\IND{[L<t]}]\le\EXP X(t)=\EXP[ X(t)\IND{[L<t]}]+\EXP[X(t)\IND{[L\ge t]}]\le\EXP[X(t)\IND{[L<t]}]
+\EXP\IND{[L\ge t]}.\]
Together with \eqref{sandwich}, we have
\[\EXP\left[e^{-\delta(t-L)^{-\gamma}}\IND{[L<\epsilon t]}\right]\le u(t)\le \EXP\left[e^{-\delta(t-S)^{-\gamma}}
\IND{[L<t]}\right]+\EXP\IND{[L\ge t]}\]
for any constant $\epsilon\in(0,1)$. Hence,
\[\EXP\left[(1-e^{-\delta(t-S)^{-\gamma}})\IND{[L<t]}\right]\le 1-u(t)\le \EXP\left[(1-e^{-\delta(t-L)^{-\gamma}})
\IND{[L<\epsilon t]}\right]+\EXP\IND{[L\ge \epsilon t]}.\]
By \autoref{expl_asymptotics}, one has $\EXP\IND{[L\ge \epsilon t]}\le Ce^{-\epsilon t}$ for all $t>0$. 
Dividing both sides of the above inequalities by $t^{-\gamma}$, we have
\begin{equation}\label{rateofconv}
\EXP\left[\frac{1-e^{-\delta(t-S)^{-\gamma}}}{t^{-\gamma}}\IND{[L<t]}\right]\le \frac{1-u(t)}{t^{-\gamma}}\le 
\EXP\left[\frac{1-e^{-\delta(t-L)^{-\gamma}}}{t^{-\gamma}}\IND{[L<\epsilon t]}\right]+Ct^\gamma e^{-\epsilon t}
\end{equation}
Note that almost surely
\[\underset{t\to \infty }{\mathop{\lim }}\,\frac{1-{{e}^{-{\delta}{{(t-S)}^{-\gamma }}}}}{{{t}^{-\gamma }}}=\delta \]
Also, $\lim_{t\to\infty}\IND{[L<t]}=\IND{[L<\infty]}=1$ a.s. due to hyperexplosion \cite{alphariccati}. By Fatou's Lemma,
\[\underset{t\to \infty }{\mathop{\liminf }}\,\frac{1-u(t)}{{{t}^{-\gamma }}}\ge 
{\EXP}\left[ \underset{t\to \infty }{\mathop{\lim }}\,\frac{1-{{e}^{-{\delta}{{(t-S)}^{-\gamma }}}}}{{{t}^{-\gamma }}}{\IND{[L<t]}} 
\right]=\delta .\]
By the inequality $1-e^{-x}\le x$, one has
\begin{eqnarray*}\textup{RHS}\eqref{rateofconv}\le
\EXP\left[\frac{\delta(t-L)^{-\gamma}}{t^{-\gamma}}\1_{[L<\epsilon t]}\right]+Ct^{\gamma}e^{-\epsilon t}
&\le& 
\EXP\left[\frac{\delta(t-\epsilon t)^{-\gamma}}{t^{-\gamma}}\1_{[L<\epsilon t]}\right]+Ct^{\gamma}e^{-\epsilon t}\\
&\le& \delta(1-\epsilon)^{-\gamma}+Ct^{\gamma}e^{-\epsilon t}
\end{eqnarray*}
Thus,
\[\underset{t\to \infty }{\mathop{\limsup }}\,\frac{1-u(t)}{{{t}^{-\gamma }}}\le \delta(1-\epsilon)^{-\gamma}.\]
Because this inequality is true for all $\epsilon\in(0,1)$, one has
\[\underset{t\to \infty }{\mathop{\limsup }}\,\frac{1-u(t)}{{{t}^{-\gamma }}}\le \delta,\]
which completes the proof.
\end{proof}

The following theorem completes the proof of \autoref{main} in the case $\alpha\in(1,2]$, $u_0=1$.  As a by-product of the stochastic transformations involved, it also includes a comparison among the different solutions so obtained.

\begin{thm}\label{main1}
Let $M,\delta>0$ and $\rho_{M,\delta}$ be defined as in \autoref{72191}, and $ X_{M,\delta}(t)$ be the process 
defined in \autoref{non-indicator}. Then, for any $\lambda\ge 0$, the process $X_{M,\delta,\lambda}(t)
=(X_{M,\delta}(t))^{\lambda/\delta}$ is a solution process satisfying \eqref{alphrecursion} with $u_0=1$, 
and the function $u_{M,\delta,\lambda}(t)=\EXP X_{M,\delta,\lambda}(t)$ solves the problem 
\eqref{diffalpha-ricc0} with $u_0=1$. Moreover: 
\begin{enumerate}[(i)]
\item \begin{equation}\label{conv-rate}
\lim_{t\to\infty}\frac{1-u_{M,\delta,\lambda}(t)}{t^{-\gamma}}={\lambda}.\end{equation}
\item For any $t>0$,
\begin{eqnarray*}
u_{M,\delta,\lambda}(t)&>&u_{M,\delta,\lambda'}(t)\ \ \ \textup{if}\ \ 0\le\lambda<\lambda',\\
u_{M,\delta,\lambda}(t)&\le&u_{M',\delta,\lambda}(t)\ \ \ \textup{if}\ \ M< M',\ 0<\delta<\min\{\delta_M,\delta_{M'}\}\\
u_{M,\delta,\lambda}(t)&\ge&u_{M,\delta',\lambda}(t)\ \ \ \textup{if}\ \ 0<\delta<\delta'<\delta_M.
\end{eqnarray*}
\end{enumerate}
\end{thm}
\begin{proof}
The fact that $X_{M,\delta,\lambda}(t)$ is a solution process satisfying \eqref{alphrecursion} with $u_0=1$, follows from 
raising to power $\lambda/\delta$ both sides of \eqref{alphrecursion} with $X=X_{M,\delta}$ and $u_0=1$. Thus, 
$u_{M,\delta,\lambda}(t)=\EXP  X_{M,\delta,\lambda}(t)$ satisfies \eqref{diffalpha-ricc0} with $u_0=1$, since, as it will 
be shown below, the expectation is finite.

To prove {\em (i)}, raise the equation \eqref{sandwich} to power $\lambda/\delta$ to obtain a.s.
\begin{equation}
e^{-\lambda(t-L)^{-\gamma}}\IND{[L<t]}\le (X_{M\delta}(t))^{\lambda/\delta}\IND{[L<t]}
\le e^{-\lambda(t-S)^{-\gamma}}\IND{[L<t]}\,.
\end{equation}
From here, one can follow the same lines of the proof of \autoref{non-indicator}, 
part  {\em (ii)}, to show that $u_{M,\delta,\lambda}(t)$ is finite and
\[\lim_{t\to\infty}\frac{1-u_{M,\delta,\lambda}(t)}{t^{-\gamma}}={\lambda}.\]
To prove {\em (ii)}, suppose $0\le \lambda<\lambda'$. Because $0< X_{M,\delta}(t)<1$ on the event 
$[L<t]$, $ X_{M,\delta}^{\lambda'/\delta}(t)< X_{M,\delta}^{\lambda/\delta}(t)$ on this event. Since $[L<t]$ is 
not a null event for any $t>0$, one has $u_{M,\delta,\lambda'}(t)<u_{M,\delta,\lambda}(t)$.

Next, suppose $M<M'$ and $0<\delta<\min\{\delta_M,\delta_{M'}\}$. From the definition of $\rho_{M,\delta}$ in \autoref{72191}, 
it is clear that $\rho_{M,\delta}(t)\le \rho_{M',\delta}(t)$, which leads to $ X_{M,\delta,n}(t)\le  X_{M',\delta,n}(t)$ for all $n$. 
Therefore, $u_{M,\delta,\lambda}(t)\le u_{M',\delta,\lambda}(t)$.

Next, suppose $0<\delta<\delta'<\delta_M$. Denote $\kappa=\delta'/\delta>1$. Note that for $t>M$,
\[\rho_{M,\delta}^\kappa(t)=(1-\delta t^{-\gamma})^\kappa>1-\kappa\delta t^{-\gamma}=\rho_{M,\delta'}(t).\]
Thus, $ X_{M,\delta,n}^{\kappa}(t)\ge X_{M,\delta',n}(t)$ for all $n$. Raising both sides to power $\lambda/\delta'$, 
one gets $ X_{M,\delta,n}^{\lambda/\delta}(t)\ge X_{M,\delta',n}^{\lambda/\delta'}(t)$. Hence, 
$u_{M,\delta,\lambda}(t)\ge u_{M,\delta',\lambda}(t)$.
\end{proof}
\begin{remark}
In \autoref{main1} (ii): the first inequality is strict, so one gets many different solutions. The other two inequalities are not strict. If one of them turns out to be a strict inequality then one would have infinitely many solutions for a fixed $\lambda>0$ by varying $M$ or $\delta$.
\end{remark}

\subsection{The case \texorpdfstring{$u_0=1$}{} and \texorpdfstring{$\alpha>2$}{}}
\label{alpha_bgr_2_case}

For this range of values of $\alpha$, the exponential transformation \eqref{trans1} between solution processes of the pantograph equation and the $\alpha$-Riccati equation is used. A suitable choice of a ground state $\sX_0$ is made for the stochastic Picard iteration \eqref{pant_bin_iter} so that, with $\gamma$ given by \eqref{gma}, the asymptotic behavior, 
$1-u(t)\sim t^{-\gamma}$ as $t\to\infty$ is obtained.
For $\alpha>2$, note that $\gamma < 1$ and thus $1-u\not\in L^1$.

We start with the following result regarding the exponential stochastic transformation  \eqref{trans1}.
\begin{prop}\label{transf_pant_to_alph}
Suppose $\sX(t)$ is a nonnegative binary-tree solution process of \eqref{markovsspantographeq}, i.e.\ satisfying
\begin{equation}\label{1193}
\sX(t)\stackrel{\rm a.s.}{=} \left\{ \begin{array}{*{35}{l}}
   0 & \text{if} & T_\theta\ge t  \\
   \sX^{(1)}(\alpha (t-T_\theta))+\sX^{(2)}(\alpha (t-T_\theta)) & \text{if} & T_\theta<t  \\
\end{array} \right.
\end{equation}
\begin{enumerate}[(i)]
\item $w(t)=\EXP\sX(t)$, if finite for all $t$, satisfies \eqref{markovsspantographeq}.
\item For any $\lambda\ge 0$, $X(t)=e^{-\lambda\sX(t)}$ satisfies \eqref{alphrecursion} and 
$u(t)=\EXP\left[e^{-\lambda\sX(t)}\right]$ satisfies \eqref{diffalpha-ricc0}.
\end{enumerate}
\end{prop}
\begin{proof}
(i) By conditioning on $T$ in \eqref{1193}, we get 
\[w(t)=\int_0^te^{-s}2w(\alpha(t-s))ds\]
which leads to \eqref{markovsspantographeq}.

(ii) First, apply the function $e^{-\lambda\cdot}$ to both sides of \eqref{1193} to 
show $X(t)=e^{-\lambda\sX(t)}$ satisfies \eqref{alphrecursion}. Note that $u$ is always 
well-defined because $e^{-\lambda\sX(t)}\in[0,1]$. By conditioning on $T$ in 
\eqref{alphrecursion}, we get
\[u(t)=e^{-t}+\int_0^te^{-s}u^2(\alpha(t-s))ds\]
which leads to \eqref{diffalpha-ricc0}.
\end{proof}

One can observe from \autoref{transf_pant_to_alph}  that if \eqref{1193} has a solution $\sX(t)\ge 0$, not identically zero, 
then $u_\lambda(t)=\EXP[ e^{-\lambda\sX(t)}]$, $\lambda\ge 0$, is an infinite family of solutions to \eqref{diffalpha-ricc0} 
corresponding to $u_0=1$. Thus, our next goal is to construct a solution process $\sX(t)\ge 0$ of \eqref{1193} that is 
not identically zero. The key idea is to use the expected value of the unary solution process given by \autoref{prop11112} 
as the ground state in the stochastic Picard iterations for \eqref{1193}.

\begin{prop}\label{martingalep}
Let $\eta(t)$ be as defined in \autoref{prop11112} with $a=2$. Define $\eta(t)=0$ if $t\le 0$. On the full binary tree $\T$ and with $\Theta_v$ as in \eqref{agedef}, 
define a.s.
\begin{equation}\label{sX_fla} 
\sX_n(t)=\sum_{|v|=n-1}2\eta(\alpha^n(t-\Theta_v)),\ \ \ \ \forall\,n\ge 1.
\end{equation}
 Then
\begin{enumerate}[(i)]
\item The sequence $\{\sX_n(t)\}$ satisfies the stochastic Picard iterations \eqref{pant_bin_iter} for the binary pantograph 
process  with ground state $\sX_0(t)=\eta(t)$, corresponding to $a=2$ and $w_0=0$, i.e.
\begin{equation}\label{itrations_bin_pant}
{{\sX}_{n}}(t)
\stackrel{\rm a.s.}{=} 
\left\{ \begin{array}{*{35}{l}}
   0 & \text{if} & T_\theta\ge t,  \\
   \mathscr{\sX}_{n-1}^{(1)}(\alpha (t-T_\theta))+\sX_{n-1}^{(2)}(\alpha (t-T_\theta)) & \text{if} & T_\theta<t.  \\
\end{array} \right.
\end{equation}
\item For each $t>0$, $\{\sX_n(t)\}$ is a martingale with respect to the filtration $\F_n=\sigma(T_v:\,|v|\le n-1)$.
\item The limit $\sX(t)=\lim_{n\to\infty}\mathscr{X}_n(t)\in[0,\infty)$ a.s. Moreover, $\EXP\sX(t)=\eta(t)$ and
$\EXP[\sX^\delta(t)]\le \eta_\delta(t)$ for any $\delta\in(1,1/\gamma)$, with $\gamma=\log_\alpha 2<1$. Here, $\eta_\delta$ 
is the function given in \autoref{exampleone}.
\item For every $t>0$, the event $[\sX(t)>0]\cap[L<t]$ has a positive measure.
\end{enumerate}
\end{prop}

\begin{proof}
(i) Note that if the ground state $X_0(t)=0$ for $t\le 0$, the iterative formula \eqref{itrations_bin_pant} can be re-written as
\begin{equation}\label{simpl_pant_iter}
\sX_n(t)
\stackrel{\rm a.s.}{=} \sX_{n-1}^{(1)}(\alpha(t-T_\theta))+\sX_{n-1}^{(2)}(\alpha(t-T_\theta)),
\end{equation}
since by induction $\sX_n(t)=0$ on $t\in(-\infty,0]$ for all $n$. In case $\sX_0(t)=\eta(t)$, the formula for $\sX_n$ given 
by \eqref{sX_fla} follows from \eqref{simpl_pant_iter} directly from the definition of $\sX_n(t)$ also by induction. Thus,  
$\sX_n$ satisfies \eqref{itrations_bin_pant}.

(ii) We will show by induction on $n$ that $\EXP[\sX_{n+1}(t)|\mathcal{F}_n]=\sX_n(t)$ following the same approach 
as in the proof of \autoref{martingale-a}, part {\em (ii)}. Namely, for $n=1$, 
$\F_0$ is the trivial $\sigma$-field and $\sX_1(t)=2\eta(\alpha(t-T_\theta))$, so 
\begin{equation}\label{step_1}
\EXP[\sX_1\,|\,\F_0]=\EXP\sX_1=\int_0^t e^{-s} 2 \eta(\alpha(t-s))\, ds=\eta(t)=X_0(t).
\end{equation}
since  by \autoref{prop11112}, $\eta$ satisfies \eqref{markovsspantographeq}.

In the case $n=2$, $\mathcal{F}_1=\sigma(T_\theta)$, and thus
\begin{eqnarray*}
\EXP[{{\sX}_{2}}\,|\,{\mathcal{F}_{1}}]&=&
\EXP[2\eta ({{\alpha }^{2}}(t-{{\Theta }_{1}}))+2\eta ({{\alpha }^{2}}(t-{{\Theta }_{2}}))\,|\,{{T}_{\theta}}]\\
&=&\EXP[2\eta (\alpha (\tau -{{T}_{1}}))+2\eta (\alpha (\tau -{{T}_{2}}))\,|\,{{T}_{\theta}}]
\end{eqnarray*}
where $\tau=\alpha(t-T_\theta)$. Thus, 
\begin{eqnarray*}
\EXP[{{\sX}_{2}}|{\F_{1}}]&=&\EXP[2\eta (\alpha (\tau -{{T}_{1}})+2\eta (\alpha (\tau -{{T}_{2}}))\,|\,{{T}_{\theta}}]\\
&=&
\EXP[\sX_1^{(1)}(\tau)\,|\, \tau)+\EXP(\sX_1^{(2)}(\tau)\,|\, \tau]= 2\,\EXP[\sX_1(\tau)\,|\, \tau]=2\sX_0(\tau)=\sX_1(t),
\end{eqnarray*}
where \eqref{step_1} and the substitution property for conditional probability was used in the second-to-last equality.

Now suppose $\EXP[\sX_n|\F_{n-1}]=\sX_{n-1}$ for some $n\ge 2$. We have
\[\EXP[{{\sX}_{n+1}}|{{\F}_{n}}]=\EXP[\sX_{n}^{(1)}(\alpha (t-{{T}_{\theta}}))+\sX_{n}^{(2)}(\alpha (t-{{T}_{\theta}}))\,|\,{{\F}_{n}}].\]
Because $\F_n=\sigma(T_0)\vee \F^{(1)}_{n-1} \vee \F^{(2)}_{n-1}$ and that $\F^{(1)}_{n-1}$ and $\F^{(2)}_{n-1}$ 
are independent, using substitution property, we get
\begin{eqnarray*}
\EXP[{{\sX}_{n+1}}|{{\F}_{n}}]
&=&\EXP[\sX_{n}^{(1)}(\alpha (t-{{T}_{\theta}}))\,|\,{{T}_{\theta}},F_{n-1}^{(1)}]+
\EXP[\sX_{n}^{(2)}(\alpha (t-{{T}_{\theta}}))\,|\,{{T}_{\theta}},\F_{n-1}^{(2)}]\\
&=&\sX_{n-1}^{(1)}(\alpha (t-{{T}_{\theta}}))+\sX_{n-1}^{(2)}(\alpha (t-{{T}_{\theta}}))
={{\sX}_{n}}(t).
\end{eqnarray*}

(iii) The fact that $\sX_n(t)\ge0$ and the fact that $\{\sX_n(t)\}_{n\ge 1}$ is a martingale imply that for 
any $t>0$, $\sX_n(t)$ is convergent a.s.\ to some random variable $\sX(t)$. By \autoref{X_n-conv}, $\sX$ is a solution process satisfying 
\[{{\sX}}(t){\xlongequal{}}\,\left\{ \begin{array}{*{35}{l}}
   0 & \text{if} & T_\theta\ge t,  \\
   \sX^{(1)}(\alpha (t-T_\theta))+\sX^{(2)}(\alpha (t-T_\theta)) & \text{if} & T_\theta<t.  \\
\end{array} \right.\]
Also,
\[\EXP{{\sX}_{1}}(t)=\EXP[2\eta (\alpha (t-{{T}_{\theta}}))]=\int_{0}^{t}2{\eta (\alpha (t-s)){{e}^{-s}}ds}=\eta (t).\]
By the martingale property, $\EXP\sX_n(t)=\eta(t)$ for all $n\ge 1$. To show that $\EXP\sX(t)=\eta(t)$, it suffices 
to show that for each $t>0$, the sequence $\EXP[\sX_n(t)^\delta]$ is bounded from above for some $\delta>1$.

Fix $\delta\in(1,1/\gamma)$. Let $\eta_\delta$ be the function from \autoref{exampleone}, i.e. 
$\eta_\delta(t)=\EXP[\tsX_*^{\delta}(t)]$ with $\tsX_*=(t-\tS)^{-\gamma}\IND{[\tS<t]}$. Recall that $\eta(t)=\EXP\tsX_*(t)$.
By Jensen's inequality,
\[\eta_\delta (t)\ge ( \EXP\tsX_*(t))^{\delta} =\eta (t)^{\delta }.\]
We will show by induction on $n\ge 1$ that $\EXP[\sX_n(t)^\delta]\le\eta_\delta(t)$. For $n=1$,
\begin{eqnarray*}
\EXP[{{\sX}_{1}}{{(t)}^{\delta }}]&=&2^\delta\EXP[\eta {{(\alpha (t-T_\theta))}^{\delta }}]\le 2^\delta 
\EXP[{{\eta }_{\delta }}(\alpha (t-T_\theta))]\\
&=&2^\delta\int_{0}^{t}{{{e}^{-s}}{{\eta }_{\delta }}(\alpha (t-s))ds}=\eta_\delta(t)
\end{eqnarray*}
according to \autoref{exampleone}. Suppose $\EXP[\sX_{n-1}(t)^\delta]\le\eta_\delta(t)$ for some $n\ge 2$. 
Using the inequality $(a+b)^\delta\le 2^{\delta-1}(a^\delta+b^\delta)$, we have
\[\mathscr{X}_n(t)^\delta\le\left\{ \begin{array}{*{35}{l}}
   0 & \text{if} & T_\theta\ge t,  \\
   2^{\delta-1}(\sX_{n-1}^{(1)}(\alpha (t-T_\theta))^\delta+\sX_{n-1}^{(2)}(\alpha (t-T_\theta))^\delta) & \text{if} & T_\theta<t.  \\
\end{array} \right.\]
Thus,
\begin{eqnarray}
\label{11133}\EXP[{{\sX}_{n}}{{(t)}^{\delta }}]&\le& 
{{2}^{\delta -1}}\int_{0}^{t}{{{e}^{-s}}2\,\EXP[{{\sX}_{n-1}}{{(\alpha (t-s))}^{\delta }}]ds}
={{2}^{\delta -1}}\int_{0}^{t}{{{e}^{-s}}2\eta_\delta(\alpha(t-s))ds}=\eta_\delta(t)
\end{eqnarray}
Thus, the sequence $\sX_n$ is uniformly integrable and  $\EXP\sX(t)=\lim_{n\to\infty}\EXP\sX_n(t)=\eta(t)$. Taking the limit of both sides of \eqref{11133} as $n\to\infty$ 
and applying Fatou's Lemma, we get $\EXP[\sX(t)^\delta]\le \eta_\delta(t)$.

(iv) Suppose by contradiction that $[\sX(t)>0]\cap[L<t]$ is a null event. Fix $\delta\in(1,1/\gamma)$. Then $\sX(t)=\sX^\delta(t)=0$ almost surely on the event $[L<t]$. On the other hand, since $\delta>1$ and $\sX(t)\ge 0$, we have $\sX(t)-\sX^\delta(t)\le 1$. Together with Part (iii), we have
\begin{align}
\eta(t)-\eta_\delta(t)\le \EXP[\sX(t)-\sX^\delta(t)]=\EXP[(\sX(t)-\sX^\delta(t))\1_{[L>t]}]\le
\EXP[\1_{[L>t]}]=\P(L>t).
\end{align}
Together with \autoref{expl_asymptotics}, we have
\begin{equation}\label{comp1}
t^{\gamma}\left(\eta(t)-\eta_\delta(t)\right)\le t^{\gamma} \P(L>t)\le Ct^{\gamma}e^{-t}.
\end{equation}
Consequently, $\lim_{t\to\infty}\textup{LHS}\eqref{comp1}=0$. On the other hand, by \autoref{prop11112} (iii) and \autoref{exampleone}, the limit of LHS\eqref{comp1} as $t\to\infty$ is equal to 1, which is a contradiction.
\end{proof}

To finish the proof of \autoref{main} in the case $\alpha>2$, $u_0=1$, we will need the 
following lemma.

\begin{lem}\label{lem1113}
For each $\delta\in(1,2)$, there exists $c_\delta>0$ such that
\[1-e^{-x}\ge x-c_\delta x^\delta\ \ \ \forall\,x>0\]
\end{lem}
\begin{proof}
For a fixed $c>0$, let $f(x)=1-e^{-x}-(x-cx^\delta)$. Then
\begin{eqnarray*}
f'(x)&=&e^{-x}-1+c\delta x^{\delta-1}\\
f''(x)&=&-e^{-x}+c\delta(\delta-1)x^{\delta-2}
\end{eqnarray*}
Since $\delta\in(1,2)$, one can choose $c$ sufficiently large such that $f''(x)>0$ for all $x>0$. 
Then $f'(x)\ge f'(0)=0$ 
for all $x>0$. Then $f(x)\ge f(0)=0$ for all $x>0$.
\end{proof}

We can now finish the proof of \autoref{main} in the case $\alpha>2$, $u_0=1$ by showing the following.

\begin{thm}\label{main2}
Let $\sX(t)$ be given by \autoref{martingalep} (iii). Then
\begin{enumerate}[(i)]
\item For each $\lambda>0$, the process 
$X_\lambda(t)=e^{-\lambda\sX(t)}$ is a solution process satisfying \eqref{alphrecursion} with 
$u_0=1$.
\item The function $u_\lambda(t)=\EXP X_{\lambda}(t)$ solves \eqref{diffalpha-ricc0}. 
\item For $0<\lambda'<\lambda$, we have $0< u_{\lambda}(t)<u_{\lambda'}(t)< 1$ for every $t>0$.
\item Let $\gamma=\log_\alpha 2<1$. For each $\lambda>0$,
\[\lim_{t\to\infty}\frac{1-u_{\lambda}(t)}{t^{-\gamma}}={\lambda}.\]
\end{enumerate}
\end{thm}

\begin{proof}
Part (i) and (ii) come directly from \autoref{transf_pant_to_alph}. We only need 
to show the convergence rate. Part (iii) follows from \autoref{martingalep} (iii) and (iv). To show (iv), note that
\[
1-u_\lambda(t)=\EXP[1-e^{-\lambda\sX(t)}]\le \EXP[\lambda\sX(t)]=\lambda\EXP\sX(t)
={\lambda}\eta(t).
\]
Together with \autoref{prop11112} (iii), one has
\[\underset{t\to \infty }{\mathop{\lim \sup }}\,\frac{1-{{u}_{\lambda }}(t)}{{{t}^{-\gamma }}}
\le\underset{t\to \infty }{\mathop{\lim \sup }}\,\frac{\eta (t)}{{{t}^{-\gamma }}}={\lambda }.\]
It remains to show that 
\[
\underset{t\to \infty }{\mathop{\lim \inf }}\,\frac{1-{{u}_{\lambda }}(t)}{{{t}^{-\gamma }}}\ge {\lambda}.
\]
Since $\gamma\in(0,1)$, there exists $\delta\in(1,2)$ such that $\gamma\delta\in(0,1)$. 
By \autoref{lem1113}, there exists a constant $c_\delta>0$ such $1-e^{-x}\ge x-c_\delta x^\delta$ for all $x\ge 0$. 
Thus,
\begin{eqnarray}
\nonumber 1-{{u}_{\lambda }}(t)=\mathbb{E}[1-{{e}^{-\lambda \mathscr{X}(t)}}]&\ge& 
\EXP[\lambda \mathscr{X}(t)-c_\delta{{\lambda }^{\delta }}\sX{{(t)}^{\delta }}]\\
\label{11134}&=&{\lambda }\eta (t)-c_\delta{{\lambda }^{\delta }}\EXP[\sX{{(t)}^{\delta }}].
\end{eqnarray}
By \eqref{11132} and \autoref{martingalep} (iii), we have
\[\EXP[{{\sX}}{{(t)}^{\delta }}]\le \eta_\delta(t)\le 
C{{t}^{-\gamma \delta }}\ \ \ \forall t>0,\,n\in \mathbb{N}.\]
Substituting this estimate into \eqref{11134}, we get
\[\underset{t\to \infty }{\mathop{\lim \inf }}\,\frac{1-{{u}_{\lambda }}(t)}{{{t}^{-\gamma }}}
\ge \underset{t\to \infty }{\mathop{\lim \inf }}\,
\left( {\lambda }\frac{\eta (t)}{{{t}^{-\gamma }}}-C_\delta\lambda^\delta\frac{{{t}^{-\gamma \delta }}}{{{t}^{-\gamma }}} \right)
={\lambda }\underset{t\to \infty }{\mathop{\lim }}\,\frac{\eta (t)}{{{t}^{-\gamma }}}={\lambda }.\]
This completes the proof.
\end{proof}

\subsection{The case
\texorpdfstring{$u_0\in D_\alpha$}{}}\label{generalinitialdata}
The proof of this general case uses the transformation \eqref{trans3}. In other words, we combine the maximal solution 
process $\uprX_{u_0}$ defined for $u_0>0$ in \eqref{max_sol_proc} with
the solution processes corresponding to $u_0=1$ 
to construct solution processes $X_{u_0,\lambda}$ for the case $u_0\in D_\alpha.$ 
This serves as an illustration of the use of solution processes to combine solutions with different initial data, an analog of which at the level of differential equations is highly nontrivial.

Considering $M$ and $\delta$ to be fixed, let us denote
\begin{equation}\label{Zproc}
\ZZ(t)=\left\{ \begin{array}{*{35}{r}}
   (X_{M,\delta}(t))^{1/\delta}    &\quad\ \textup{if} & 1<\alpha \le 2,  \\
   e^{-\sX(t)} & \textup{if} & \alpha >2  \\
\end{array} \right.\end{equation}
where $X_{M,\delta}$ is the process defined in \autoref{non-indicator} and $\sX$ is the process defined in \autoref{martingalep} (iii). Almost surely, $\ZZ(t)\in[0,1]$. By \autoref{non-indicator} (i) and \autoref{martingalep} (iv), the event $[\ZZ(t)\in(0,1)]\cap[L<t]$ is of positive probability. 
For $\lambda>0$, let $X_{M,\delta,\lambda}$ be the process defined in \autoref{main1} 
for $\alpha\in(1,2]$, and $X_\lambda$ the process defined in 
\autoref{main2} for $\alpha>2$.
Then,
\[\ZZ^\lambda(t)=\left\{ \begin{array}{*{35}{r}}
   (X_{M,\delta}(t))^{\lambda/\delta}=X_{M,\delta,\lambda}(t)    &\quad\ \textup{if} & 1<\alpha \le 2,  \\
   e^{-\lambda\sX(t)}=X_{\lambda}(t) & \textup{if} & \alpha >2.  \\
\end{array} \right.\]
By \autoref{main1} and \autoref{main2}, for each $\lambda>0$, the function 
$v_\lambda=\EXP [\ZZ^\lambda]$ 
is a solution to \eqref{diffalpha-ricc0} with $u_0=1$ and
\[\lim_{t\to\infty}\frac{1-v_\lambda(t)}{t^{-\gamma}}={\lambda}\]
where $\gamma>0$ is given by \eqref{gma}. 

It follows from \cite{alphariccati}*{Theorem 4.2 and Remark 3.3} that
\begin{equation}\label{wideu}
\widebar{u}(t):=\EXP\uprX_{u_0}(t)\le 1+\frac{2\alpha-1}{4}C_0\, e^{-t}\qquad\forall t\ge 0,\,u_0\in D_\alpha
\end{equation}
where
\begin{equation}\label{constC0}
C_0=\frac{1}{\sqrt{\pi}}\sum_{n=1}^\infty\frac{1}{n^{3/2}}\approx 1.4739.\end{equation}
Fix $u_0\in D_\alpha$. Define $X_{u_0,\lambda}(t)= \uprX_{u_0}(t)\ZZ^\lambda(t)$.
Observe that $\uprX_{u_0}(t)$ satisfies \eqref{alphrecursion} and $\ZZ^\lambda(t)$ satisfies \eqref{alphrecursion} with $u_0$ being replaced by 1. Since these two processes are defined on the same probability space, $X_{u_0,\lambda}(t)$ also satisfies \eqref{alphrecursion}. 
Thus, the function $u_{\lambda}(t):=\EXP X_{u_0,\lambda}(t)$ satisfies \eqref{diffalpha-ricc0}. 
Since $\uprX_{u_0}(t)=1$ on the event $[L<t]$, one has $X_{u_0,\lambda}(t)=\ZZ^\lambda(t)$ on $[L<t]$. Hence, $X_{u_0,\lambda}(t)\in (0,1)$ with a positive probability. Therefore, for $0<\lambda'<\lambda$,
\[0<u_\lambda(t)<u_{\lambda'}(t)<\EXP[\uprX_{u_0}(t)]\quad\forall t>0.\]
On the other hand,
\begin{eqnarray*}
u_{\lambda}(t)=\EXP[\uprX_{u_0}(t)\ZZ^\lambda(t)]
&=&\EXP [\ZZ^\lambda(t)]-\EXP[(1-\uprX_{u_0}(t))\ZZ^\lambda(t)]\\
&=&v_\lambda(t)-\EXP[(1-\uprX_{u_0}(t))\ZZ^\lambda(t)\IND{[L>t]}]
\end{eqnarray*}
and
\begin{equation}\label{uvcom}
\frac{1-u_{\lambda}(t)}{t^{-\gamma}}=
\frac{1-v_\lambda(t)}{t^{-\gamma}}
+t^{\gamma}\EXP[(1-\uprX_{u_0}(t))\ZZ^\lambda(t)\IND{[L>t]}].
\end{equation}
If $u_0>1$ then $\uprXu(t)\ge 1$ and
\[
0\le\EXP[(\uprX_{u_0}(t)-1)\ZZ^\lambda(t)\IND{L>t}]\le \EXP[\uprXu(t)-1]=\overline{u}(t)-1\le c e^{-t}. 
\]
If $u_0\in[0,1]$ then 
$\uprXu(t)\in[0,1]$. Together with \autoref{expl_asymptotics}, we have
\[0\le \EXP[(1-\uprX_{u_0}(t))\ZZ^\lambda(t)\IND{[L>t]}]\le \P(L>t)\le C e^{-t}.\]
In both cases, $t^{\gamma}\EXP[(1-\uprX_{u_0}(t))\ZZ^\lambda(t)\IND{[L>t]}]\to 0$ as $t\to\infty$.
We then infer from \eqref{uvcom}  that
\[\frac{1-u_{\lambda}(t)}{t^{-\gamma}}={\lambda}.\]

 \section{Structure of the explosion event and applications}\label{Sec5}

To analyze the structure of the explosion event, let us introduce the following concepts. 
Recall that  $v*w=(v_1,\ldots,v_k, w_1,\ldots,w_l)$ is the concatenation of vertex $v=(v_1,\ldots,v_k)$ and vertex $w=(w_1,\ldots,w_l)\in\T$. Denote by $\T_v=\{v*w:\,w\in \T\}$ the \emph{subtree of $\T$ rooted at $v$}.

\begin{defin}\label{mhs}
For $v\in\T$ define 
$$S_v=\inf_{s\in\partial\T: v\in s}\sum_{j=|v| }^\infty Y_{s|j}, ~~~~L_v=\sup_{s\in\partial\T: v\in s}\sum_{j=|v|}^\infty\ Y_{s|j}$$
to be the shortest and longest path on $\T_v$. Define also
\begin{itemize}
\item \emph{$(t,v)$-explosion} to be the event $[S_v\le t-\Theta_{\p{v}}]$,
\item \emph{$(t,v)$-hyperexplosion} to be the event
$[L_v \le t-\Theta_{\p{v}}],$ and refer to $\T_v$ as a \emph{$t$-hyperexplosive subtree.}
\end{itemize}
$\T_v$ is said to be a \emph{maximal t-hyperexplosive subtree} if it is $t$-hyperexplosive and it is not contained in another subtree that is $t$-hyperexplosive.
\end{defin}
 
Note that $S_v$ and $L_v$ have the same distribution as $\alpha^{|v|}S_\theta$ and $\alpha^{|v|}L_\theta$, respectively.  

\subsection{Proof of \autoref{second}}
In this section we present the proof of \autoref{second} by proving each statement as a proposition. Using \autoref{mhs}, 
\autoref{second} (i) is equivalent to the statement that conditionally on the explosion by time $t$, there exists almost surely a $t$-hyperexplosive subtree. To prove this statement, let $H_t$ be the event that there exists a $t$-hyperexplosive subtree $$H_t=\bigcup_{j=0}^{\infty}\bigcup_{|v|=j} [L_v\le t-\Theta_{\overleftarrow{v}}].$$  
For two events $A$ and $B$, we say that $A=B$ almost surely (respectively, $A\subset B$ almost surely)  if $\1_A=\1_B$ almost surely (respectively, $\1_A\le \1_B$).

\begin{prop}\label{hypersubprop}
For $\alpha >1$ and $t>0$, 
$H_t=[S\le t]$ almost surely.
\end{prop}

\begin{proof} 
The total accumulation time along a ray $s\in\partial\T$ and the replacement time at $v\in\T$ are
\begin{equation}
\Theta_s=\sum_{j=0}^\infty\alpha^{-j}T_{s|j},\  s\in\partial\T,\qquad
\Theta_v=\sum_{j=0}^{|v|}\alpha^{-j}T_{v|j},\  v\in\T.
\end{equation}
Fix $t>0$. 
Since $H_t\subset [S\le t]$ and since $S$ has a continuous distribution (\autoref{contdist}), the proposition follows once we show that $\P(H_t^c \cap [S < t]) = 0$.

Note that $H_t= \cup_j\cup_{|v|=K_j} [L_v \le t-\Theta_{\overleftarrow{v}}]$ where $K_j, j=1,\dots$ is an arbitrary increasing sequence of positive integers.   Indeed if $v\in\T$ and $L_v \le t-\Theta_{\overleftarrow{v}}$,  let $j$ be such that $K_{j-1} < |v| \le K_j$.  Recall that  $v\in w$ means $w|_{|v|} = v$. Thus one has that $L_w \le t - \Theta_{\overleftarrow{w}}$, $\forall w\in\T, |w|=K_j, v\in w$ .   Consequently, for any $n$, 
\begin{equation}
\label{nohyperexplosion}
H_t^c = \bigcap_{j=0}^\infty \bigcap_{|v|=K_j} [L_v > t-\Theta_{\overleftarrow{v}} ]
 \subset \bigcap_{|v|=K_n} [L_v > t-\Theta_{\overleftarrow{v}}]
 \end{equation}
 
Likewise, one can describe the event $[S < t]$ using only vertices of $\T$ of length $K_n$.  Indeed for any $N\ge 1$,
$$
[S < t] \subset\bigcup_{n\ge N}\bigcup_{|v|=K_n}\left(\cup_{s\in\partial\T, v\in s}
[\Theta_s \le t-\frac{1}{n}]\right),
$$
and since for $v\in s$, $[\Theta_s \le t-\frac{1}{n}] \subset [\Theta_{\overleftarrow{v}} \le t -\frac{1}{n}]$, one has
\begin{equation}
\label{explosion}
[S < t] \subset \bigcup_{n\ge N}\bigcup_{|v|=K_n} [\Theta_{\overleftarrow{v}} \le t -\frac{1}{n}].
\end{equation}
Using \eqref{nohyperexplosion} and \eqref{explosion}
$$
H_t^c \cap [S < t] \subset  \bigcup_{n\ge N}\bigcup_{|v|=K_n} ( [L_v > t-\Theta_{\overleftarrow{v}}]\cap  [\Theta_{\overleftarrow{v}} \le  t -\frac{1}{n}])
$$
We now choose $K_n$ to complete the proof.  Note that 
\begin{equation}
\label{inclusion}
([L_v > t-\Theta_{\overleftarrow{v}}]\cap  [\Theta_{\overleftarrow{v}} \le t -\frac{1}{n}] )\subset [L_v > \frac{1}{n}].
\end{equation}  Recalling that in distribution $L_v = \frac{ L_\theta}{\alpha^{|v|}}$ one has
$$
\P\left(L_v > \frac{1}{n}\right)
= \P\left(L_\theta > \frac{\alpha^{|v|}}{n}\right)\le c\exp\left(-\frac{\alpha^{|v|}}{n}\right)
$$
since by \autoref{expl_asymptotics}, $\P(L_\theta > r)\le ce^{-r}, r\ge 0$, for some constant $c>0$.
Now, choose $K_n>n$, such that 
$$ c\exp\left(-\frac{\alpha^m}{n}\right)< 3^{-m}, \quad \forall m\ge K_n.$$
This is possible since for every $n\ge 1$, $3^m\exp(-\frac{\alpha^m}{n})=o(1)$ as $m\to\infty$. 
Thus, by \eqref{inclusion} we have 
\[\P\left([L_v>t-\Theta_{\overleftarrow{v}}]\cap\left[t-\Theta_{\overleftarrow{v}}{{\ge}}\frac{1}{n}\right]\right)\le 3^{-|v|}\quad \mbox{for all $v\in\T$ with $|v|=K_n$}\]
and 
\begin{eqnarray}
\P(H_t^c\cap[S <t]) &\le& \sum_{n\ge N}\sum_{|v|=K_n}
\P\left([L_v>t-\Theta_{\overleftarrow{v}}]\cap\left[t-\Theta_{\overleftarrow{v}}{{\ge}}\frac{1}{n}\right]\right) \nonumber\\
&\le& \sum_{n\ge N}
\sum_{|v|=K_n}3^{-|v|}
=  \sum_{n\ge N}2^{K_n}3^{-K_n}\nonumber
\le  \sum_{n\ge N}\left(\frac{2}{3}\right)^n=3\left(\frac{2}{3}\right)^N.
\end{eqnarray}
Thus, letting $N\to\infty$,
$\P(H_t^c\cap[S < t]) = 0$ as claimed.
\end{proof}

The second statement in \autoref{second} can be further refined using the terminology introduced in \autoref{mhs}. Indeed, the genealogically oldest particles mentioned in \autoref{second}, correspond to the roots of maximally $t$-hyperexplosive subtrees.  Furthermore, the finite cardinality of $\partial V(t)$ is equivalent to the finite cardinality of $\Vo(t)$ excluding vertices in any of these maximally $t$-hyperexplosive subtrees. The precise statement is as follows.

\begin{prop}\label{finitemaximal}
For $t>0$, the following are true:
\begin{enumerate}[(i)]
\item $[|\partial V(t)|<\infty]=[\text{there are only finitely many maximal $t$-hyperexplosive subtrees}]$ almost surely.
\item On $[|\partial V(t)|<\infty]$, let $\T_{v_1}$, $\T_{v_2}$,..., $\T_{v_n}$ be the maximal $t$-hyperexplosive subtrees. Then there are only finitely many vertices $v\in \T\backslash(\T_{v_1}\cup...\cup \T_{v_n})$ such that $\Theta_v<t$.
\end{enumerate}
\end{prop}
\begin{proof}
Let $W(t)$ be the event that there are only finitely many maximal $t$-hyperexplosive subtrees.\\
(i) First, we show that $[|\partial V(t)|<\infty]\subset W(t)$. On the event $[|\partial V(t)|<\infty]$, the set \[A=\{v|j:\ 0\le j\le |v|,\,v\in\partial V(t)\}\] is finite. Suppose by contradiction that there are infinitely many maximal hyperexplosive subtrees with respect to the horizon $t$. One can label them by $\T_{v_1}$, $\T_{v_2}$,$\T_{v_3}$,... where $v_1,v_2,v_3,...$ are distinct vertices. Let $v_j'$ be the sibling of $v_j$. Note that $v'_1,v'_2,v'_3,...$ are also distinct vertices. We claim that the subtree $\T_{v'_j}$ is not hyperexplosive with respect to $t$. Indeed, suppose that it is. Then the subtree $\T_{\p{v_j}}$, where $\overleftarrow{v_j}$ is the parent of $v_j$, is hyperexplosive with respect to $t$, which contradicts the maximality of $\T_{v_j}$.

Since $\T_{v'_j}$ is not hyperexplosive with respect to $t$, there exists a path $s_j\in\partial\T$ passing through $v'_j$ that crosses the $t$-horizon, i.e.
$\sum_{i=0}^\infty{\alpha^{-i}T_{s_j|i}}>t$. Each path $s_j$ must contain one of the vertices  in $\partial V(t)$. Thus, $v'_j\in A$ for all $j\in \mathbb{N}$. This contradices the fact that $A$ is a finite set.

Next, we show that $W(t)\backslash [|\partial V(t)|<\infty]$ is a null event. For each $v\in\T$, denote by $D_v$ the event that the subtree $\T_v$ is $t$-explosive and does not contain any $t$-hyperexplosive subtree. In other words,
\[D_v=H^c_{v,t}\cap[S_v\le t-\Theta_{\overleftarrow{v}}]\]
where $H_{v,t}=\cup_{w\in\T_v}[L_w\le t-\Theta_{\overleftarrow{w}}]$. Note that the cascade $\{T_w\}_{w\in\T_v}$ has the same distribution as the Yule cascade $\{T_v\}_{v\in\T}$. By \autoref{hypersubprop}, $\P(D_v|\Theta_{\overleftarrow{v}})=0$. Thus,
\[\P(D_v)=\EXP[\P(D_v|\Theta_{\p{v}})]=0.\]
Let $D=\cup_{v\in\T}D_v$, which is also a null event because $\T$ is a countable set. The proof will be complete if we can show that $W(t)\backslash D\subset [|\partial V(t)|<\infty]$.

On the event $W(t)\backslash D$, let $\T_{v_1}$, $\T_{v_2}$,..., $\T_{v_n}$ be all the maximal $t$-hyperexplosive subtrees. Let $h=\max\{|v_1|,..,|v_n|\}+1$ and
\begin{eqnarray*}
A'&=&\left\{v\in\T:\,|v|=h,\ v\not\in\bigcup_{j=1}^n \T_{v_j}\right\},\\
A_1&=&\bigcup_{v\in A',\,\Theta_{\p{v}}> t} \{v|j:\,0\le j\le |v|\},\\
A_2&=&\bigcup_{v\in A',\,\Theta_{\p{v}}\le t}\partial V_v(t-\Theta_{\p{v}}),\\
A''&=&A_1\cup A_2
\end{eqnarray*}
where $\partial V_v(\tau)$ denotes the set of $\tau$-leaves of the subtree $\T_v$ 
\begin{equation*}
\partial{V}_v(\tau) =\left\{w\in\T_v: \sum_{j=|v|}^{|w|-1}\alpha^{-j}T_{w|j}<\tau\le\sum_{j=|v|}^{|w|}\alpha^{-j}T_{w|j}\right\}.
 \end{equation*}
We claim that $\partial V(t)\subset A''$. Indeed, for each $w\in\partial V(t)$, note that $w$ cannot be a descendant of any of $v_1,v_2,...,v_n$ because $\Theta_w>t$. If $|w|<h$ then $w\in A_1$. If $|w|\ge h$ then $w|h\in A'$ and $w\in A_2$. The claim has been proven. 

For each $v\in A'$ with $\Theta_{\p{v}}\le t$, we claim that $\T_v$ is not $(t-\Theta_{\overleftarrow{v}})$-explosive. Suppose otherwise. Then it has a $(t-\Theta_{\overleftarrow{v}})$-hyperexplosive subtree. This contradicts the fact that $v\in A'$. This means $\T_v$ is not $(t-\Theta_{\overleftarrow{v}})$-explosive. Thus, the set $\partial V_v(t-\Theta_{\p{v}})$ is finite. Therefore, $A''$ is also finite and so is $\partial V(t)$.\\\\
(ii) On $[|\partial V(t)|<\infty]$, let $A$ be the set defined in Part (i). Each vertex $w\in \T\backslash(\T_{v_1}\cup...\cup \T_{v_n})$ with $\Theta_w<t$ is not the root of a $t$-hyperexplosive subtree. Hence, there exists a path $s_w\in \T$ that contains $w$ and crosses the horizon $t$. Then $w\in A$. Since $A$ is finite, there can be only finitely many such vertices $w$.
\end{proof}

\autoref{finitemaximal} characterizes the event $[|\partial V(t)|<\infty]$ in terms of maximal $t$-hyperexplosive subtrees. Its complement event, $[|\partial{V}(t)|=\infty]$ referred
to as {\it $t$-leaf percolation}, cannot be ruled out a priori. In fact, according to the theorem below (proven in \cite{dascaliuc2023errata}), this event  occurs with a positive probability only for a restricted range of $\alpha$ in the explosion regime.
One possible realization of the $t$-leaf percolation event is that, starting from some generation onward, the tree $\T$ contains at least one $t$-leaf at each generation.

\begin{thm}[$t$-Leaf Percolation, \cite{dascaliuc2023errata}]\label{alphexplosion}  
One has the following statements.
\begin{enumerate}[(i)]
\item If $\alpha\in[0,1]\cup[2,\infty)$, then
$\P(|\partial V(t)|<\infty) = 1$ for all $t\ge 0$;
\item If $1<\alpha<2$, then $\P(|\partial V(t)|=\infty)>0$
for all $t>0$.
\end{enumerate}
\end{thm}
The following proposition gives characterization of the $t$-leaf percolation event and establishes statement (iii) in \autoref{second}.

\begin{prop}\label{infleaves}
For $t>0$, almost surely 
\[[|\pV(t)|=\infty]=[\exists s\in\partial\T:\ \Theta_s=t].\]
\end{prop}



\color{black}
\begin{proof}
Let \[\overline{W}(t)=[\exists s\in\partial\T:\ \Theta_s=t].\] 
First, let us show that $[|\partial V(t)|=\infty]\backslash \overline{W}(t)$ is a null event. Let 
\[A=[\forall \epsilon>0,\ \exists s\in\partial\T:\ t-\epsilon<\Theta_s\le t]\,.\]
The a.s.\ continuity of $\Theta:\pT\rightarrow [0,\infty]$ stated in \autoref{contdist} (i) implies that $A\backslash \overline{W}(t)$ is a null event. 
\color{black}
It is therefore sufficient to show that $[|\partial V(t)|=\infty]\backslash A$ is a null event. On this event, there exists $\epsilon>0$ such that for every $s\in\partial\T$, either $\Theta_s\le t-\epsilon$ or $\Theta_s>t$. Let
\[\mathcal{V}=\{v\in\T:\ \exists s',s''\in\partial\T\text{~passing~through~}v\text{~such~that~}\Theta_{s'}\le t,\ \Theta_{s''}>t\}.\]
We claim that $\mathcal{V}$ is an infinite set. Suppose otherwise. Let $N=1+\max\{|w|:\ w\in\mathcal{V}\}$. For each $v\in\T$ with $|v|=N$, the subtree $\T_v$ is either $t$-nonexplosive or $t$-hyperexplosive. Let $v_1,v_2,...,v_n$ be the vertices of height $N$ such that $\T_{v_k}$ is $t$-nonexplosive. Each vertex in $\partial V(t)$ is a $(t-\Theta_{v_k})$-leaf of some subtree $\T_{v_k}$. Each subtree $\T_{v_k}$ contributes finitely many elements to the $t$-leaves $\partial V(t)$. Therefore, $\partial V(t)$ is a finite set, which is a contradiction. We have shown that $\mathcal{V}$ is an infinite set.

For each $v\in\mathcal{V}$, we have $L_v\ge\Theta_{s''}\ge \Theta_{s''}-\Theta_{s'}>t-(t-\epsilon)=\epsilon$. For each $m\in\mathbb{N}$, there exists $v\in\mathcal{V}$ such that $|v|\ge m$. Thus,
\[[|\partial V(t)|=\infty]\backslash A\subset \bigcup_{v\in\T,\,|v|\ge m}[L_v\ge\epsilon]=\bigcup_{n\ge m}\bigcup_{|v|=n}[L_v\ge\epsilon]\]
Therefore,
\begin{eqnarray*}
\P([|\partial V(t)|=\infty]\backslash A)\le\sum_{n=m}^\infty\sum_{|v|=n}\P(L_v\ge\epsilon)=\sum_{n=m}^\infty 2^n\P(L\ge\alpha^n\epsilon)\le\sum_{n=m}^\infty c2^ne^{-\alpha^n\epsilon}
\end{eqnarray*}
which tends to 0 as $m\to\infty$. Therefore, $[|\partial V(t)|=\infty]\backslash A$ is a null event. 

Next, let us show that $\overline{W}(t)\backslash [|\partial V(t)|=\infty]$ is a null event. Let $D$ be the null event introduced in the proof of \autoref{finitemaximal} (i). In other words, $D$ is the event that there exists a vertex $v\in\T$ such that $\T_v$ is $t$-explosive but does not have any $t$-hyperexplosive subtrees. Let
\[D'=\bigcup_{v\in\T}[L_v=t-\Theta_{\p{v}}]\]
The event $[L_v=t-\Theta_{\p{v}}]=[L_v+\Theta_{\p{v}}=t]$ is a null event because $L_v$ has a continuous distribution (\autoref{contdist}), that $\Theta_{\p{v}}$ also has a constinuous distribution, and that $L_v$ and $\Theta_{\p{v}}$ are independent. Then $D'$ is a countable union of null events and thus is also a null event. On the event $\overline{W}(t)\backslash (D\cup D')$, there exists a random path $s\in\partial\T$ such that $\Theta_s=t$. For each $j\in\mathbb{N}$, 
\[\sum_{k=j}^\infty \alpha^{-k}T_{s|k}=t-\Theta_{s|(j-1)}.\]
Thus, $L_{s|j}\ge t-\Theta_{s|(j-1)}$. Because the event $D'$ is excluded, equality does not occur. So we have $L_{s|j}> t-\Theta_{s|(j-1)}$. In other words, the subtree $\T_{s|j}$ is $t$-explosive but not $t$-hyperexplosive. It must have a $t$-hyperexplosive subtree according to \autoref{hypersubprop}. Consequently, it must have a maximal $t$-hyperexplosive subtree, say $\Pi_j$. Because $\Pi_j\neq\T_{s|j}$, $\Pi_j$ is also a maximal $t$-hyperexplosive subtree of $\T$. A sequence of distinct maximal $t$-hyperexplosive subtrees $\Pi_1$, $\Pi_2$, $\Pi_3$,\ldots is so obtained. Therefore, $\overline{W}(t)\backslash (D\cup D')\subset [|\partial V(t)|=\infty]$ and the proof is complete.
\end{proof}

\color{black}
The properties of AAS cascades established above, in particular \autoref{hypersubprop}, 
lead to a complete classification of solutions processes that are limits of stochastic Picard iterations with constant ground states, as well as their corresponding mean flow counterpart as given in the \autoref{groundstatesec}.

\subsection{Proof of \autoref{X>0}}

We begin this proof by using results obtained in \autoref{Sec5} to establish the almost sure inclusion
\begin{equation}
\label{inclusionone}
[|\partial V(t)|<\infty] \subset [X_{M,\delta}(t) >0].
\end{equation} 
Consider $M$ and $\delta$ as fixed. Let $X_n(t)=X_{M,\delta,n}(t)$ be the stochastic processes defined by \eqref{eq6211}. By \autoref{non-indicator} (i), $\lim X_n(t)=X(t):=X_{M,\delta}(t)$.
On the nonexplosion event $[S>t]$, the sequence $\{X_n(t)\}_n$ is eventually constant, and thus, $X(t)=X_n(t)>0$ for sufficiently large $n$. It remains to show that
\begin{equation}
\label{inclusiononeone}
[S\le t]\cap[|\partial V(t)|<\infty] \subset [X_{M,\delta}(t) >0].
\end{equation} 
For this, note that 
by \autoref{finitemaximal} (i), there are only finitely many maximal $t$-hyperexplosive subtrees on the event $E=E(t)=[S\le t]\cap[|\partial V(t)|<\infty]$. Let $\T_{v_1},\ldots,\T_{v_k}$ be the maximal $t$-hyperexplosive subtrees. By \autoref{finitemaximal} (ii), the set 
\[A=\{v\in\T\backslash(\T_{v_1}\cup\ldots\cup\T_{v_k}):\ \Theta_v<t\}\]
is finite. Let 
\[h=\max\{|v|:\ v\in A\}+\max\{|v_1|,\ldots,|v_k|\}.\]
For $n>h$, a.s.
\begin{equation}\label{Xnjestimate}{{X}_{n}}(t)\IND{E}=\IND{E}\prod\limits_{|w|=n-1}{{{\rho }^{2}}({{\alpha }^{n}}(t-{{\Theta }_{w}}))}=\IND{E}\prod\limits_{j=1}^{k}{\prod\limits_{\begin{smallmatrix} 
 |w|=n-1 \\ 
 w\in {{\T}_{{{v}_{j}}}} 
\end{smallmatrix}}{{{\rho }^{2}}({{\alpha }^{n}}(t-{{\Theta }_{w}}))}}=\prod\limits_{j=1}^{k}{{{X}^{(j)}_{n}}(t)}\end{equation}
where
\[{{X}^{(j)}_{n}}(t)=\IND{E}\prod\limits_{|w|=n-1-|{{v}_{j}}|}{{{\rho }^{2}}({{\alpha }^{n}}(t-{{\Theta }_{{{v}_{j}}*w}}))}.\]
One can estimate $X^{(j)}_{n}(t)$ in a manner analogous to \eqref{Xnestimate} as follows. For $1\le j\le k$, let $L_j=\sup_{s\in\partial{\T}}\{\Theta_{v_j*s}\}$ denote the time of hyperexplosion on $\T_{v_j}$, and define
\[
\tilde{L} = \max\{L_1, \dots, L_k\}.
\]
Similarly, let $S_j =\inf_{s\in\partial{\T}}\{\Theta_{v_j*s}\}$ denote the time of explosion on $\T_{v_j}$, 
and define
\[
\tilde{S} = \inf\{S_1, \dots, S_k\}.
\]
Then
\[\left( 1-\delta \frac{{{(t-L_j)}^{-\gamma }}}{{{2}^{n}}} \right)^{2\cdot2^{n-1-|v_j|}}\IND{E}
\le  X^{(j)}_{n}(t)\le \left( 1-\delta \frac{{{(t-S_j)}^{-\gamma }}}{{{2}^{n}}} \right)^{2\cdot2^{n-1-|v_j|}}\IND{E}.\]
Substitute this estimate into \eqref{Xnjestimate} to get
\[\left( 1-\delta \frac{{{(t-\tilde{L})}^{-\gamma }}}{{{2}^{n}}} \right)^{2\kappa\cdot2^{n-1}}\IND{E}
\le  X_{n}(t)\IND{E}\le \left( 1-\delta \frac{{{(t-\tilde{S})}^{-\gamma }}}{{{2}^{n}}} \right)^{2\kappa\cdot2^{n-1}}\IND{E}\]
where $\kappa=\sum_{j=1}^k 2^{-|v_j|}$.
Letting $n\to\infty$, we have
\[e^{-\delta\kappa(t-\tilde{L})^{-\gamma}}\IND{E}\le X(t)\IND{E}\le e^{-\delta\kappa(t-\tilde{S})^{-\gamma}}\IND{E}\,.\]
Therefore, $X(t)>0$ on the event $E$ and \eqref{inclusionone} is established.  Equivalently, we have
\begin{equation}
    \label{inclusiontwo}
    [X_{M,\delta}(t)=0]\subset [|\partial V(t)|=\infty]
\end{equation}

New tools are needed to obtain the reverse inclusion of \eqref{inclusiontwo} that would complete the proof of \autoref{X>0}. Heuristically, by its definition $X_n(t) = \prod_{v\in\T, |v|=n-1} \rho^2(\alpha^n(t-\Theta_v))$, where $\rho$ is defined in \eqref{eq:72191}. The strategy of the proof presented below will establish that for each $k\in\mathbb{N}$, there exists a subsequence $m_n$ such that $X_{m_n+k}$ has $2^k$ factors that evaluate $\rho$ at points in the interval $[M,M\alpha]$. Each of these factors attain values that are bounded by $a=1-\delta (M\alpha)^{-\gamma}.$ The lemmata below impose restrictions on $M, \delta$ that in particular will result in $a^{2^k}$ being arbitrarily small. 

The key step in carrying out this program is showing
that, for sufficiently small $\delta>0$, on
the event $[|\partial V(t)|=\infty]$ there
are infinitely many vertices $v\in\T$ for which $\alpha^{|v|+1}(t-\Theta_v) \in [\delta,\alpha\delta]$ (\autoref{R02}). The main tools are L\'{e}vy's conditional Borel-Cantelli
lemma, stated next, and the cutset theory developed in
\cite{part2_2021} and briefly summarized below.

\begin{lem}[Conditional Borel-Cantelli lemma]\cite[Cor 68, p.\ 249]{levy}, \cite[Cor VII-2-6, p.\ 152]{neveu}, \cite[Thm 12.15]{williams}\label{cbc}
Let $(\mathcal A_n)_{n\ge0}$ be a filtration of a probability space $(\Omega,\mathcal{F},\P)$, and let $E_n\in\mathcal A_n$. Then almost surely,
\[\left[E_n\ \text{infinitely often}\right]
=
\left[\sum_{n=1}^{\infty}
\P(E_n\mid\mathcal A_{n-1})=\infty\right].\]
\end{lem}

For a fixed $t>0$, define the \emph{scaled remaining time at death} of each $v\in\T$ as 
\begin{equation}\label{Rv1}
\TT_v=\alpha^{|v|+1} (t-\Theta_{v}).\end{equation}
Note that $\{\TT_v\}_{v\in\T}$ can be defined recursively as
\begin{equation}\label{Rv2}
\TT_{v}=\alpha(\TT_{\p{v}}-T_v),\ \text{with } \TT_{\p{\theta}}=t.
\end{equation}
\begin{remark}\label{cutsetrem}
For each $s\in\pT$, the sequence $\{\TT_{s|j}\}_{j\ge 0}$ is a Markov process on $\R$ with iterated maps $\TT_{s|j}=\alpha(\TT_{s|j-1}-T_{s|j})$ for $j\ge 1$, and probability density function 
$$f(x|r)=\frac{1}{\alpha}e^{-(r-\frac{x}{\alpha})}\1_{(-\infty,\alpha r]}(x).$$ 
Note that the transition probabilities do not depend on the ray $s\in\pT$. Moreover, for any $v\in\T$, the families $\{\TT_{w}\}_{w\in\T_{v*1}}$ and $\{\TT_{w}\}_{w\in\T_{v*2}}$ are conditionally independent of each other given $\TT_v$.
\end{remark}
Let us order the vertices of $\T$ \emph{lexicographically} with the root being the smallest element. Thus, $\theta\preceq 1\preceq 2\preceq 11\preceq 12\preceq 21\prec 22\preceq 111\preceq 112\preceq ...$ Denote by $p(v)$ the vertex preceding $v$ in this order (for example, $p(111)=22$). Note that $\p{v}\prec p(v)$ for all $v\in\T\backslash\{\theta\}$.

A natural filtration that respects this order of vertices is $\{\F_v\}_{v\in\T}$, where \begin{equation}\label{filtration}
\F_v=\sigma(T_w:\,w\preceq v). 
\end{equation}
Note that $\TT_{\p{v}}$ is $\F_{\p{v}}$-measurable, and thus is $\F_{p(v)}$-measurable. Note also that $T_v$ is independent of $\F_{p(v)}$.
\begin{lem}\label{R01}
Let $\alpha>1$, $b>0$, $t>0$. Then $[|\pV(t)|=\infty]\subset [\TT_v\in [0,b] \io]$ a.s.
\end{lem}
\begin{proof}
Define the events
\[A_v=\left[0\le \TT_{\p{v}}<T_v\right],\quad B_v=\left[0\le \TT_{\p{v}}-\beta\le T_v\le \TT_{\p{v}}\right],\]
where $\beta=b/\alpha\in(0,b)$. Let $\{\F_v\}_{v\in\T}$ be the filtration given by \eqref{filtration}. Then $A_v$ and $B_v$ are $\F_v$-measurable. The event $A_v$ occurs if and only if $v$ is a $t$-leaf. The event $B_v$ occurs if $\TT_{v}\in[0,b]$ and $\TT_{\p{v}}\ge\beta$. 
Thus,
\begin{equation}\label{avoften}
[A_v \io ]=[|\pV(t)|=\infty] \text{ a.s.}
\end{equation}
Since $T_v$ is independent of $\F_{p(v)}$,
\begin{align}
\label{av}\P(A_v|\F_{p(v)})&=\P\left(0\le \TT_{\p{v}}<T_v\big|\TT_{\p{v}}\right)=e^{-\TT_{\p{v}}}\1_{[\TT_{\p{v}}\ge 0]},\\
\label{bv}\P(B_v|\F_{p(v)})&=\P\left(0\le \TT_{\p{v}}-\beta\le T_v<\TT_{\p{v}}\big|\TT_{\p{v}}\right)=e^{-\TT_{\p{v}}}(e^{\beta}-1)\1_{[\TT_{\p{v}}\ge \beta]}.
\end{align}
By the conditional Borel-Cantelli lemma (\autoref{cbc}),
\begin{align}
\label{av2}
[A_v\io]&=\left[\sum_{v\in\T}e^{-\TT_{\p{v}}}\1_{[\TT_{\p{v}}\ge 0]}=\infty\right]\text{ a.s.}\\
\label{bv2}[B_v\io]&=\left[\sum_{v\in\T}e^{-\TT_{\p{v}}}\1_{[\TT_{\p{v}}\ge \beta]}=\infty\right]\text{ a.s.}
\end{align}
Let $E_{\delta}=[\TT_v \in [0,\delta] \io].$
Then almost surely,
\begin{equation}\label{Ea}
[|\pV(t)|=\infty]\cap E^c_{\beta}\xlongequal{\eqref{avoften}, \eqref{av2}}[A_v\io] \cap E^c_{\beta}\xlongequal{\eqref{bv2}}[B_v\io]\end{equation}
Recall that if $B_v$ occurs then $\TT_{v}\in[0,b]$. Thus, $[B_v \io]\subset E_{b}$ and \eqref{Ea} implies that almost surely,
\[[|\pV(t)|=\infty]\cap E^c_{\beta}\subset E_b.\]
Together with the fact that $E_{\beta}\subset E_b$, we obtain that almost surely,
\[[|\pV(t)|=\infty]=([|\pV(t)|=\infty]\cap E_{\beta})\cup([|\pV(t)|=\infty]\cap E^c_{\beta})\subset E_{\beta}\cup E_{b}\subset E_{b} = [\TT_v\in[0,b]\io]\]
as claimed.
\end{proof}
The proof of the following lemma uses the notion of cutsets. A cutset is a finite subset of $\T$ that intersects each ray $s\in\pT$ at a unique vertex.
\begin{lem}\label{R02}
Let $\alpha>1$, $\delta\in(0,\ln 2)$, $t>0$. Then $[|\pV(t)|=\infty]\subset [\TT_v \in [\delta,\delta\alpha] \io]$ a.s.
\end{lem}
\begin{proof}
Let $A=\R\backslash [0,\delta]$. For any $a\in\R$ and $s\in\pT$, define
\begin{align*}I_{n}(a,A)&=\P(\TT_{s|1}\not\in A,\, \TT_{s|2}\not\in A,...,\,\TT_{s|n}\not\in A\,|\,\TT_\theta=a)\\
&=\P(\TT_{s|1}\in[0,\delta],\, \TT_{s|2}\in[0,\delta],...,\,\TT_{s|n}\in[0,\delta]\,|\,\TT_\theta=a)\quad \forall n\in\mathbb{N}.
\end{align*}
Note that $I_{n}(a,A)$ does not depend on the choice of the ray $s\in\pT$ because the transition probabilities of the Markov process $\{\TT_{s|j}\}_{j\ge 0}$ does not depend on $s$ (\autoref{cutsetrem}). On the event $[\TT_{s|1},\TT_{s|2},...,\TT_{s|n},\in [0,\delta]]$, we have
\[0\le \TT_{s|j}=\alpha(\TT_{s|j-1}-T_{s|j})\le\alpha(\delta-T_{s|j})\quad\forall j=2,...,n.\]
Thus, $T_{s|2},T_{s|3},...,T_{s|n}\in[0,\delta]$. Because these random variables are i.i.d.\ mean-one exponentially distributed,
\[I_n(a,A)\le\P(T_{s|2},T_{s|3},...,T_{s|n}\in[0,\delta])=\P(T_\theta\in[0,\delta])^{n-1}=(1-e^{-\delta})^{n-1}.\]
Because $\delta\in(0,\ln 2)$, we have $1-e^{-\delta}\in(0,1/2)$. By \cite[Thm 2.1]{part2_2021}, $A$ is cutset recurrent with respect to $\{\TT_v\}_{v\in \T}$ in the following sense: the sequence of random ``passage sets" $\{\Pi_n\}_{n\ge 1}$ defined below are cutsets. 
\begin{align*}
\Pi_1  &=\{v\in\T\backslash\{\theta\}:\ \TT_v\in A,\ \TT_{v|j}\not\in A\text{ for all } 1\le j<|v|\}\\
&=\{v\in\T\backslash\{\theta\}:\ \TT_v\not\in [0,\delta],\ \TT_{v|j}\in [0,\delta]\text{ for all } 1\le j<|v|\}
\end{align*}
and for $n\ge 2$,
\begin{align*}
\Pi_n&=\{v\in\T\backslash\{\theta\}:\ \TT_v\in A,\ \exists 1\le j< |v|,\ v|j\in\Pi_{n-1},\ \TT_{v|i}\not\in A\ \forall\, j<i<|v|\}\\
&=\{v\in\T\backslash\{\theta\}:\ \TT_v\not\in [0,\delta],\ \exists 1\le j< |v|,\ v|j\in\Pi_{n-1},\ \TT_{v|i}\in [0,\delta]\ \forall\, j<i<|v|\}.
\end{align*}
We say that a \emph{vertex $v\in\T$ lies in between two passage sets $\Pi_{n-1}$ and $\Pi_{n}$} if there is a ray $s\in\pT$ and integers $1\le i<j\le k$ such that $s|i\in\Pi_{n-1}$, $s|j=v$, and $s|k\in\Pi_n$.

On the event $[|\pV(t)|=\infty]$, we know by \autoref{R01} that almost surely there are infinitely many vertices $v\in\T$ such that $\TT_v\in[0,\delta]$. Thus, there exists a strictly increasing sequence of indices $1\le n_1<n_2<...$ such that $\TT_{v_m}\in[0,\delta]$ and $v_{m}$ lies between $\Pi_{n_m-1}$ and $\Pi_{n_m}$. There are  integers $1\le i_m<j_m<k_m$ and
a ray $s_m\in\pT$ such that $s_m|i_m\in \Pi_{n_m-1}$, $s_m|j_m=v_{m}$, and $s_m|k_m\in\Pi_{n_m}$. Let $w_m=s_m|(i_m+1)$, which also lies between $\Pi_{n_m-1}$ and $\Pi_{n_m}$. Because $i_m<i_m+1\le j_m<k_m$, we have 
\[\TT_{\p{w}_m}=\TT_{s|i_m}\not\in[0,\delta],\quad \TT_{w_m}=\TT_{s|i_m+1}\in[0,\delta].\]
If $\TT_{\p{w}_m}<0$ then $\TT_{{w}_m}=\alpha(\TT_{\p{w}_m}-T_{w_m})<0$, which is a contradiction. Hence, $\TT_{\p{w}_m}>\delta$. The sequence $\{w_m\}_{m\ge 1}$ is a sequence of distinct vertices with $\TT_{\p{w}_m}>\delta$ and $\TT_{{w}_m}\in[0,\delta]$. Let
\[A_v=\left[0\le \TT_{\p{v}}-\delta<T_v\le \TT_{\p{v}}\right],\quad B_v=\left[0\le \TT_{\p{v}}-\delta\le T_v\le \TT_{\p{v}}-\delta/\alpha\right].\]
Note that $A_v$ occurs if and only if $\TT_{\p{v}}\ge\delta$ and $\TT_v\in[0,\delta\alpha],$ whereas the event $B_v$ occurs if and only if $\TT_{\p{v}}\ge\delta$ and $\TT_v\in[\delta,\delta\alpha].$ In particular \[B_v\subset[\TT_v \in [\delta,\delta \alpha]].\] The existence of the sequence $\{w_m\}$ mentioned earlier shows that $[|\pV(t)|=\infty]\subset [A_v\io]$ almost surely.

Let $\{\F_v\}_{v\in\T}$ be the filtration given by \eqref{filtration}. Since $T_v$ is independent of $\F_{p(v)}$,
\begin{align}
\label{av}\P(A_v|\F_{p(v)})&=\P\left(\TT_{\p{v}}-\delta<T_v\le \TT_{\p{v}}\big|\TT_{\p{v}}\right)\1_{[\TT_{\p{v}}\ge \delta]}=(e^\delta-1)e^{-\TT_{\p{v}}}\1_{[\TT_{\p{v}}\ge \delta]},\\
\label{bv}\P(B_v|\F_{p(v)})&=\P\left(\TT_{\p{v}}-\delta\le T_v<\TT_{\p{v}}-\delta/\alpha\big|\TT_{\p{v}}\right)
\1_{[\TT_{\p{v}}\ge \delta]}=(e^{\delta}-e^{\delta/\alpha})e^{-\TT_{\p{v}}}\1_{\TT_{\p{v}}\ge \delta}.
\end{align}
By the conditional Borel-Cantelli lemma (\autoref{cbc}),
\begin{align}
\label{abv1}
[A_v\io]&=\left[\sum_{v\in\T}e^{-\TT_{\p{v}}}\1_{\TT_{\p{v}}\ge \delta}=\infty\right]=[B_v\io]\text{ a.s.}
\end{align}
Therefore, almost surely, $[|\pV(t)|=\infty]\subset [B_v\io] \subset [\TT_v\in[\delta,\delta\alpha] \io]$ as claimed.
\end{proof}

\begin{lem}\label{R03}
Let $\alpha>1$, $M>0$, $t>0$. Let $k\in\mathbb{N}$ be any number such that $M\alpha^{-k+1}<\ln 2$. Then, almost surely, 
\[[|\pV(t)|=\infty]\subset [\TT_{v*w}\in [M,M\alpha]\ \forall |w|=k, \io].\]
\end{lem}
\begin{proof}
Let $\delta\in(M\alpha^{-k+1},\,\ln 2)$ and $I=[\delta,\delta\alpha]$. Our first step is to show that
\begin{equation}\label{jumpk}
\inf\{\P(\TT_{v*w}\in[M,M\alpha]\text{ for all }|w|=k\,|\,\TT_v=r):\,r\in I,\,v\in\T\}>0.
\end{equation}
Denote $\TT^{(v)}_w=\TT_{v*w}$ and $T^{(v)}_w=T_{v*w}$. For $v,w\in\T$, $|w|=k$, one has
\begin{equation}\label{Rvw}
\TT^{(v)}_{w}=\alpha^k\left(\TT^{(v)}_{\theta}-T^{(v)}_{w|1}-\frac{T^{(v)}_{w|2}}{\alpha}-...-\frac{T^{(v)}_{w|k}}{\alpha^{k-1}}\right).
\end{equation}
Then by the recursion formula \eqref{Rv2},
\[
\TT^{(v)}_{w}=\alpha^k(\TT_v+\alpha T^{(v)}_\theta-\alpha\Theta^{(v)}_w)
\]
where
\[\Theta^{(v)}_{w}=T^{(v)}_\theta+\frac{T^{(v)}_{w|1}}{\alpha}+...+\frac{T^{(v)}_{w|k}}{\alpha^k}=\sum_{j=0}^{k} \frac{T^{(v)}_{w|j}}{\alpha^j}.\]
Because the family $\{T^{(v)}_{w}:\,|w|\le k\}$ has the same distribution as the family $\{T_w:\,|w|\le k\}$, the family $\{\Theta^{(v)}_w:\,|w|=k\}$ has the same distribution as the replacement times $\{\Theta_{w}:\,|w|=k\}$ defined by \eqref{agedef}. Then for any $
r\in I$, by \eqref{Rvw} we have
\begin{align}
\nonumber&\P(\TT^{(v)}_{w}\in[M,M\alpha]\text{ for all }|w|=k\,|\,\TT_v=r)\\
\nonumber&=\P(r+\alpha T^{(v)}_\theta-\alpha\Theta^{(v)}_{w}\in[M\alpha^{-k},M\alpha^{-k+1}]\text{ for all }|w|=k)\\
\nonumber&=\P(r+\alpha T_\theta-\alpha\Theta_{w}\in[M\alpha^{-k},M\alpha^{-k+1}]\text{ for all }|w|=k)\\
\nonumber&=\P(\Theta_{w}\in[ T_\theta +\beta_1(r),\, T_\theta +\beta_2(r)]\text{ for all }|w|=k)\\
\label{820261}&=\P(T_\theta +\beta_1(r)\le S_k\le L_k
\le T_\theta +\beta_2(r))
\end{align}
where 
\begin{align*}
&\beta_1(r)=(r-M\alpha^{-k+1})/\alpha,\quad\beta_2(r)=(r-M\alpha^{-k})/\alpha,\\
&S_k=\min\{\Theta_w:\,|w|=k\},\quad L_k=\max\{\Theta_w:\,|w|=k\}.
\end{align*}
By conditioning on $T_\theta$, one can write \eqref{820261} as 
\begin{align}
\nonumber&\P(\TT^{(v)}_{w}\in[M,M\alpha]\text{ for all }|w|=k\,|\,\TT_v=r)\\
\nonumber&=\int_0^\infty \P( s +\beta_1(r)\le S_k\le L_k
\le s +\beta_2(r))e^{-s}ds\\
\label{820262}&\ge\int_0^1 h(r,s)e^{-s}ds
\end{align}
where $h(r,s)=\P( s +\beta_1(r)\le S_k\le L_k
\le s +\beta_2(r))$, which is a continuous function on the compact set $I\times[0,1]$. Let $h(r_0,s_0)=\min_{I\times[0,1]}h(r,s)\ge 0$. Let $a= s_0+\beta_1(r_0)$ and $b= s_0+\beta_2(r_0)$. Then $0<(\delta-M\alpha^{-k+1})/\alpha\le \beta_1(r_0)\le a<b$ and
\[h(r_0,s_0)=\P(a\le S_k\le L_k\le b).\]
Let $c=\sum_{j=0}^k\alpha^{-j}$. On the event $[a/c\le T_u\le b/c\text{ for all }|u|\le k]$, we have that for any $w\in\T$, $|w|=k$,
\begin{align*}
a=\frac{a}{c}\sum_{j=0}^k\frac{1}{\alpha^j}\le\Theta_w=\sum_{j=0}^k\frac{T_{w|j}}{\alpha^j}\le \frac{b}{c}\sum_{j=0}^k\frac{1}{\alpha^j}=b.
\end{align*}
Hence,
\[h(r_0,s_0)\ge \P(a/c\le T_u\le b/c\text{ for all }|u|\le k)=\P(a/c\le T_\theta\le b/c)^{2^{k+1}-1}=(e^{-a/c}-e^{-b/c})^{2^{k+1}-1}.\]
Combining this estimate with \eqref{820262}, one gets
\begin{align}
\label{820263}&\P(\TT^{(v)}_{w}\in[M,M\alpha]\text{ for all }|w|=k\,|\,\TT_v=r)\ge\int_0^1 h(r_0,s_0)e^{-s}ds=h(r_0,s_0)(1-e^{-1})\ge q,
\end{align}
where
\begin{equation}\label{qlowerb}
q=(e^{-a/c}-e^{-b/c})^{2^{k+1}-1}(1-e^{-1})>0.
\end{equation}
Therefore, \eqref{jumpk} has been proven.

Now let $\{\mathcal{G}_n\}_{n\ge 1}$ be a filtration given by $\mathcal{G}_n=\sigma(T_v: |v|\le n)$. Define recursively a sequence of stopping times $\{\tau_n\}_{n\ge 1}$ as follows:
\begin{align}
\tau_1&=\inf\{j\ge 1:\,\exists w\in\T,\,|w|=j,\,\TT_w\in[\delta,\delta\alpha]\},\\
\tau_n&=\inf\{j\ge \tau_{n-1}+k+1:\,\exists w\in\T,\,|w|=j,\,\TT_w\in[\delta,\delta\alpha]\},\quad n\ge 2.
\end{align}
Each $\tau_n$ is indeed a stopping time with respect to the filtration $\{\mathcal{G}_j\}_{j\ge 1}$ because the event $[\tau_n\le m]$ is $\mathcal{G}_m$-measurable. By \autoref{R02}, almost surely each $\tau_n$ is finite on the event $[|\pV(t)|=\infty]$. 
Recall that $(\Omega,\F,\P)$ is the probability space upon which the family $\{T_v\}_{v\in\T}$ is defined. Let 
$$\H_{n}=\mathcal{G}_{\tau_n}:=\{A\in\F:\,A\cap[\tau_n\le m]\in\mathcal{G}_m\text{ for all }m\in\mathbb{N}\}$$ 
be the $\sigma$-algebra at the stopping time $\tau_n$. 
Since the sequence of stopping times $\{\tau_n\}_{n\ge 1}$ is increasing, the family $\{\H_n\}_{n\ge 1}$ is a filtration \cite[p.\ 173]{gray}.

On the event $[|\pV(t)|=\infty]$, let $v_n\in\T$ be the minimal vertex of $\T$ in the lexicographical order such that $|v_n|=\tau_n$ and $\TT_{v_n}\in[\delta,\delta\alpha]$ for every $n\in\mathbb{N}$. Then $\TT_{v_n}$ is $\H_n$-measurable. For $n\ge 2$, denote
\[A_n=[\tau_{n-1}<\infty]\cap[\TT_{v_{n-1}*w}\in[M,M\alpha]\text{ for all }w\in\T, |w|=k].\]
Then $A_n\in\mathcal{G}_{\tau_{n-1}+k+1}\subset\mathcal{G}_{\tau_n}=\H_n$. Moreover, the family $\{T_{v_{n-1}*w}:1\le|w|\le k\}$ is independent of $\H_{n-1}$. Therefore,
\[\P(A_n|\H_{n-1})=\P(\TT_{v_{n-1}*w}\in[M,M\alpha]\text{ for all }w\in\T,|w|=k\,|\, \TT_{v_{n-1}})\1_{[\tau_{n-1}<\infty]}.\]
Because $\TT_{v_{n-1}}\in[\delta,\delta\alpha]$, together with \eqref{820263} we get $\P(A_n|\H_{n-1})\ge q\1_{[\tau_{n-1}<\infty]}$ for every $n\ge 2$, where $q>0$ is the constant given by \eqref{qlowerb}. Consequently,
\[[|\pV(t)|=\infty]\subset\left[\sum_{n=2}^\infty\P(A_n|\H_{n-1})=\infty\right]\text{ a.s.}\]
Together with the conditional Borel-Cantelli lemma, we conclude that almost surely,
\[[|\pV(t)|=\infty]\subset[A_n\io] = [\TT_{v_{n-1}*w}\in[M,M\alpha]\ \forall |w|=k, \io]\]
as claimed.
\end{proof}

\begin{proof}[\textbf{Completing the proof of \autoref{X>0}}]
We now complete the proof that $[|\partial V(t)|=\infty]\subset [X(t)=0]$ almost surely. Fix $t>0$ and define $\{\TT_v\}_{v\in\T}$ be the family defined by \eqref{Rv2}. By \eqref{eq6211},
\begin{equation}
X_{n}(t)
= \prod_{v\in \T,\ |v|=n-1}\rho^2(\alpha^n(t-\Theta_v))=\prod_{v\in \T,\ |v|=n-1}\rho^2(\TT_v).
\end{equation}
Fix $k\in\mathbb{N}$ such that $M\alpha^{-k+1}<\ln 2$. By \autoref{R03}, on the event $[|\pV(t)|=\infty]$, almost surely there exists a sequence $\{v_{n}\}_{n\ge 1}$ such that $m_{n}=|v_{n}|$ is strictly increasing and $\TT_{v_n*w}\in[M,M\alpha]$ for any $w\in\T$, $|w|=k$. Then
\begin{equation}\label{Xmnk}
X_{m_n+k+1}(t)=\prod_{v\in \T,\ |v|=m_n+k}\rho^2(\TT_v)\le \prod_{w\in\T,|w|=k}\rho^2(\TT_{v_n*w}).
\end{equation}
By the definition of $\rho$ in \autoref{72191}, $0<\rho(t)<1-\epsilon$ for all $t\in[M,M\alpha]$, where $\epsilon=\delta(M\alpha)^{-\gamma}\in(0,1)$. With this estimate, one can infer from \eqref{Xmnk} that
\[X_{m_n+k+1}(t)\le \prod_{w\in\T,\,|w|=k}(1-\epsilon)^2=(1-\epsilon)^{2^{k+1}}.\]
Thus, $0\le X(t)=\lim_{n\to\infty} X_{m_n+k+1}(t)\le (1-\epsilon)^{2^{k+1}}$. Because $k\in\mathbb{N}$ was chosen arbitrarily as long as $M\alpha^{-k+1}<\ln 2$, we must have $X(t)=0$.
\end{proof}

\subsection{Proof of \autoref{main3}}

\begin{proof}[{\textbf{Proof of Part (i)}
}]

Fix $u_0\in D_\alpha$. In \autoref{generalinitialdata}, we constructed a family of solutions $\{u_\lambda\}_{\lambda>0}$ to the problem \eqref{diffalpha-ricc0} as $u_\lambda(t)=\EXP[\uprX_{u_0}(t)\ZZ^\lambda(t)]$,
where $\uprX_{u_0}$ is given by \eqref{max_sol_proc} and $\ZZ$ is given by \eqref{Zproc}. 
  Note that because $\ZZ$  is a solution process to \eqref{alphrecursion} with $u_0$ being replaced by 1, so are the processes $\1_{[\ZZ(t)>0]}$ and $\1_{[\ZZ(t)=1]}$. Thus, the processes 
\begin{align}\label{Xu0*}
X_{u_0,*}(t):=\uprX_{u_0}(t)\IND{[\ZZ(t)>0]},\\
\label{Xu0inf}
X_{u_0,\infty}(t):=\uprX_{u_0}(t)\IND{[\ZZ(t)=1]}
\end{align}
are solution processes to \eqref{alphrecursion} as well. 
Moreover, by \eqref{wideu},
\begin{equation}\label{u*upper}
\EXP\uprX_{u_0}(t)\le 1+ce^{-t}\quad\forall t>0.\end{equation}
By Lebesgue's Monotone Convergence Theorem, the following limits are finite:
\begin{align}\label{ustar}
u_*(t)&:=\lim_{\lambda\to 0}u_\lambda(t)=\EXP[\uprX_{u_0}(t)\IND{[\ZZ(t)>0]}]=\EXP[X_{u_0,*}(t)],\\
\label{uinfty}u_\infty(t)&:=\lim_{\lambda\to\infty }u_\lambda(t)=\EXP[\uprX_{u_0}(t)\IND{[\ZZ(t)=1]}]=\EXP[X_{u_0,\infty}(t)].
\end{align}
In the integral form, $u_\lambda$ satisfies
\begin{equation}\label{ulam}
u_\lambda(t) = e^{-t}u_0+ \int_0^t e^{-s}u^2_\lambda(\alpha(t-s))ds\,,
\quad  t\ge 0.
\end{equation}
We also showed in \autoref{main} (i) that for each $t>0$, $u_\lambda(t)$ is strictly decreasing with respect to $\lambda>0$.  By taking the limit of both sides of \eqref{ulam} as $\lambda\to \infty$ or as $\lambda\to 0$ and using Lebesgue's Monotone Convergence Theorem, we conclude that $u_\infty$ and $u_*$ satisfy \eqref{mildalpha} as claimed.
\end{proof}

\begin{proof}[\textbf{Proof of Part (ii)}]
In light of \eqref{uinfty}, the proof of this part of \autoref{main3} will be completed once we show that almost surely,
\begin{equation}
    \label{S=Z}
    [S\le t]=[\ZZ(t)<1].
\end{equation}
Indeed, assuming this equality holds, then using 
\eqref{uinfty}, one has
\[u_\infty(t)=\EXP[\uprX_{u_0}(t)\1_{[\ZZ(t)=1]}]=\EXP[\uprX_{u_0}(t)\1_{[S>t]}]=\EXP[u_0^{N(t)}\1_{[S>t]}]=\underline{u}(t)\]
as claimed.

To establish equality \eqref{S=Z} we proceed as follows.
For $n\ge 1$, let
\begin{equation}\label{Znproc}
Z_n(t)=\left\{ \begin{array}{*{35}{r}}
   (X_{M,\delta,n}(t))^{1/\delta}    &\quad\ \textup{if} & 1<\alpha \le 2,  \\
   e^{-\sX_n(t)} & \textup{if} & \alpha >2  \\
\end{array} \right.\end{equation}
where $X_{M,\delta,n}$ and $\sX$ are the stochastic processes defined by \eqref{eq6211} and \eqref{sX_fla}, respectively. Then for each $t>0$, $\ZZ(t)=\lim Z_n(t)$ almost surely according to \autoref{non-indicator} (i) and \autoref{martingalep} (iii).

On the event $[S>t]$ of non-explosion by time $t$, $Z_n(t)=1$ for sufficiently large $n$ and thus, $\ZZ(t)=1$. This implies $[S>t]\subset [\ZZ(t)=1]$, or equivalently $[\ZZ(t)<1]\subset [S\le t]$.

Next, we show that $[S\le t]\setminus [\ZZ(t)<1]$ is a null event. Since $[S=t]$ is a null event (\autoref{contdist}) and $[S<t]=\cup_{k=1}^\infty[S\le t-1/k]$, it suffices to show that for each $\epsilon\in(0,t)$, $[S\le t-\epsilon]\subset [\ZZ(t)<1]$ almost surely. Fix $\epsilon\in(0,t)$. By \autoref{hypersubprop}, on the event $[S\le t-\epsilon]$, there exists a.s.\ a $(t-\epsilon)$-hyperexplosive subtree, which we denote by $\T_{v_0}$. Then $L_{v_0}\le t-\epsilon$. 
For $w\in\T_{v_0}$ with $|w|=n-1>|v_0|$, we have $\Theta_{\p{w}}+L_w\le t-\epsilon$ and thus,
\[\frac{T_w}{\alpha^{|w|}}+\epsilon<L_w+\epsilon\le t-\Theta_{\p{w}}.\]
This implies $\epsilon<t-(\Theta_{\p{w}}+\alpha^{-|w|}T_w)=t-\Theta_w$. 
We now consider two cases: $\alpha\in(1,2]$ and $\alpha\in(2,\infty)$.\\
$\bullet$ The case $\alpha\in(1,2]$: For $n>|v_0|+1$, we have
\begin{align}
\label{Xnupbound}{{X}_{M,\delta,n}}(t)\1_{[S\le t-\epsilon]}=\1_{[S\le t-\epsilon]}\prod\limits_{|w|=n-1}{{{\rho }^{2}}({{\alpha }^{n}}(t-{{\Theta }_{w}}))}\le \1_{[S\le t-\epsilon]}\prod\limits_{\begin{smallmatrix} 
 |w|=n-1 \\ 
 w\in {{\T}_{{{v}_0}}} 
\end{smallmatrix}}{{{\rho }^{2}}({{\alpha }^{n}}(t-{{\Theta }_{w}}))}.
\end{align}
For $n>\max\{|v_0|+1,\,\log_\alpha(M/\epsilon)\}$ and $w\in\T_{v_0}$ with $|w|=n-1$, we have 
$\alpha^n(t-\Theta_w)\ge \alpha^n\epsilon>M.$ By its definition in \autoref{72191}, $\rho$ is increasing on $(M,\infty)$, which implies
\[\rho(\alpha^n(t-\Theta_w))\le \rho(\alpha^n t)=1-\delta t^{-\gamma}2^{-n}.\]
Substituting this estimate into \eqref{Xnupbound}, we get
\[X_{M,\delta,,n}(t)\1_{[S\le t-\epsilon]}\le \1_{[S\le t-\epsilon]}(1-\delta t^{-\gamma}2^{-n})^{2^{n-|v_0|}}.\]
Letting $n\to\infty$, we obtain
\[(\ZZ(t))^{\delta}\1_{[S\le t-\epsilon]}\le \1_{[S\le t-\epsilon]}e^{-\delta t^{-\gamma}2^{-|v_0|}}.\]
Therefore, $\ZZ(t)\le e^{-t^{-\gamma}2^{-|v_0|}}<1$ almost surely on the event $[S\le t-\epsilon]$. In other words, $[S\le t-\epsilon]\subset[\ZZ(t)<1]$ almost surely.\\
$\bullet$ The case $\alpha>2$: By \autoref{prop11112} (iii), there exists $t_0>0$ such that $\eta(t)>\frac{1}{2}t^{-\gamma}$ for all $t>t_0$. For $n>\max\{|v_0|+1,\,\log_\alpha(t_0/\epsilon)\}$ and $w\in\T_{v_0}$ with $|w|=n-1$, we have 
$\alpha^n(t-\Theta_w)\ge \alpha^n\epsilon>t_0.$ By \eqref{sX_fla},
\begin{align}
\nonumber
{\sX_{n}}(t)\1_{[S\le t-\epsilon]}&\ge \1_{[S\le t-\epsilon]}\sum\limits_{\begin{smallmatrix} 
 |w|=n-1 \\ 
 w\in {{\T}_{{{v}_0}}} 
\end{smallmatrix}}2{{{\eta }}({{\alpha }^{n}}(t-{{\Theta }_{w}}))}\ge \1_{[S\le t-\epsilon]}\sum\limits_{\begin{smallmatrix} 
 |w|=n-1 \\ 
 w\in {{\T}_{{{v}_0}}} 
\end{smallmatrix}}({{\alpha }^{n}}(t-{{\Theta }_{w}}))^{-\gamma}\\
&\ge \1_{[S\le t-\epsilon]} 2^{n-1-|v_0|}(\alpha^n t)^{-\gamma}=\1_{[S\le t-\epsilon]} 2^{-|v_0|-1}t^{-\gamma}.
\end{align}
Taking the limit as $n\to\infty$, one gets $\sX(t)\ge 2^{-|v_0|-1}t^{-\gamma}>0$ on the event $[S\le t-\epsilon]$. Thus, $\ZZ(t)=e^{-\sX(t)}<1$ on the event $[S\le t-\epsilon]$.
\end{proof}

\begin{proof}[\textbf{\textit{Proof of Part (iii)}}]
Let $\alpha>1$ and $u_0\in D_\alpha$. It is clear that
\begin{equation}\label{ustarvsubar}
u_*(t)=\EXP[\uprX_{u_0}(t)\IND{[\ZZ(t)>0]}]\le \EXP\uprX_{u_0}(t)=\widebar{u}(t)\quad\forall t\ge 0.\end{equation}
$\bullet$ Consider the case $\alpha>2$. Then $\ZZ(t)=e^{-\sX(t)}>0$ a.s.\ according to \eqref{Zproc}. Then \eqref{ustarvsubar} becomes an equality.\\
$\bullet$ Consider the case $\alpha\in (1,2]$. Then $\ZZ(t)=(X_{M,\delta}(t))^{1/\delta}$ for some $M,\delta>0$, where $X_{M,\delta}$ is defined by \autoref{non-indicator} (i). By \autoref{X>0}, one has the inclusion
\begin{equation}\label{Ninc}
[N(t)<\infty]\subset [\ZZ(t)>0]\end{equation}
where $N(t)=|\partial V(t)|$. Note that $D_\alpha=[0,1]$. If $u_0\in[0,1)$ then
\[u_*(t)=\EXP[u_0^{N(t)}\1_{[\ZZ(t)>0]}]\ge \EXP[u_0^{N(t)}\1_{[N(t)<\infty]}]=\EXP[u_0^{N(t)}]=\widebar{u}(t).\]
Consider the case $u_0=1$ and $\alpha=2$. By \autoref{alphexplosion} (i), $N(t)<\infty$ a.s. Together with \eqref{Ninc}, one has $\ZZ(t)>0$ a.s. Therefore, \eqref{ustarvsubar} becomes an equality.

Consider the remaining case when $u_0=1$ and $\alpha\in(1,2)$. We have
\[u_*(t)=\EXP[\uprX_{u_0}(t)\IND{[X(t)>0]}]=\EXP[\IND{[X(t)>0]}]=\P(X(t)>0).\]
By \autoref{X>0}, the right hand side is exactly $\P(|\pV(t)|<\infty)=\tilde{u}(t)$.
\end{proof}

\noindent
\begin{proof}[{\textbf{\textit{Proof of Part (iv)}}}]
Let $\alpha>1$ and $u_0\in D_\alpha$. We have shown in Part (i) of \autoref{main3} that $u_\infty=\underline{u}$. By \cite[Remark 3.3]{alphariccati}, there exists a constant $c>0$ such that $\underline{u}(t)\le ce^{-t}$ for all $t\ge 0$. Thus, $e^t\underline{u}$ is bounded on $[0,\infty)$. Since $\underline{u}$ satisfies the integral equation \eqref{mildalpha}, $e^t\underline{u}$ is a non-decreasing function. This implies that the limit $\lim_{t\to\infty} e^t\underline{u}(t)$ exists and is finite. The limit is zero if and only if it $e^t\underline{u}\equiv 0$, which implies that $\underline{u}\equiv 0$ and $u_0=0$.

Recall that $u_*$ is given by \eqref{ustar}.  By \eqref{max_sol_proc}, $\uprX_{u_0}(t)=1$ on the event $[L<t]$. By \autoref{non-indicator} (i) and \autoref{martingalep} (iii), $\ZZ(t)>0$ on the event $[L<t]$. Thus,
\[u_*(t)\ge\EXP[\uprX_{u_0}(t)\IND{[\ZZ(t)>0]}\1_{[L<t]}]=\EXP\1_{[L<t]}=\P(L<t).\]
Together with \eqref{u*upper}, one gets
\begin{equation}\label{qtrap}
-ce^{-t}\le 1-u_*(t)\le\P(L\ge t)\quad\forall t>0.\end{equation}
Since the right hand side converges to 0 exponentially fast as $t\to\infty$ (\autoref{expl_asymptotics}), there exists a constant $c_1>0$ such that $|1-u_*(t)|\le c_1 e^{-t}$ for all $t\ge 0$. In particular, $u_*(t)\le 1+c_1$ for all $t\ge 0$. Consequently,
\[|e^t(1-u_*^2(\alpha t))|\le e^t|1-u_*(\alpha t)|(1+u_*(\alpha t))\le (2+c_1)e^t|1-u_*(\alpha t)|\le (2+c_1)c_1 e^t e^{-\alpha t}.\]
Thus, $e^t (1-u_*^2(\alpha t))\in L^1(0,\infty)$. Now that $u_*$ satisfies the integral equation \eqref{mildalpha}, we have
\[e^t (1-u_*(t))=1-u_0+\int_0^t e^s(1-u^2_*(\alpha s))ds.\]
Therefore,
\[\lim_{t\to\infty}e^t (1-u_*(t))=1-u_0+\int_0^\infty e^s(1-u^2_*(\alpha s))ds\in\R.\]
If the limit is zero then $u_*\equiv 1$ according to \autoref{alphasharprmk}. By \autoref{main3} (iii), this happens if and only if $u_0=1$ and $\alpha\in[2,\infty)$.
\end{proof}

\begin{proof}[{\textbf{\textit{Proof of Part (v)}}}]

We show that the solution map $h$ defined by \eqref{homotopy} is continuous on $Q_\alpha=D_\alpha\times[0,\infty]\times[0,\infty)$, where the topology on $[0,\infty]$ is understood as the one-point compactification of the Euclidean topology on $[0,\infty)$. We have 
\begin{equation}\label{homotopy2}h(b,\lambda,t)=\EXP[b^{N(t)}\ZZ^\lambda(t)] \qquad\forall(b,\lambda,t)\in Q_\alpha
\end{equation}
where $N(t)=|\pV(t)|$ and $\{\ZZ(t)\}_{t\ge 0}\subset[0,1]$ is the stochastic process defined by \eqref{Zproc}. Here, we use the convention $0^0=1$ for $b^{N(t)}$, and the convention $0^0=0$ for $\ZZ^\lambda(t)$. By \eqref{diffalpha-ricc0}, we have
\begin{equation}\label{heq}
\partial_t h(b,\lambda,t)=h^2(b,\lambda,\alpha t)-h(b,\lambda,t)\qquad \forall (b,\lambda,t)\in Q_\alpha.
\end{equation}
By \eqref{wideu}, we have
\begin{equation}\label{hbound}
h(b,\lambda,t)\le \EXP[b^{N(t)}]=\widebar{u}(b;t)\le 1+\frac{2\alpha-1}{4}C_0\, e^{-t}\le C_1\qquad \forall (b,\lambda,t)\in Q_\alpha
\end{equation}
where $C_0\approx 1.4739$ is given by \eqref{constC0} and $C_1=1+\frac{2\alpha-1}{4}C_0$.

Let $(b_n,\lambda_n,t_n)$ be a sequence in $Q_\alpha$ that converges to $(b_0,\lambda_0,t_0)\in Q_\alpha$. We decompose $h(b_n,\lambda_n,t_n)-h(b_0,\lambda_0,t_0)=p_n+q_n$, where
\begin{equation}
p_n=h(b_n,\lambda_n,t_n)-h(b_n,\lambda_n,t_0),\qquad q_n=h(b_n,\lambda_n,t_0)-h(b_0,\lambda_0,t_0).
\end{equation}
By \eqref{heq} and \eqref{hbound}, we have $|\partial_t h(b_n,\lambda_n,t)|\le C_1^2+C_1$ for all $n$. Thus,
\[|p_n|=\left|\int_{t_0}^{t_n}\partial_t h(b_n,\lambda_n,t)dt\right|\le (C_1^2+C_1)|t_n-t_0|.\]
Thus, $\lim p_n=0$. It remains to show that $\lim q_n=0$. Recall that if $\alpha>2$ then $\ZZ(t)=e^{-\sX(t)}>0$ a.s.\ according to \autoref{martingalep} (iii), and $N(t)<\infty$ a.s.\ according to \autoref{alphexplosion}. If $\alpha\in(1,2]$ then $[\ZZ(t)=0]=[N(t)=\infty]$ a.s.\ according to \autoref{X>0}. In both cases, we have $[\ZZ(t)=0]=[N(t)=\infty]$ a.s. Therefore, almost surely
\begin{equation}\label{hlim}
\lim_{n\to\infty}b_n^{N(t_0)}(\ZZ(t_0))^{\lambda_n}=b_0^{N(t_0)}(\ZZ(t_0))^{\lambda_0}.
\end{equation}
Consider two cases: $(b_0,\alpha)\in [0,1]\times (1,\alpha_0]$ and $(b_0,\alpha)\in [0,\Phi(\alpha))\times (\alpha_0,\infty)$, where $\Phi$ and $\alpha_0$ are given in the definition of $D_\alpha$ at \eqref{R_alpha}.\\
$\bullet$ The case $(b_0,\alpha)\in [0,1]\times (1,\alpha_0]$: in this case, $b_n\in[0,1]$ for all sufficiently large $n$. Thus, $0\le b_n^{N(t_0)}(\ZZ(t_0))^{\lambda_n}\le 1$ for all sufficiently large $n$. By \eqref{hlim} and Lebesgue's Dominated Convergence Theorem, we have
\[\lim_{n\to\infty} h(b_n,\lambda_n,t_0)=\lim_{n\to\infty}\EXP[b_n^{N(t_0)}(\ZZ(t_0))^{\lambda_n}]=\EXP[b_0^{N(t_0)}(\ZZ(t_0))^{\lambda_0}]=h(b_0,\lambda_0,t_0).\]
$\bullet$ The case $(b_0,\alpha)\in [0,\Phi(\alpha))\times (\alpha_0,\infty)$: in this case, one can assume that $b_n<\tilde{b}=\frac{1}{2}(b_0+\Phi(\alpha))<\Phi(\alpha)$ for all $n$. Then
\[0\le b_n^{N(t_0)}(\ZZ(t_0))^{\lambda_n}\le \tilde{b}^{N(t_0)}\qquad \forall n.\]
Moreover, $\EXP[\tilde{b}^{N(t_0)}]<\infty$ according to \eqref{hbound}. Then by Lebesgue's Dominated Convergence Theorem, we have $\lim_{n\to\infty} h(b_n,\lambda_n,t_0)=h(b_0,\lambda_0,t_0)$.
\end{proof}

\section{Stochastic Picard iterations with other ground states}
\label{Sec7}
In \autoref{Sec6}, the choice of the ground state $X_0$ in the stochastic Picard iteration \eqref{X_n-eq} has proven to be crucial for the construction of a non-indicator solution process to \eqref{INDrecur}. For $\alpha\in(1,2]$, we chose $X_0(t)=\rho_{M,\delta}(t)$, where $\rho_{M,\delta}$ is the function given by \eqref{eq:72191}. For $\alpha\in(2,\infty)$, we chose $X_0(t)=e^{-\eta(t)}$, where $\eta$ is the function given by \autoref{prop11112} with $a=2$. In this section, we consider two additional types of ground states and their applications. They are the Fr\'{e}chet cumulative distribution function and the constant functions.

\subsection{An expanded approach to Athreya's solution.}\label{lastcase}

In the case $u_0=0$, \eqref{diffalpha-ricc0} has a minimal solution $\underline{u}\equiv 0$ and a maximal solution 
$\overline{u}(t)=\EXP\overline{X}(t)=\P(L\le t)$. Athreya \cite{athreya}*{Thm 2} 
derives a third solution to \eqref{diffalpha-ricc0} with $u_0=0$ by exploiting max-stable properties of the Fr\'{e}chet distribution $F_\gamma$ defined by
\begin{equation}
\label{mugroundstate}
F_\gamma(t)=e^{-t^{-\gamma}}\IND{(0,\infty)}(t),\qquad\gamma=\log_\alpha 2.
\end{equation}
More specifically, he uses the fact that  $F'_\gamma+F_\gamma-F_\gamma^2(\alpha t)>0$ to construct a monotone sequence of \emph{deterministic} Picard iteration \[u_1(t)=F_\gamma(t),\qquad u_{n+1}(t)=\int_0^t e^{-(t-s)}u^2_n(\alpha s)ds.\]
The solution $u$ obtained as the limit of this sequence was shown to satisfy 
$\liminf_{t\to\infty} t^\gamma(1-u(t))>0$. 
Below, we show that it is possible to extend this approach and generate a \emph{one-parameter} family of solutions to  \eqref{diffalpha-ricc0} with initial data $u_0=0$.  This extension uses the ground state $X_0=F_\gamma$. Details are as follows.
\begin{prop}
Assume $u_0=0$, $\alpha>1$ and that $F_\gamma(t)$ is given by \eqref{mugroundstate}.  Then
the sequence of stochastic processes
\[X_n(t)
\stackrel{\rm a.s.}{=} 
\prod_{|v|=n-1}F_\gamma^2(\alpha^n(t-\Theta_v)),\ \ \ \ \forall\,n\ge 1.\]
converges almost surely to a solution process $X(t)$ as $n\to\infty$ of the $\alpha$-Riccati equation. Moreover, $u_\lambda(t)=\EXP[X(t)^\lambda]$ 
satisfies \eqref{diffalpha-ricc0} and
\[\lim_{t\to\infty}\frac{1-u_\lambda(t)}{t^{-\gamma}}={\lambda}.\]
\end{prop}
\begin{proof}
Following the approach from the proof of \autoref{martingale-a} one can rewrite $X_n(t)$ as 
\begin{eqnarray*}X_n(t)&=&\prod_{|v|=n-1}\exp\left(-2(\alpha^n(t-\Theta_v))^{-\gamma}\right)\IND{[L_n<t]}\\
&=&\prod_{|v|=n-1}\exp\left(-\frac{(t-\Theta_v)^{-\gamma}}{2^{n-1}}\right)\IND{[L_n<t]}=e^{-M_n(t)}\IND{[L_n<t]},
\end{eqnarray*}
where \[M_n(t)=\frac{1}{2^{n-1}}\sum_{|v|=n-1}(t-\Theta_v)^{-\gamma}, ~~~L_n = \sup_{|v|=n} \sum_{j=0}^n \frac{T_{v|j}}{\alpha^j}.\]
On the event $[L>t]$, $X_n(t)=0$ for sufficiently large $n$. On the event $[L\le t]$, 
\begin{eqnarray*}
{{M}_{n+1}}(t)&=&\frac{1}{{{2}^{n}}}\sum\limits_{|w|=n}{{{(t-{{\Theta }_{w}})}^{-\gamma }}}\\
&=&\frac{1}{{{2}^{n}}}\sum\limits_{|v|=n-1}{\left( {{(t-{{\Theta }_{v}}-{{\alpha }^{-n}}{{T}_{v1}})}^{-\gamma }}
+{{(t-{{\Theta }_{v}}-{{\alpha }^{-n}}{{T}_{v2}})}^{-\gamma }} \right)}\\
&>&\frac{1}{{{2}^{n}}}\sum\limits_{|v|=n-1}{\left( {{(t-{{\Theta }_{v}})}^{-\gamma }}+{{(t-{{\Theta }_{v}})}^{-\gamma }} \right)}\\
&=&\frac{1}{{{2}^{n-1}}}\sum\limits_{|v|=n-1}{{{(t-{{\Theta }_{v}})}^{-\gamma }}}={{M}_{n}}(t).
\end{eqnarray*}
Thus, the sequence $M_n(t)$ is increasing on the event $[L\le t]$. Thus,  a limit $M(t)=\lim_{n\to\infty} M_n(t)$
exists almost surely. Then, $X_n(t)$ converges almost surely to $X(t)=e^{-M(t)}\IND{[L<t]}$. Note that
\[\frac{1}{2^{n-1}}\sum_{|v|=n-1}(t-S_n)^{-\gamma}\le M_n(t)\le\frac{1}{2^{n-1}}\sum_{|v|=n-1}(t-L_n)^{-\gamma}.\]
In other words,
\[{(t-S_n)^{-\gamma}}\le M_n(t)\le {(t-L_n)^{-\gamma}}\]
Thus,
\begin{equation}\label{rate}{(t-S)^{-\gamma}}\IND{[L<t]}\le M(t)\IND{[L<t]}\le {(t-L)^{-\gamma}}\IND{[L<t]}.\end{equation}
Consequently, $0<X(t)<1$ on the event $[L<t]$. As in the proof of \autoref{martingale-a}, we have that $X_n(t)$ are 
stochastic Picard iterations of $X_0(t)=F_\gamma(t)$, i.e.
\[{X_{n}}(t)=\left\{ \begin{array}{*{35}{l}}
   0 & \text{if} & T_\theta\ge t,  \\
    X_{n-1}^{(1)}(\alpha (t-T_\theta)) X_{n-1}^{(2)}(\alpha (t-T_\theta)) & \text{if} & T_\theta<t.  
\end{array} \right.\]
where $ X_{n-1}^{(1)}$ and $ X_{n-1}^{(2)}$ are conditionally on $T_\theta$ i.i.d.\ copies of $ X_{n-1}$. Thus,
\[{X}(t)=\left\{ \begin{array}{*{35}{l}}
   0 & \text{if} & T_\theta\ge t,  \\
    X^{(1)}(\alpha (t-T_\theta)) X^{(2)}(\alpha (t-T_\theta)) & \text{if} & T_\theta<t.  
\end{array} \right.\]
where $ X^{(1)}$ and $ X^{(2)}$ are i.i.d.\ copies of $ X$. For each $\lambda>0$, $u_\lambda(t)=\EXP[X(t)^\lambda]$ 
solves \eqref{diffalpha-ricc0}. From here, one can use the same estimating technique used in the proof 
of \autoref{non-indicator} (ii) to show that
\[\lim_{t\to\infty}\frac{1-u_\lambda(t)}{t^{-\gamma}}={\lambda}.\]
\end{proof}

\subsection{Stochastic Picard iterations with constant ground state}\label{groundstatesec}

In the case $\alpha\in(0,1]$, the cascade is nonexplosive and ground states are irrelevant. 
In the case $\alpha>1$ (explosive case), different solution processes 
arise even in the simple case that the ground state is constant. 
To describe the results, let us consider an alternative representation
of $\partial\T$. 

Identify $\partial\T=\{1,2\}^\infty$ with points in the interval $[0,1]$ 
under dyadic expansion. 
In particular, a ray $s=(s_1,s_2,\dots)\in\partial\T$ corresponds to 
$x_s=\sum_{j=1}^\infty(s_j-1)2^{-j}\in[0,1]$. We say that \emph{a ray $s\in\pT$ passes through a $t$-hyperexplosive subtree} if there exists a vertex $v$ on $s$ such that the subtree $\T_v$ is hyperexplosive subtree by time $t$ (or equivalently, if the $(t,v)$-hyperexplosion event occurs; see \autoref{mhs}).
For any $v\in\T$, the set of rays passing through $v$ corresponds 
to the subinterval 
$$J_v=[x_v,\,x_v+2^{-|v|}], \quad x_v=\sum_{j=1}^{|v|}(v_j-1)2^{-j}.$$

\begin{lem}\label{counting-vertices}
For $t>0$ let 
\[M_n(t) =|\{v\in V(t): |v|=n\}|,\qquad N_n(t) = |\{v\in \partial V(t): |v|< n\}|.\]
Denote
$$\mathcal{H}_t = \{x_s\in[0,1]: 
\text{ $s\in\pT$ passes through a $t$-hyperexplosive subtree}\}.$$
Then
\begin{equation}\label{lim-bounds}
\lim\limits_{n\to\infty}\frac{N_n(t)}{2^n}=0\quad\text{and}\quad
\liminf\limits_{n\to\infty}\frac{M_n(t)}{2^n}\ge|\mathcal{H}_t|,
\end{equation}
where $|\mathcal{H}_t|$ is the Lebesgue measure of 
$\mathcal{H}_t$.
\end{lem}
\begin{proof} 
On the non-explosion event $[S>t]$, we have $\mathcal{H}_t=\emptyset$ and  
$M_n(t)=0$, $N_n(t)=|\partial V(t)|<\infty$ for sufficiently large $n$, 
so \eqref{lim-bounds} holds. On the hyperexplosion event $[L\le t]$, we have $\mathcal{H}_t=[0,1]$ and $|M_n(t)|=2^n$, $N_n(t)=0$, so \eqref{lim-bounds} also holds. 

Now consider the event $[S\le t<L]$. On this event, $M_n(t)\ge 1$ for all 
$n\in\mathbb{N}$, and by \autoref{hypersubprop}, $V(t)$ contains maximal hypexplosive
subtrees. By \autoref{mhs}, at generation $n>|v|$, 
a maximal hyperexplosive subtree rooted at $v$ contributes 
$2^{n-|v|}=|J_v|2^n$ vertices to the count $M_n(t)$, where $|J|$ denotes the Lebesgue measure of $J\subset [0,1]$. 
Notice that the roots of non-intersecting hyperexplosive 
subtrees correspond to distinct dyadic rationals, so the intervals 
corresponding to their rays have non-intersecting interiors. 
Therefore, $|\mathcal{H}_t|=\sum|J_v|$, where the sum is
taken over all $v\in\T$ that are the roots of distinct
maximal hyperexplosive subtrees. 
By Proposition \ref{hypersubprop}, almost surely every exploding
tree has a hyperexplosive subtree, which implies $|\mathcal{H}_t|>0$. For a fixed $\epsilon>0$, there exist vertices $w_1,w_2,\dots, w_{m_\epsilon}$, arranged so that $|w_i|\le|w_{i+1}|$ for all $i$, that are roots of maximal hyperexplosive subtrees such that
$$|\mathcal{H}_t\backslash\cup_{i=1}^{m_\epsilon}J_i| < \epsilon, 
\quad 
\cup_{i=1}^{m_\epsilon}J_{w_i}\subset \mathcal{H}_t.$$
If $n>m_\epsilon$,  these hyperexplosive subtrees contribute 
$|\cup_{i=1}^{m_\epsilon}J_{w_i}|2^n$ vertices to the count $M_n(t)$. Thus,
\begin{equation}\label{M-count}M_n(t)\ge (|\mathcal{H}_t|-\epsilon)2^n.\end{equation}

Next, consider the $t$-leaf count.  Each $t$-leaf $v\in \partial V(t)$ corresponds to the dyadic interval
$J_v = [x_v, x_v+2^{-|v|}]$, so that distinct $t$-leaves correspond
to intervals having non-overlapping interiors.  In this case, each $J_v$ contributes {\it only once} to the count 
$N_n$, provided $n> |v|$. Let,
$$\mathcal{L}_t = \cup_{v\in \partial V(t)}J_v.$$
That is, $\mathcal{L}_t$ is the subset of $[0,1]$ corresponding to the rays 
that pass through the $t$-leaves. On the event $[S\le t<L]$, there exists at least one $t$-leaf and $|\mathcal{L}_t|=\sum_{v\in \partial V}|J_v|>0$.
Note that $\mathcal{L}_t$ and $\mathcal{H}_t$ have disjoint interiors. 
By a similar argument as before now applied to , for the $\epsilon>0$ there 
exist $t$-leaves $v_1,\dots, v_{\ell_\epsilon}$, arranged so that $|v_i|\le|v_{i+1}|$ for all $i$, such that 
$$\left|\mathcal{L}_t\backslash\cup_{i=1}^{\ell_\epsilon}J_{v_i}\right|=|\mathcal{L}_t|-\sum_{i=1}^{\ell_\epsilon}|J_{v_i}|<\epsilon,\qquad
\cup_{i=1}^{\ell_\epsilon}J_{v_i}\subset\mathcal{L}_t.$$
For any $m>|v_{\ell_\epsilon}|$, the tree has at least $\ell_\epsilon$ 
$t$-leaves and at most $\epsilon 2^m$ $t$-leaves at generation $m$. Thus, for $n>|v_{\ell_\epsilon}|+1$,
\begin{equation}\label{N-count}\ell_\epsilon\le N_n(t) 
\le \ell_\epsilon+\sum_{m=|v_{\ell_\epsilon}|+1}^{n-1}\epsilon2^m\le\ell_\epsilon+\epsilon2^n.\end{equation}
Collecting the counts \eqref{M-count} and \eqref{N-count}, 
keeping in mind that $\epsilon>0$ is arbitrary and $|\mathcal{H}_t|>0$, 
one obtains \eqref{lim-bounds} in a limit as $n\to\infty$.
\end{proof}

With the notation from \eqref{counting-vertices}, we have the following
characterization of the solution processes obtained via Picard iterations 
with a constant ground state.

\begin{lem}\label{hypersubtreelem}
Let $X_0=\delta$ and let $M_n(t)$ and $N_n(t)$ be as in 
\autoref{counting-vertices}. 
Then, the stochastic Picard iteration \eqref{X_n-eq} at generation $n$ is a.s. 
given by
\begin{equation}\label{expl-iter}
X_n(t) = u_0^{N_n(t)}\delta^{M_n(t)}.
\end{equation}
Moreover, for $u_0>0$, on the event $[S\le t]$ one has:
\begin{enumerate}
\item[(i)] If $\delta\in[0,1)$ then
$X_n(t)\to0$ almost surely as $n\to\infty$.
\item[(ii)] 
If $\delta>1$ then
$X_n(t)\to\infty$ almost surely as $n\to\infty$.
\end{enumerate}
\end{lem}
\begin{proof}
The explicit representation \eqref{expl-iter} follows directly from \eqref{X_n-eq} by induction on $n$. 

As noted in the proof of \autoref{counting-vertices}, $|\mathcal{H}_t|>0$ on the event $[S\le t]$ according to 
\autoref{hypersubprop}. Moreover,
the set $V(t)$ is infinite, and thus $V(t)$ has 
vertices of every generation, which implies that when $\delta=0$, 
$X_n(t)=0$ for all $n\in\mathbb{N}$, so the {\em (i)} holds for $\delta=0$. Now consider the case $\delta>0$. Then 
\[\frac{\ln X_n(t)}{2^n}=\frac{N_n(t)}{2^n}{\ln u_0}+\frac{M_n(t)}{2^n}
\ln \delta.\]
If $\delta\in(0,1)$, then \eqref{lim-bounds} implies that
\[\limsup_{n\to\infty}\frac{\ln X_n(t)}{2^n}\le 
-|\mathcal{H}_t|\ln\frac{1}{\delta}<0,\]
so $X_n(t)\to 0$ as $n\to\infty$.
If $\delta>1$, then \eqref{lim-bounds} implies that
\[\liminf_{n\to\infty}\frac{\ln X_n(t)}{2^n}\ge |\mathcal{H}_t|\ln\delta>0,\]
so $X_n(t)\to \infty$ as $n\to\infty$.

\end{proof}

The fine structure of the hyperexplosive event $[L\le t]$ leads to the following classification of limits of stochastic Picard iterations with constant ground state $X_0=\delta$ and initial data $u_0.$

\begin{prop}\label{alphgrndstateprop}
Let $\alpha>1$. Consider the stochastic Picard iterations $X_n(t)$ defined by \eqref{X_n-eq} 
with a constant 
ground state $X_0(t)\equiv \delta >0$. As in \autoref{ssprobintrosec}, denote $N(t)=|\partial V(t)|$. 
\begin{enumerate}
\item If $u_0=0$ then 
\begin{enumerate}
\item[(i)] For $\delta\in(0,1)$, 
$$X_n(t)=\lwrX(t)=0.$$
Consequently, the minimal solution $\underline{u}(t)={\EXP}\underline{X}(t)=0$ for all $t>0$.
\item[(ii)] For $\delta=1$, $$X_n(t)\to\overline{X}(t)=\begin{cases}
1 \ &\ \text{if}\ L\le t\\
0, \ &\ \text{if}\ L> t.\end{cases}$$
Consequently, the solution $\overline{u}(t)={\EXP}\overline{X}(t)=\P(L\le t)\to1$ as $t\to\infty$. 
\item[(iii)] For $\delta >1$, 
$$X_\infty(t) := \lim_{n\to\infty}X_n(t) = \begin{cases}
\infty \ &\ \text{if}\ L\le t\\
0 \ &\ \text{if}\ L> t.
\end{cases}$$
Consequently,  ${\EXP}X_\infty(t)=\infty$ for all $t>0$.
\end{enumerate}

\item If $u_0\in(0,1]$ then 
\begin{enumerate}
\item[(i)] For $\delta\in(0,1)$, 
$$X_n(t)\to\underline{X}(t)=\begin{cases}
0 \ &\ \text{if}\ S\le t\\
u_0^{N(t)}, \ &\ \text{if}\ S> t.\end{cases}$$
Consequently, the minimal solution $\underline{u}(t)={\EXP}\underline{X}(t)$ is 
well-defined and $\underline{u}(t)\to0$ as $t\to\infty$.
\item[(ii)] For $\delta=1$, $$X_n(t)\to\overline{X}(t)=\begin{cases}
1 \ &\ \text{if}\ L\le t\\
u_0^{N(t)}, \ &\ \text{if}\ L> t.\end{cases}$$
Consequently, the solution $\overline{u}(t)={\EXP}\overline{X}(t)$ is well-defined and $\overline{u}(t)\to1$ as $t\to\infty$.
\item[(iii)] For $\delta >1$,  
$$X_\infty(t) = \lim_{n\to\infty}X_n(t) = \begin{cases}
\infty \ &\ \text{if}\ S\le t\\
u_0^{N(t)}, \ &\ \text{if}\ S> t.\end{cases}$$
Consequently,  ${\EXP}X_\infty(t)=\infty$ for all $t>0$.
\end{enumerate}
\item If $u_0>1$ then
\begin{enumerate}
\item[(i)] For $\delta\in [0,1)$,
$$X_n(t)\to\overline{X}(t)=\begin{cases}
0 ,\ &\ \text{if}\ S\le t\\
u_0^{N(t)}, \ &\ \text{if}\ S> t.\end{cases}$$
Consequently, for all $t>0$
$$\underline{u}(t) ={\EXP}\underline{X}(t)
\ \begin{cases} <\infty, \ &\ \text{if}\ u_0<(2\alpha-1)/4,\\
                         =\infty, \ & \text{in finite time if}\  u_0>2\alpha-1,\\
                         \text{unknown} & \text{in other cases.}
                        \end{cases}
      $$
 If there is a locally integrable function $g$ such that $X_n\le g$ for all $n$, then
  $\underline{u}(t)={\EXP}\underline{X}(t)<\infty$.
\item[(ii)] For $\delta =1$, 
$$X_n(t)\to\overline{X}(t) =\begin{cases}
\infty & \text{on} \ [N(t)=\infty]\\
u_0^{N(t)} & \text{on} \ [N(t)<\infty].
\end{cases}$$
Consequently, if $\alpha\in(1,2)$, then $\EXP\overline{X}(t) = \infty$, and if $\alpha\ge 2$ then 
$\overline{X}(t)\in(0,\infty)$ a.s.\ for all $t>0$. Moreover,
\[
\overline{u}(t) ={\EXP}\overline{X}(t)
\ \begin{cases} <\infty, \ &\ \text{if}\ u_0<\Phi(\alpha),\ \text{where}\ \Phi\ \text{is given by \eqref{Phifunc}}\\
                         =\infty \ & \text{in finite time if}\  u_0>2\alpha-1\ \text{or}\ \alpha\in(1,2),\\
                         \text{unknown} & \text{in other cases}.
                        \end{cases}
\]
\item[(iii)] For $\delta>1$, 
$$X_n(t)\to{X}_{\infty}(t)=\begin{cases}
\infty \ &\ \text{if}\ S\le t\\
u_0^{N(t)} \ &\ \text{if}\  S> t.\end{cases}$$
Consequently, $u(t) ={\EXP}{X}_\infty(t) = \infty$ for all $t>0$.
\end{enumerate} 
\end{enumerate}
\end{prop}

\begin{proof} In all cases, the particular form of limits of $X_n$ follow from the explicit representation 
\eqref{expl-iter} and \autoref{hypersubtreelem}. The statements about the expectations are proven 
in \cite{alphariccati}, with additional input from \cite{dascaliuc2023errata} in the case $\alpha\in(1,2)$. 
Namely, the cases in Parts 1 and 2 follow from  \cite[Prop 4.1 and Thm 4.1]{alphariccati}.
The case 3{\it (i)} is proven in \cite[Thm 3.3 and 5.1]{alphariccati}. The Part 3{\it(ii)} is proven in 
\cite{alphariccati}*{Theorems 4.2 and 5.1} (noting that $\overline{u}(t)\ge\underline{u}(t)$); 
moreover, \cite[Prop 2]{dascaliuc2023errata}  is used to conclude infiniteness of the expectation in the case $\alpha\in(1,2)$. 
Finally, the case 3{\it(iii)} follows from the fact that $\P(S\le t)>0$ for all $t>0$, see \autoref{contdist}.
\end{proof}

Next, we will present several illustrations of the use of the processes constructed in \autoref{firstcase}--\,\ref{generalinitialdata} to use Monte Carlo techniques to simulate solutions of \eqref{genpantode} and \eqref{diffalpha-ricc0}.

\section{Numerical results based on Monte Carlo simulations}\label{simulation}

In this section, we illustrate the results of using Monte Carlo algorithm to simulate solution processes $\sX$ and $X$.
Our main interest to simulate nontrivial solutions to the pantograph equation
\begin{equation}\label{panto0}w'=-w+aw(\alpha t),\ \ \ w(0)=0\end{equation}
and non-constant solutions to the $\alpha$-Riccati equation
\begin{equation}\label{alpha1}u'+u=u^2(\alpha t),\ \ \ u(0)=1\end{equation}

\subsection{Pantograph equation}
It has been shown in \autoref{prop11112} that, for the case $\alpha>\max\{a,1\}$, \eqref{panto0} has a nontrivial solution $\eta(t)=\EXP\tsX_*(t)$, where $\tsX_*(t)=(t-\tilde{S})^{-\gamma}\IND{[\tilde{S}<t]}$ and $\tilde{S}=\sum_{i=0}^\infty\alpha^{-i}T_i$. It can be easily seen that the $n$'th derivative $\eta^{(n)}(t)$ satisfies the equation
\[w'=-w+a\alpha^n w(\alpha t),\ \ \ w(0)=0\]
and the $n$'th iterated integral $\int_0^t\int_0^{t_1}\ldots\int_0^{t_{n-1}}\eta(s)dsdt_{n-1}\ldots dt_1$ satisfies the equation
\[w'=-w+\frac{a}{\alpha^n} w(\alpha t),\ \ \ w(0)=0.\]
Thus, by means of differentiation and integration, one can obtain a nontrivial solution to \eqref{panto0} for any $a>0,\alpha>1$. It is sufficient to illustrate the simulation of $\eta$, $\eta'$, and $\int_0^t\eta(s)ds$. We will mention an analytic method and probabilistic method to do so. 

\textbf{Analytic method}: let $\tsX_n(t)=(t-\tS_n)^{-\gamma}\IND{[\tS_n<t]}$ where $\tS_n=\sum_{i=0}^{n-1}\alpha^{-i}T_i$. Then $\eta(t)=\lim\eta_n(t)$ where $\eta_n(t)=\EXP\tsX_n(t)$. The sequence $(\eta_n)$ satisfies a Picard iteration
\[\eta_0(t)=t^{-\gamma},\ \ \ \eta_n(t)=\int_0^t ae^{-s}\eta_{n-1}(\alpha(t-s))ds\]
One can use Mathematica to get an explicit formula for $\eta_1$, $\eta_2$, $\eta_3$,... However, these formula exhibit 
numerical artifacts for large values of $t$. They collapse to 0 instead of decaying as $t^{-\gamma}$ (\autoref{f1}).
\begin{figure}[h!] 
\begin{align*}
\includegraphics[scale=.65]{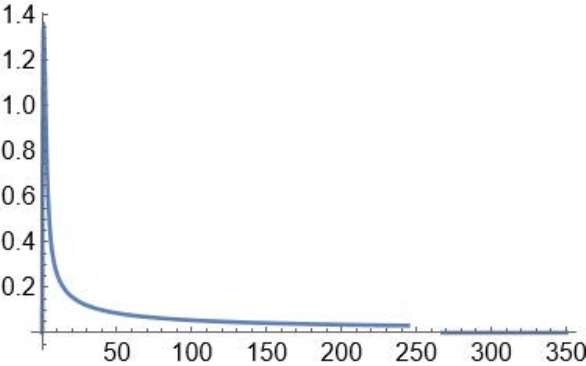}\ \ \ \ 
\includegraphics[scale=.65]{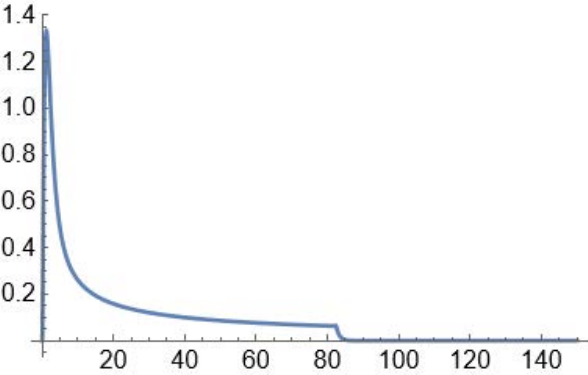}
\end{align*}
\caption{Graphs of $\eta_2$ (left) and of $\eta_3$ (right) with numerical artifacts ($\alpha=3,a=2$)}
\label{f1}
\end{figure}
To guarantee that $\eta$ decays as $t^{-\gamma}$ (\autoref{prop11112}), we will use the approximations $\eta_n\approx \tilde{\eta}_n$ where
\[\tilde{\eta}_0(t)=\eta_0(t)=t^{-\gamma},\ \ \ \tilde{\eta}_n(t)=\left\{ \begin{array}{*{35}{r}}
   \int_{0}^{t}{a{{e}^{-s}}{{\tilde\eta }_{n-1}}(\alpha (t-s))} & \text{if} & t<50  \\
   {{t}^{-\gamma }} & \text{if} & t>50  \\
\end{array} \right.\]
A comparison of $\eta_0$, $\eta_1$,..., $\eta_6$ where $\alpha=3,a=2$ is in \autoref{etacompare}. We observed that $\eta(t)\approx\tilde\eta_2(t)$ is already a good approximation for the simulation of the $\alpha$-Riccati in the next section.
\begin{figure}[h!] 
\centering
\includegraphics[scale=.7]{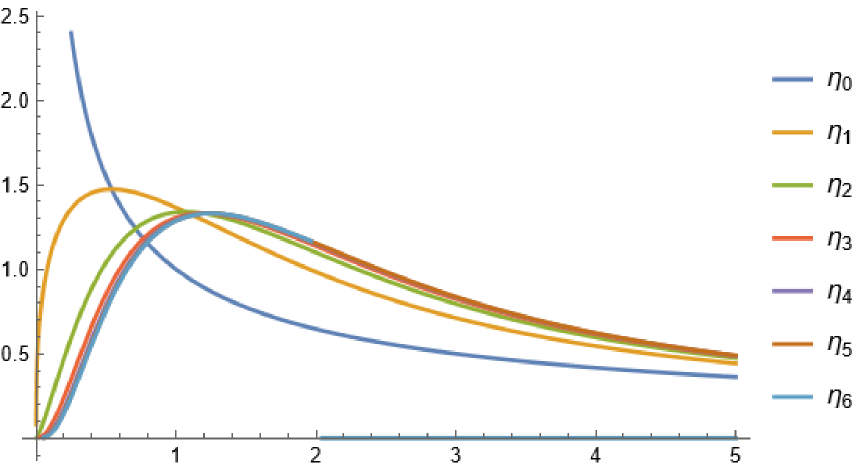}
\caption{Comparison of $\eta_0$,$\eta_1$,...,$\eta_6$ where $\alpha=3,a=2$}
\label{etacompare}
\end{figure}

\textbf{Probabilistic method}: we approximate the unary process $\tsX_*=(t-\tilde{S})^{-\gamma}\IND{\tilde{S}<t}$ by $\tsX_*\approx \tsX_k$.
With $\alpha=3$ and $a=2$, we simulate the solution on the time interval $[0,12]$ by Monte Carlo algorithm as follows. The time interval is divided into a lower 
range $[0,2.5]$, middle range $[2.5,5.7]$, and upper range $[5.7,12]$. Each range is then subdivided by evenly spaced grid-points in log-space. The number of grid-points in each range is 80, 10, 10, respectively. Independent vectors of exponential times $(T_0,T_1,...,T_{k})$ are drawn with sample size $N=10000,\,3000,\,2000$ for the lower, middle, and upper range respectively. At each time node $t$, we approximate 
\[\eta(t)\approx\EXP\tsX_{k}(t)\approx \frac{1}{N}\sum_{j=1}^N \tsX_{k,j}(t)\] 
where $\tsX_{k,j}$ is the realization of $\tsX_k$ in the $j$'th sampling. One can simulate the derivative $\eta'$ at the same time by noticing that
\[\eta'(t)=a\eta(\alpha t)-\eta(t)=\EXP[a\tsX_*(\alpha t)-\tsX_*(t)]\approx \frac{1}{N}\sum_{j=1}^N (a \tsX_{k,j}(\alpha t)-\tsX_{k,j}(t)).\]
The integral $w(t)=\int_0^t\eta(s)ds$ satisfies $w'=\frac{a}{\alpha}w(\alpha t)-w(t)$. From this identity, one can easily check that
\[w(t)=\sum_{i=1}^\infty\left(\frac{\alpha}{a}\right)^iw'(\alpha^{-i}t)=\sum_{i=1}^\infty\left(\frac{\alpha}{a}\right)^i\eta(\alpha^{-i}t).\]
Therefore, $w$ can be simulated at the same time of simulating $\eta$ by
\[w(t)=\EXP\left[\sum_{i=1}^\infty\left(\frac{\alpha}{a}\right)^i\tsX_*(\alpha^{-i}t)\right]\approx\frac{1}{N}\sum_{j=1}^N\sum_{i=1}^{\bar{k}}\left(\frac{\alpha}{a}\right)^i\tsX_{k,j}(\alpha^{-i}t)\]
where $\bar{k}$ is a truncation index. The numerical results are visualized in \autoref{etaf}.
\begin{figure}[htp] 
\centering
\includegraphics[scale=.55]{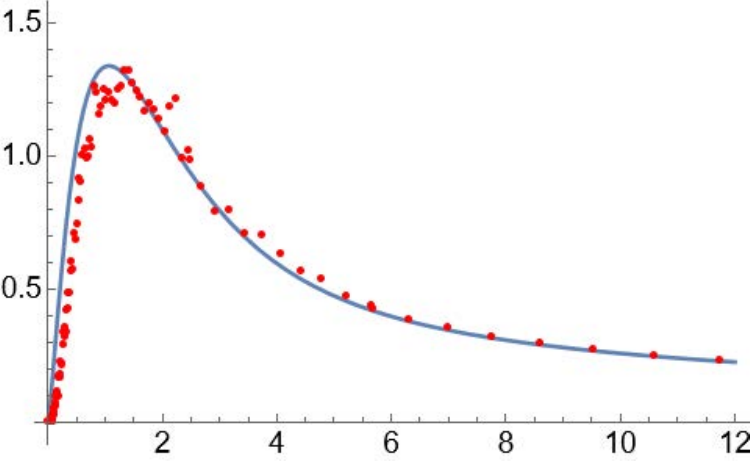}\ \ \,
\includegraphics[scale=.55]{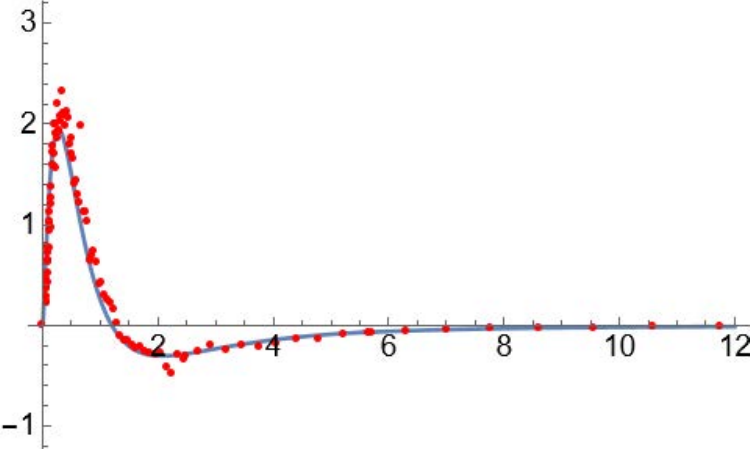}\vspace{1em}
\includegraphics[scale=.55]{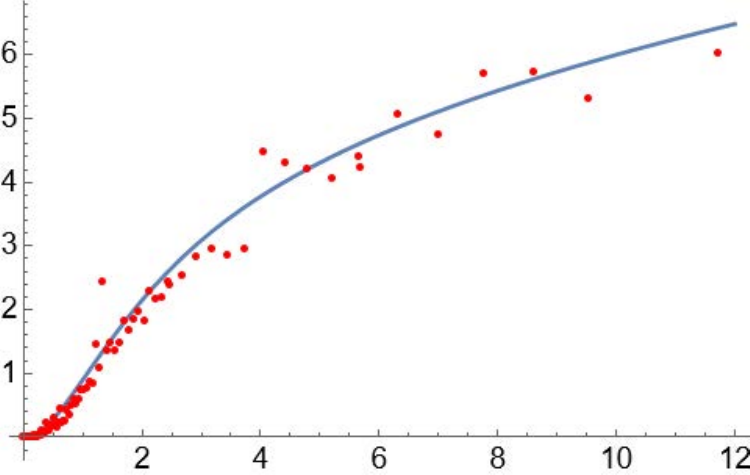}
\caption{$\alpha=3,a=2,K=15,\bar{K}=6$. Left figure: graph of $\tilde\eta_2$ and its Monte Carlo approximation. Right figure: graph of $\tilde\eta'_2\approx-\tilde\eta_2+a\tilde\eta_2(\alpha t)$ and its Monte Carlo approximation. Bottom figure: graph of $\int_0^t\tilde\eta_2(s)ds$ and its Monte Carlo approximation.}
\label{etaf}
\end{figure}

\subsection{$\alpha$-Riccati equation}
Our goal is to simulate solutions $u_\lambda=\EXP[\ZZ^\lambda]$ for various values of $\lambda>0$, where $U(t)$ is the stochastic process
given by \eqref{Zproc}. These solutions differ from each other by their convergence rate as $t\to\infty$: 
\begin{eqnarray}\label{rate1}
\lim_{t\to\infty}\frac{1-u_{\lambda}(t)}{\lambda t^{-\gamma}}=1,\ \text{where } \gamma=\log_\alpha 2\qquad  \text{(\autoref{main1}, \autoref{main2}).}
\end{eqnarray}
 We note that the nonuniqueness of the $\alpha$-Riccati equation was also found numerically via a different method in \cite{hale}. Before considering each case in detail, we would like to point out two different sampling methods. Fix $\lambda=1$ for illustration purpose.

The first method is to use the recursion formula of $\ZZ$ and sample the exponential times $T_v$ one at a time. For example, to compute $\ZZ\approx \ZZ_{n=10}$, we sample $T_0$. Then the computer program spawns two independent processes to compute two independent versions of $\ZZ_{n=9}$. One process samples $T_1$ and the other process samples $T_2$. Each process then spawns another two independent processes to compute a total of four different versions of $\ZZ_{n=8}$, and so on.

The second method is to sample the exponential times $T_v$ all at once. For example, to compute $X_{n}$, we sample a vector of $2^{n}-1$ exponential times and label them serially as $(\tau_1,\tau_2,\ldots,\tau_{2^n-1})$. This vector constitutes a realization of a binary tree of height $n$ where $T_\theta=\tau_1$, $T_1=\tau_2$, $T_2=\tau_3$, $T_{11}=\tau_4$, $T_{12}=\tau_5,$..., $T_{212}=\tau_{13}$, and so on. At generation $n$, the vertices from left to right are labeled by $2^{n-1}$, $2^{n-1}+1$,..., $2^n-1$. For $2^{n-1}\le j\le 2^n-1$, we calculate the survival times
\[\theta_j=\sum_{i=0}^{n-1}\frac{\tau_{[j/2^i]}}{\alpha^{n-1-i}}\]
where $[s]$ denotes the floor of a real number $s$. We then use the closed form of $\ZZ_n$ to approximate $\ZZ$:
\[\ZZ(t)\approx \ZZ_n(t)=\prod_{j=2^{n-1}}^{2^n-1}\ZZ_0^2(\alpha^n(t-\theta_j))\]
We divide the time interval into the lower 
range $[0,1]$, middle range $[1,11]$, and upper range $[11,111]$, and subdivide each range with 
equally spaced grid-points in log-space. The number of grid-points in each range is 10, 25, 25, respectively. At each node $t$, we approximate
\begin{equation}\label{ulapprox}u_\lambda(t)\approx \frac{1}{N}\sum_{k=1}^N \ZZ_{n,k}^\lambda(t)\end{equation}
where $N$ is the sample size and $\ZZ_{n,k}$ is the realization of $\ZZ_n$ in the $k$'th sampling. The variance of $\ZZ^\lambda$ is
\[\frac{\text{Var}(\ZZ^\lambda(t))}{t^{-\gamma}}=\frac{\EXP[\ZZ^{2\lambda}]-\EXP[\ZZ^\lambda]^2}{t^{-\gamma}}=\frac{u_{2\lambda}(t)-u^2_\lambda(t)}{t^{-\gamma}}=(1+u_\lambda(t))\frac{1-u_\lambda(t)}{t^{-\gamma}}-\frac{1-u_{2\lambda}(t)}{t^{-\gamma}}\]
Thus, Var$(\ZZ^\lambda(t))=o(t^{-\gamma})$ as $t\to\infty$. It can be expected that for upper time range, one obtains a good approximation for $u_\lambda(t)$ even with a relatively small sample size.

\textbf{In the case $1<\alpha\le 2$}, the ground state is given by \eqref{eq:72191}:
\begin{equation*}\ZZ_0(t)=(\rho_{M,\delta}(t))^{1/\delta}=\left\{ \begin{array}{*{35}{l}}
   1 & \text{if} & t\le M ,  \\
   (1-\delta{{t}^{-\gamma}})^{1/\delta} & \text{if} & t>M   \\
\end{array} \right.\end{equation*}
where $M>0$ is sufficiently large and $\delta>0$ is sufficiently small. 
For numerical simulation, we choose $\delta=4$, $M=10$, and sample size $N=200$. For each sample $\ZZ_{n=10,k}$ of 
$\ZZ_{10}$, we compute $\ZZ_{10,k}^\lambda$ with $\lambda=0.5,1,3$ and approximate 
$u_\lambda(t)$ according to \eqref{ulapprox}. The convergence rate \eqref{rate1} is illustrated by the 
log-log plot of $1-u_\lambda$ with the log-log plot of $\lambda t^{-\gamma}$ (a straight line). See \autoref{alpha1.4-fig}.

\begin{figure}[h!] 
\begin{align*}
\includegraphics[scale=.55]{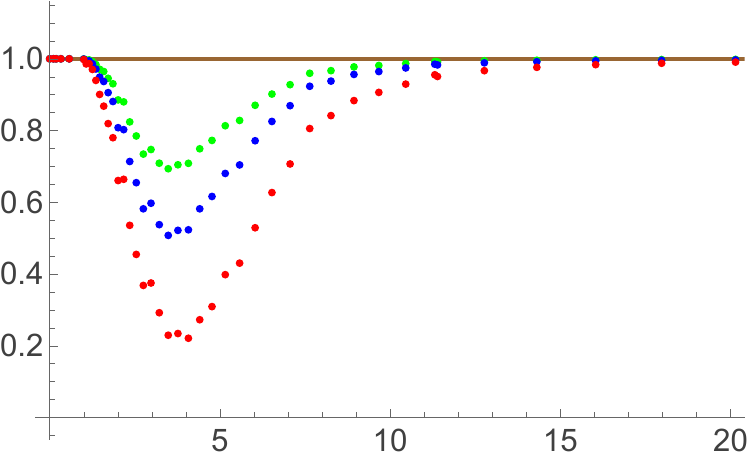}\ \ \,
\includegraphics[scale=.55]{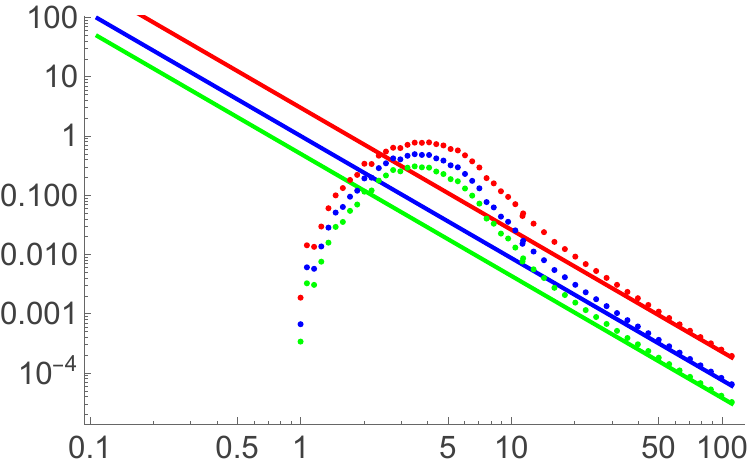}
\end{align*}
\caption{$\alpha=1.4$ and $\lambda=0.5$ (green), $\lambda=1$ (blue), $\lambda=3$ (red). Left figure: discrete graph of $u_\lambda$ via Monte Carlo. Right figure: log-log plot of $\lambda t^{-\gamma}$ and of $1-u_\lambda$.}
\label{alpha1.4-fig}
\end{figure}

\textbf{In the case $\alpha> 2$}, the ground state is 
\begin{equation*}\ZZ_0(t)=\left\{ \begin{array}{*{35}{l}}
   1 & \text{if} & t\le 0 ,  \\
   e^{-\eta(t)}\approx e^{-\tilde{\eta}_2(t)} & \text{if} & t>0   \\
\end{array} \right.\end{equation*}
 We discretize
 the time interval the same way as in the case $1<\alpha\le 2$. See \autoref{alpha3-fig}.

\begin{figure}[h!] 
\begin{align*}
\includegraphics[scale=.55]{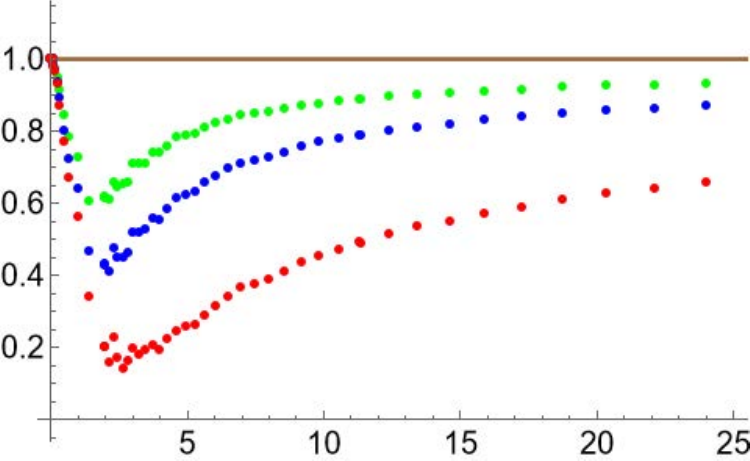}\ \ \,
\includegraphics[scale=.55]{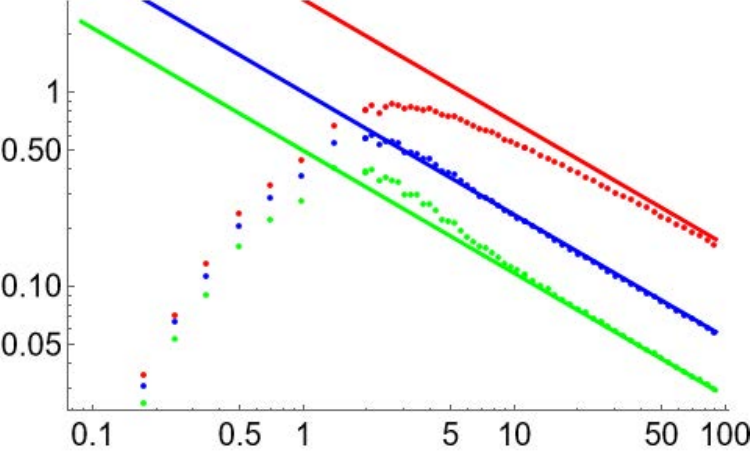}
\end{align*}
\caption{$\alpha=3$ and $\lambda=0.5$ (green), $\lambda=1$ (blue), $\lambda=3$ (red). Left figure: discrete graph of $u_\lambda$ via Monte Carlo. Right figure: log-log plot of $\lambda t^{-\gamma}$ and of $1-u_\lambda$.}
\label{alpha3-fig}
\end{figure}

\appendix

\section{Pantograph and Feller's Semigroup Theory}\label{fellerapp}

The space-time form \eqref{spacetimepant} of the pantograph equation can be viewed  as a 
Kolmogorov backward equation for a jump Markov process for $0<a\le 1$. First, 
let us assume $a=1$ and $\ww_0\equiv 1$ in \eqref{mildspacetimepant} and \eqref{mildpant} to start. On the state space $[0,\infty)$, with $\ell=0$ being an absorbing state, the corresponding Markov process holds in state $\ell> 0$ for an exponential 
time of intensity $\lambda(\ell) = \ell$ before transitioning to state 
$\alpha\ell$, i.e.\ with a transition probability kernel $k(\ell_1,d\ell_2) = \delta_{\{\alpha\ell_1\}}(d\ell_2)$.
The constant solution $\ww(\ell,t)\equiv 1$  is a solution (for $\ww_0\equiv 1$) 
and is unique if and only if there is no explosion (in the sense of Feller). The jump Markov process starting at 
$\ell>0$ is explosive if and only if $\alpha >1$ since the mean time between the 
$n$-th and $(n+1)$-st jumps is $\frac{1}{\ell\alpha^{n}}$ for $n\ge0$. In fact, for $\alpha>1$, the {\em explosion time}
 \[\tS_\ell=\sum_{j=0}^\infty \frac{T_j}{\ell\alpha^j}\] 
 is finite a.s. Here, $\{T_j\}_{j\ge 0}$ is an i.i.d.\ sequence of mean-one
 exponential random variables. 
 The minimal jump process may be continued beyond explosion 
 time by absorbing the process in an adjoined spatial point at infinity.  It is standard 
 Markov process theory \cites{breiman1992probability,bhattacharya2023} that this does not alter the infinitesimal behavior 
 of (\ref{spacetimepant}). Let $p(1; \ell_1,d\ell_2)$ be the 
 substochastic transition probabilities  of the minimal process. Then $w(s)=p(1;s,[0,\infty))$ satisfies \eqref{mildpant} with $a=1$ and $w_0\equiv 1$.
This solution  is
 called {\em minimal solution} to \eqref{mildpant}, which differs from the 
 steady state $w\equiv 1$ in the case $\alpha>1$. Denote by $\tilde{N}_\ell(t)$ the number of clock rings by time $t$ with initial state $\ell$. By the convolution formula for non-identical 
 exponential distributions \cite[p.\ 40]{Feller2},
  \begin{eqnarray}\label{distinctconv}
 \P(\tilde{N}_\ell(t) = n) &=&
 \P\left(\sum_{j=0}^{n}
 \frac{1}{\ell\alpha^j}T_j>t\right)-
 \P\left(\sum_{j=0}^{n-1}
 \frac{1}{\ell\alpha^j}T_j>t\right)\nonumber\\
&=& \sum_{j=0}^n
\frac{e^{-\alpha^j\ell t}}{\prod_{k\in I_{n,j}}
(1-\alpha^{j-k})} -  \sum_{j=0}^{n-1}
\frac{e^{-\alpha^j\ell t}}{
\prod_{k\in I_{n-1,j}}(1-\alpha^{j-k})}\nonumber\\
&=& \sum_{j=0}^n
\frac{e^{-\alpha^j\ell t}\alpha^{j-n}}{\prod_{k\in I_{n,j}}(1-\alpha^{j-k})}
 \end{eqnarray}
where $I_{n,j}=\{0,1,...,n\}\backslash\{j\}$. Hence, the minimal solution has 
 an explicit series representation:
 \begin{equation}\label{pantographformula}
 w(s)= \P_{\ell=s}(\tS_s>1) = \sum_{n=0}^\infty \P(\tilde{N}_{s}(1)=n) =\sum_{n=0}^\infty \sum_{j=0}^n
 \frac{e^{-\alpha^j s}\alpha^{j-n}}{\prod_{k\in I_{n,j}}(1-\alpha^{j-k})}.
\end{equation}
In the case where one can exchange the order of summation, (\ref{pantographformula}) is equivalent to the series representation obtained in \cite[Eq.\ (5.2)]{kato1971functional} by an analytic method. 

\begin{remark} In the case $\alpha\in(0,1)$, the uniqueness of solutions 
to the Kolmogorov backward equation implies that the series (\ref{pantographformula}) is identically one 
for all $t\ge 0$. In the case $\alpha=1$, the minimal solution $w(t)$ is also identically one because $\tilde{N}_s(1)$ 
has a Poisson distribution.  However, if $\alpha>1$ then $w(t)<1$ for all $t>0$. 
\end{remark}

Now let us consider the case $0<a < 1$. One can adapt the above approach by considering a jump Markov process on the compactified 
half-line $[0,\infty]$ with $\lambda(\ell) = \ell\,\IND{\ell<\infty}$ (with the convention that $\infty\cdot 0=0$) and transition probabilities $k(\ell_1,d\ell_2)= a\delta_{\{\alpha \ell_1\}}(d\ell_2)
+(1-a)\delta_{\{\infty\}}(d\ell_2)$ for $\ell_1\in[0,\infty]$ (thus, $\infty$ is absorbing). By the same arguments 
as above with initial data $\ww(\ell,0)=1$ for $\ell\in(0,\infty)$, $\ww(\infty,0)=0$, one obtains a solution to \eqref{mildpant} that is vanishing at infinity. 
Note that the constant $w\equiv 1$ is no longer a solution to the pantograph equation 
since $a\neq 1$. Also, for $0<a<1$, there are only finitely many finite jumps before the Markov process is absorbed at infinity, regardless of the value of $\alpha>0$. Therefore, stochastic explosion has probability 
 zero and the unique solution, in the
 space of bounded measurable functions, to \eqref{genpantode} with $w_0=1$ is given by
 \begin{equation}\label{pant_sol_small_a}
 w(s) = p(1;s, [0,\infty)) =w_0\sum_{n=0}^\infty a^n\P(\tilde{N}_s(1)=n)= \sum_{n=0}^\infty a^n\sum_{j=0}^n
 \frac{e^{-\alpha^j s}\alpha^{j-n}}{\prod_{k\in I_{n,j}}(1-\alpha^{j-k})}. 
 \end{equation}
Note that in the case $\alpha>1$, \autoref{prop11112} together with the linearity of the pantograph equation implies that $w+c\eta$ is also a solution to \eqref{genpantode} for any $c\in\R$. However, such a solution is unbounded on $[0,\infty)$ unless $c=0$. For $w_0=0$, the minimal solution is identically zero because $\lambda(0)=\lambda(\infty)=0$.

The condition $0<a\le 1$ is clearly an obstruction to the construction of solutions to \eqref{genpantode} mentioned above. However, in view of the Hille-Yosida-Ray theorem \cite[p.\ 165]{ethierkurtz}, \cite[Thm 7.2.4]{Applebaum}, some
restrictions are intrinsic to Feller's semigroup theory underlying this approach. In particular, for any $\alpha>0$, the linear operator $A$ on $\mathcal{D}(A)\subset C_0(S)$, $S=[0,\infty)$, defined by $Af(\ell)=-\ell f(\ell)+a\ell f(\alpha\ell)$ satisfies the positive maximum principle if and only if $0<a\le 1$. Thus, $A$ can extend to an infinitesimal generator of Feller semigroup if and only if $0<a\le 1$.

\section{Rates of Convergence to 1 of $\alpha$-Riccati Solutions}\label{appendix1}

The following analytical result is used in \autoref{expl_asymptotics} to establish asymptotic properties of the distribution of the explosion time $S$ and hyperexplosion time $L$. In its statement, the following standard comparison notation is used: $f(t)\gtrsim g(t)$ means that there exist $c, T>0$ such that $g(t)\ge c g(t)$ for all 
$t\ge T$ and $f\sim g$  means $f\gtrsim g$ and $g\gtrsim f$.

\begin{thm}\label{alphasharprmk}
Let $\alpha>1$ and $u_0\in\R$.  If $u(t)$ is a solution to \eqref{diffalpha-ricc0} 
such that $u(t)\to 1$ as $t\to\infty$, then only one of the following scenarios is possible:
\begin{enumerate}
\item[(a)] \ $\limsup_{t\to\infty}\frac{|u(t)-1|}{t^{-\gamma}}>0$, where $\gamma>0$ is given by \eqref{gma},
\item[(b)]\ $|u(t)-1|\sim e^{-t}$,
\item[(c)]\ $u(t) = 1$ for  all $t\ge 0$.
\end{enumerate}
\end{thm}

\begin{remark}
The case (a) is illustrated by a special solution of \cite{athreya}. Case (b) is illustrated by 
$u(t)=\P(L> t)$, and case (c) by $u(t)\equiv 1$.
\end{remark}

The proof of  \autoref{alphasharprmk} rests on the following lemma which couples the initial 
data $0$ and $1$ through the inclusion-exclusion principle.  Namely, by inclusion-exclusion, 
$\w(t) = \P(L> t)$ solves \eqref{alphinclexcl} below with $v(0)=1$.

\begin{lem}\label{sharplem} 
Assume $\alpha >1$ in the $\alpha$-Riccati equation. Suppose that $v(t)$, $t>0$ solves
\begin{equation}\label{alphinclexcl}
\w^\prime(t) = -\w(t) + 2\w(\alpha t)-\w^2(\alpha t),
\end{equation}
and assume that $|\w(t)|\lesssim e^{-\rho t}$ as $t\to\infty$ for some $\rho>0$.
Then:
\begin{enumerate}
\item[(i)] \ either $\w\equiv 0$,
\item[(ii)]\ or $|\w(t)| \sim e^{-t}$ as $t\to\infty$.
\end{enumerate}
\end{lem}
\begin{proof}

First we will show that $|\w(t)|\lesssim e^{-t}$, which is automatically 
the case if $\rho\le 1$. In the case $\rho\in(0,1)$, we begin by noting that 
$|2\w(t)-\w^2(t)|\le ce^{-\rho t}$ 
for large enough $c=c_\rho$. Using (\ref{alphinclexcl}), one has
$$ -\w(t) - c\,e^{-\alpha\rho t}\le \w^\prime(t)\le -\w(t) +c\,e^{-\alpha\rho t}.$$
In the case $\alpha\rho\ne 1$, integrating on $[t_0,t]$ one has
\begin{eqnarray*}
|\w(t) - \w(t_0)e^{-(t-t_0)}| &\le& e^{-t}\int_{t_0}^te^sc\,e^{-\alpha\rho s}ds=\frac{c}{1-\alpha\rho}e^{-t}\left(e^{(1-\rho\alpha)t}
-e^{(1-\rho\alpha)t_0}\right).
\end{eqnarray*}
Thus, as $t\rightarrow \infty$,
$$
|\w(t)| \lesssim \begin{cases} e^{-t} \ & \ \text{if} \ \alpha\rho >1 \\
                             e^{-\alpha\rho t} \ &\  \text{if}\ \alpha\rho<1.
                             \end{cases}
$$
If $\alpha\rho=1$ we have $|\w(t)|\lesssim t e^{-t}$, so 
$|\w(t)|\lesssim  e^{-\tilde{\rho}t}$ for any $\tilde{\rho}<1$,
and we can repeat the bootstrap argument above with $\rho$ 
replaced by $\tilde{\rho}=1/\sqrt{\alpha}$. Thus, without loss of generality, 
we can assume $\alpha\rho\ne 1$.

In the case $\alpha\rho >1$, the process since $|\w(t)|\lesssim e^{-t}$ 
is established, while in the case $\alpha\rho <1$, the bootstrap process 
is repeated with $\rho$ replaced by $\alpha\rho (>\rho)$.
Note that each time the bootstrap process is applied, another factor
of $\alpha$ appears in the exponent.  Thus, after $k$ steps with 
$\alpha^k\rho<1$,
$$\w(t) \lesssim e^{-\alpha^k\rho t}.$$
Now, for $k\ge \frac{\ln(1/\rho)}{\ln(\alpha)}$ the process stops
and $\w(t) \lesssim e^{-t}$ is achieved. In other words, there exist 
$C_0, t_0 > 0$ such that $|v(t)| \le \min\{C_0 e^{-t}, 1\}$ for all 
$t \ge t_0$. 

It remains to show that if $\w\not\equiv0$, then $|\w|\gtrsim e^{-t}$. Indeed, assume $\w\not\equiv0$. From \eqref{alphinclexcl}, we have 
\[e^t v(t) = e^s v(s) + \int_s^t e^\tau \left(2v(\alpha \tau) - v^2(\alpha \tau)\right) d\tau, \quad \forall t \ge s.\]
Take the absolute value of both sides:
\begin{align*}
e^t |v(t)| &\ge e^s |v(s)| - \left| \int_s^t e^\tau \left(2v(\alpha \tau) - v^2(\alpha \tau)\right) d\tau \right| \\
&\ge e^s |v(s)| - \int_s^t e^\tau \left|2v(\alpha \tau) - v^2(\alpha \tau)\right| d\tau \\
&\ge e^s |v(s)| - \int_s^\infty e^\tau \left|2v(\alpha \tau) - v^2(\alpha \tau)\right| d\tau.
\end{align*}
Using the estimate $|2v(\alpha \tau) - v^2(\alpha \tau)| \le 2|v(\alpha \tau)| + v^2(\alpha \tau) \le 3|v(\alpha \tau)|$ for all $\tau \ge t_0$, we have
\begin{equation}\label{onestar}
e^t |v(t)| \ge e^s |v(s)| - 3 \int_s^\infty e^\tau |v(\alpha \tau)| d\tau, \quad \forall t \ge s \ge t_0.
\end{equation}
If there exists $s\ge t_0$ such that RHS\eqref{onestar} is positive then $|v(t)|\gtrsim e^{-t}$. 
We show that this is indeed the case. Suppose by contradiction that 
\begin{equation}\label{twostar}
e^s |v(s)| \le 3 \int_s^\infty e^\tau |v(\alpha \tau)| d\tau, \quad \forall s \ge t_0.
\end{equation}
Define a sequence $(C_n)_{n \ge 0}$ recursively by $C_{n+1} = \frac{3C_n}{\alpha^{n+1} - 1}$. We show by induction that for all $n \ge 0$,
\begin{equation}\label{threestar}
|v(s)| \le C_n e^{-\alpha^n s}, \quad \forall s \ge t_0.
\end{equation}
By the definition of $C_0$ and $t_0$, \eqref{threestar} is satisfied for $n = 0$. Suppose \eqref{threestar} is true for some $n \ge 0$. Applying \eqref{threestar} to the RHS of \eqref{twostar}, we get
\[e^s |v(s)| \le 3 \int_s^\infty e^\tau C_n e^{-\alpha^{n+1}\tau} d\tau = \frac{3C_n}{\alpha^{n+1} - 1} e^{(1 - \alpha^{n+1})s} = C_{n+1} e^{(1 - \alpha^{n+1})s}.\]
Thus,
\[|v(s)| \le C_{n+1} e^{-\alpha^{n+1}s}, \quad \forall s \ge t_0.\]
We have proven \eqref{threestar} where $n$ is replaced by $n+1$. Thus, \eqref{threestar} is true for all $n \ge 0$. Note that $\lim_{n \to \infty} C_n = 0$ because
\[\frac{C_{n+1}}{C_n} = \frac{3}{\alpha^{n+1} - 1} \le \frac{3}{\alpha^{N+1} - 1} < 1\]
for all $n \ge N = \log_\alpha 4$. Hence, for all $s \ge t_0$,
\[
|v(s)| \le C_n e^{-\alpha^n s} \le C_n \to 0 \quad \text{as } n \to \infty.
\]
This implies that $v(s) = 0$ for all $s \ge t_0$. By (\ref{alphinclexcl}), 
one then infers that $v(s) = 0$ for all $s \in [t_0/\alpha, t_0]$. 
Applying (\ref{alphinclexcl}) again, one gets $v(s) = 0$ for all 
$s \in [t_0/\alpha^2, t_0/\alpha]$ and so on. Therefore, $v \equiv 0$, 
which is a contradiction.
 
\end{proof}

\begin{proof}[Proof of \autoref{alphasharprmk}]
Let $\w(t)=1-u(t)$, then $\w$ satisfies \eqref{alphinclexcl}. Assume $|\w(t)|=o(t^{-\gamma})$. 
Note that \eqref{alphinclexcl} implies that
\[
\w'(t)= -\w(t)+(2-\w(\alpha t))\w(\alpha t),
\]
and thus
\begin{equation}\label{thm3.5-1}
\w(t)=\w(t_0)\,e^{-(t-t_0)}+\int\limits_{t_0}^t (2-\w(\alpha s))\, e^{-(t-s)} \w(\alpha s)\, ds.
\end{equation}
As in the proof of Theorem 9(ii) in \cite{kato1971functional}, it then follows that 
$|\w(t)|=O(e^{-\epsilon t})$ for some $\epsilon>0$. 
For completeness, we will present this argument below.

For $\tau>0$ consider 
\[
m(\tau)=\sup\limits_{t\ge\tau}\,t^{\gamma}|\w(t)|.
\]
Note that $m(\tau)$ is a bounded decreasing to zero function. Moreover, from \eqref{thm3.5-1}
\[
|\w(t)|\le t_0^{-\gamma}e^{-(t-t_0)}m(t_0)+m(\alpha t_0) e^{-t} \int\limits_{t_0}^t (2+|\w(\alpha s)|)\, 
e^{s} (\alpha s)^{-\gamma} ds.
\]
Let $b\in(0,1)$ be such that $\nu=\alpha^{1/2}-b>0$. 
Fix a $\tau_0>1$ big enough that $|w(t)|\le t^{-\gamma}$ and $t^{\gamma}e^{-t}\le e^{-bt}$ 
for all $t>\tau_0$. Then, for $t>t_0>\tau_0$ we have
\[
t^{\gamma}|\w(t)|= t_0^{-\gamma}e^{-bt+t_0}+ m(\alpha t_0) \, t^{\gamma}e^{-t} 
\int\limits_{t_0}^t  e^{s} \,(2+(\alpha s)^{-\gamma}) (\alpha s)^{-\gamma} ds.
\]
Using the estimate $\int_1^t t^k e^s ds \le (1+ \frac{c_k}{t})\,e^t t^k$ (valid for any $k\in\mathbb{R}$ 
and big enough $c_k>0$), we conclude
\[
\begin{aligned}
t^{\gamma}|\w(t)| &= t_0^{-\gamma}e^{-bt+t_0}+ m(\alpha t_0)\left(2\alpha^{-\gamma}\left(1+\frac{c_\gamma}{t}\right)
+\alpha^{-2\gamma}\left(1+\frac{c_{2\gamma}}{t}\right)t^{-\gamma}\right)\\
&\le  t_0^{-\gamma}e^{-bt+t_0}+ m(\alpha t_0)\left(1+\frac{C}{t^{\delta}}\right) \,.
\end{aligned}
\]
where $\delta=\min\{1,\gamma\}$, and $C>c_{\gamma}$ large enough, independent of $t$. In the above, 
we used the fact that $2\alpha^{-\gamma}=1$. Consider $\tau=\alpha^{1/2}t_0$ and take 
$\sup_{t\ge\tau}$ in the right-hand side of the inequality above. We obtain:
\begin{equation}\label{thm3.5-2}
m(\tau)\le t_0^{-\gamma}e^{-b\tau +\alpha^{-1/2}\tau}+m(\alpha^{1/2} \tau)\left(1+\frac{C}{\tau^{\delta}}\right)
= e^{-\nu\tau}+\left(1+\frac{C}{\tau^{\delta}}\right)\,m(\alpha^{1/2} \tau)
\,.
\end{equation}

We can iterate \eqref{thm3.5-2}, by applying it to $m(\alpha^{1/2}\tau)$ in the right-hand-side 
(with $\tau$ replaced with $\alpha^{1/2}\tau$), obtaining:

\[
\begin{aligned}
m(\tau)&\le e^{-\epsilon\tau}+
\left(1+\frac{C}{\tau^{\delta}}\right)\,
\left( e^{-\nu\alpha^{1/2}\tau}+\left(1+\frac{C}{(\alpha^{1/2}\tau)^{\delta}}\right)\,m(\alpha^{2/2} \tau)\right)\\
&\le 
e^{-\nu\tau}+\left(1+\frac{C}{\tau^{\delta}}\right) e^{-\nu\alpha^{1/2}\tau}
+ \left(1+\frac{C}{\tau^{\delta}}\right) \left(1+\frac{C}{\alpha^{\delta/2}\tau^{\delta}}\right)\,m(\alpha^{2/2} \tau)\\
&\le
\left(1+\frac{C}{\tau^{\delta}}\right) \left(1+\frac{C}{\alpha^{\delta/2}\tau^{\delta}}\right)
\left(
e^{-\nu\tau}+e^{-\nu\alpha^{1/2}\tau}+m(\alpha^{2/2} \tau)
\right)\\
&\le 
e^{\left(1+\frac{1}{\alpha^{\delta/2}}\right)\frac{C}{\tau^{\delta}}}
\left(
e^{-\nu\tau}+e^{-\nu\alpha^{1/2}\tau}+m(\alpha^{2/2} \tau)
\right)
\,.
\end{aligned}
\]
Iterating this process $n$ times, we estimate:
\[
\begin{aligned}
m(\tau)
&\le 
e^{\sum_{j=0}^{n-1}\frac{1}{\alpha^{j\delta/2}}\frac{C}{\tau}} \left(
\sum\limits_{j=0}^{n-1}e^{-\nu \alpha^{j/2}\tau}+m(\alpha^{n/2}\tau)
\right)\\
&\le
e^{\frac{\alpha^{\delta/2}}{\alpha^{\delta/2}-1}\frac{C}{\tau}} \left(
\sum\limits_{j=0}^{n-1}e^{-\nu \alpha^{j/2}\tau}+m(\alpha^{n/2}\tau)
\right)\,.
\end{aligned}
\] 
Now take $n\to\infty$, using that $m(t)\to 0$ as $t\to\infty$ to obtain
\[
m(\tau)\le e^{\frac{\alpha^{\delta/2}}{\alpha^{\delta/2}-1}\frac{C}{\tau}}\sum\limits_{j=0}^{\infty}e^{-\nu \alpha^{j/2}\tau}\,,
\]
and consequently, $m(\tau)=O(e^{-\nu \tau})$. Thus, $\w(t)=o(e^{-\epsilon t})$ with $\epsilon=\nu/2$.
 Then, the conclusions of \autoref{alphasharprmk} follow from  \autoref{sharplem}.
\end{proof}
\section*{Acknowledgment}
The authors would like to thank the anonymous referee for valuable suggestions which improved the exposition and clarified the arguments.\\\\
\noindent\textbf{Statement on AI use.}
In the preparation of the revised version of the manuscript, AI (Chat GPT 5.6 and Opus 5.5 with Claude Pro) was used in its search engine capacity for references and to look for relevant results in the literature. AI was also used to check for correctness before submission. All the mathematical arguments and proofs are constructed by the authors who take full responsibility for the accuracy of the final work. 
\bibliography{NSF2022-3}
\end{document}